\documentclass[11pt]{article}
 \usepackage{comment}
\usepackage{appendix}
\usepackage[T1]{fontenc}
\usepackage[latin9]{inputenc}
\usepackage{babel}
\usepackage{amsmath}
\usepackage{amsthm}
\usepackage{bm}
\usepackage{cite}
\usepackage{mathalpha}
\usepackage[unicode=true,pdfusetitle,
 bookmarks=true,bookmarksnumbered=false,bookmarksopen=false,
 breaklinks=false,pdfborder={0 0 1},backref=false,colorlinks=false]
 {hyperref}
 \usepackage{graphicx}

\numberwithin{equation}{section}
\numberwithin{figure}{section}
\theoremstyle{plain}
\newtheorem{thm}{\protect\theoremname}[section]
\theoremstyle{remark}
\newtheorem{claim}[thm]{\protect\claimname}
\theoremstyle{plain}
\newtheorem{lem}[thm]{\protect\lemmaname}
\newtheorem{cor}[thm]{Corollary}
\newtheorem{rem}[thm]{Remark}
\newtheorem{assumption}{Assumption}

\newtheorem{prop}[thm]{Proposition}

\theoremstyle{definition}
\newtheorem{defn}[thm]{Definition}

\usepackage{xcolor}
\usepackage{mathrsfs,amsthm,amssymb,bbm,fullpage}

\renewcommand{\limsup}{\varlimsup}

\newcommand{\cC}{\mathcal{C}}

\newcommand{\cM}{\mathscr{M}}
\newcommand{\cN}{\mathcal{N}}

\newcommand{\cP}{\mathcal{P}}

\newcommand{\cZ}{\mathcal{Z}}
\newcommand{\C}{\mathbb{C}}
\newcommand{\R}{\mathbb{R}}

\newcommand{\prob}{\mathbb{P}}
\newcommand{\E}{\mathbb{E}}

\newcommand{\eps}{\epsilon}

\newcommand{\eqdist}{\stackrel{(d)}{=}}
\newcommand{\tensor}{\otimes}

\newcommand{\MP}{{\textsc{mp}}}
\newcommand{\NN}{{\textsc{nn}}}
\newcommand{\REM}{{\textsc{rem}}}
\newcommand{\Eekt}{E_{\eps,K,\tau}}

\newcommand{\abs}[1]{\lvert#1\rvert}
\newcommand{\norm}[1]{\lvert\lvert#1\rvert\rvert}

\newcommand{\supp}{\operatorname{supp}}

\newcommand{\diag}{\operatorname{diag}}

\newcommand{\barq}{\bar{q}}

\newcommand{\wt}{\widetilde}

\newcommand{\bfG}{{\mathbf G}}
\newcommand{\bfO}{{\mathbf O}}

\newcommand{\Dll}{D_{\ell\ell}}
\newcommand{\bfx}{\mathbf x}
\newcommand{\bfGd}{\bfG^{(d)}}

\newcommand{\tSd}{\widetilde{S}^{(d)}}
\newcommand{\tS}{\widetilde{S}}

\newcommand{\tFd}{\widetilde F^{(d)}}
\newcommand{\dist}{\operatorname{dist}}
\newcommand{\mfrp}{\mathfrak{p}}
\newcommand{\mfrq}{\mathfrak{q}}

\newcommand{\bmu}{\boldsymbol{\mu}}

\renewcommand{\Im}{\operatorname{Im}}
\renewcommand{\Re}{\operatorname{Re}}
\providecommand{\claimname}{Claim}
\providecommand{\lemmaname}{Lemma}
\providecommand{\theoremname}{Theorem}

\title{Local geometry of high-dimensional mixture models: \\ Effective spectral theory and dynamical transitions}

\author{Gerard Ben Arous\thanks{Courant Institute, New York University. Email: \url{gba1@nyu.edu}}, Reza Gheissari\thanks{Department of Mathematics, Northwestern University. Email: \url{gheissari@northwestern.edu}}, Jiaoyang Huang\thanks{Department of Statistics and Data Science, University of Pennsylvania. Email: \url{huangjy@wharton.upenn.edu}}, Aukosh Jagannath\thanks{Department of Statistics and Actuarial Science, Department of Applied Mathematics, and Cheriton School of Computer Science, University of Waterloo. Email: \url{a.jagannath@uwaterloo.ca}}}

\begin{document}
 \date{}
\maketitle

\begin{abstract}
We study the local geometry of empirical risks in high dimensions via the spectral theory of their Hessian and information matrices.
We focus on settings where the data, $(Y_\ell)_{\ell =1}^n\in \R^d$, are i.i.d. draws of  a $k$-Gaussian mixture model, and the loss depends on the projection of the data into a fixed number of vectors, namely $\bfx^\top Y$, where $\bfx\in \R^{d\times \cC}$ are the parameters, and $\cC$ need not equal $k$. This setting captures a broad class of problems such as classification by one and two-layer networks and regression on multi-index models. We provide exact formulas for the limits of the empirical spectral distribution and outlier eigenvalues and eigenvectors of such matrices in the proportional asymptotics limit, where the number of samples and dimension $n,d\to\infty$ and $n/d=\phi \in (0,\infty)$. These limits depend on the parameters $\bfx$ only through the summary statistic of the $(\cC+k)\times (\cC+k)$  Gram matrix of the parameters and class means, $\bfG = (\bfx,\bmu)^\top(\bfx,\bmu)$.    

It is known that under general conditions, when $\bfx$ is trained by online stochastic gradient descent, the evolution of these same summary statistics along training converges to the solution of an autonomous system of ODEs, called the \emph{effective dynamics}. This enables us to connect the training dynamics to the spectral theory of these matrices generated with test data. We demonstrate our general results by analyzing the \emph{effective spectrum} along the effective dynamics
in the case of multi-class logistic regression. In this setting, the empirical Hessian and information matrices have substantially different spectra, each with their own static and even dynamical spectral transitions. 
\end{abstract}

\section{Introduction}
\subsection{Background}
 What is the geometry of the loss landscape in statistical tasks?  In particular, what does the geometry look like locally from the point of view of the learning algorithm? There is a rich and classical literature on these questions in the fixed dimensional, or large sample size, regime (see, e.g.,~\cite{amari2016information, marriott2002local,nielsen2022geometry}). 
    The modern, high-dimensional setting where the sample size and the dimension of the data and parameter space are proportionately large, is comparatively less understood. 
    In this setting, despite the fact that the loss landscape is often high-dimensional and non-convex, simple, first-order methods such as stochastic gradient descent (SGD) seem to perform remarkably well.

Towards an explanation for this, a rich line of research has emerged focusing on the geometry of the loss landscape (curvature/flatness and directions of steepest descent), in particular as seen from the trajectory taken by the training algorithm.  Much of this work has focused on tracking two natural random matrices: the Hessian of the empirical risk and the empirical second moment matrix for the gradient of the loss (i.e., the empirical Fisher Information matrix when the loss is the log-likelihood). Studying these objects for this reason was proposed going back more than two decades to the article of LeCun et al~\cite{LeCun-EfficientBackprop}. 
Around a decade ago, a belief emerged in the machine learning literature that the spectra of these random matrices, evaluated along the SGD trajectory, describes a local geometry that has many ``flat'' directions in which the SGD diffuses, and a hidden low-dimensional structure spanned by outlier eigenvectors, within which the bulk of training occurs. 
More precisely, the work of Sagun--Bottou--LeCun ~\cite{Sagun-Singularity-and-beyond} proposed that the spectrum of the Hessian of the empirical risk after training has the following structure:
\begin{enumerate}
    \item It has a \emph{bulk} that is dependent on the network architecture, and is concentrated around $0$, becoming more-so as the model becomes more overparametrized;
    \item It has a few \emph{outlier eigenvalues} that are dependent on the data, and evolve non-trivially along training while remaining separated from the bulk. 
\end{enumerate}
Following those works, this paradigm has been demonstrated in subsequently more refined and large-scale experimentation, including ones that showed  a similar picture for the Fisher information matrix: see e.g.,~\cite{SagunEtAl,Ghorbani-Hessian-eigenvalue-density,Papyan-Class-CrossClass,Li-etal-Hessian-based-analysis,MartinMahoney,EdgeOfStability-GD,xie2023on}. These works studied both the setting where these matrices were computed using the training data and where they were computed using the test data. The work of Papyan~\cite{Papyan-3-level} refined the above decomposition, noticing a three-level hierarchical decomposition of the Hessian for deep networks after training, with a bulk, the emergent outliers, and a class-cross-class \emph{minibulk} in between. This also served as a motivation for the observed neural collapse phenomenon of Papyan--Han--Donoho~\cite{PapyanDonoho}. 

In certain settings, these empirical studies have been supported by theoretical analyses. A set of such works focus on special points in the landscape like the limiting spectra at initialization, or at critical points (often towards computing expected numbers of saddle points and local minima)~\cite{PenningtonBahri,PenningtonWorah-FIM,maillard-landscape-complexity} and the very recently posted~\cite{asgari2025localminimaempiricalrisk}. Liao and Mahoney~\cite{liao2021hessian} derived deterministic equivalents for the Hessian spectrum of generalized linear models, and described conditions under which there may or may not be outliers. Recently, Garrod and Keating~\cite{garrod2024unifyinglowdimensionalobservations} proved results on the spectra of deep linear networks. In terms of relating the spectral theory to training dynamics, Gur-Ari--Roberts--Dyer ~\cite{GD-in-tiny-subspace} rigorously considered a classification of a mixture of two atoms (the zero variance limit of a binary Gaussian mixture) and showed that gradient descent aligns with a \emph{low-dimensional} principal subspace of the Hessian matrix. They postulated that this alignment may be key to understanding the training dynamics in many classification tasks in machine learning.  We also refer the reader to Jacot--Gabriel--Hongler~\cite{Jacot-NTK-Hessian} and Fan and Wang~\cite{ZhouWang} for study of the Hessian spectrum over the course of training in the neural tangent kernel (infinite width) limit. 

Closest to the context of this paper, \cite{BGJH23} established the structure of 1--2 above, and the alignment of the SGD  with the outlier eigenvectors, in high-dimensional Gaussian mixture classification tasks when the variance of the Gaussians was $\Omega(1)$ but a sufficiently small constant, and the ratio of samples to dimension was $O(1)$ but a sufficiently large constant. Taking these parameters to be small and large respectively, allowed perturbative analysis about simpler matrices and training dynamics trajectories.  

In this paper, we seek to understand the spectra sharply, i.e., for all order-$1$ variances $\lambda^{-1}$, and all order-$1$ sample complexities $\phi$, with exact characterizations of limiting (as $n,d\to\infty$) bulk distributions, outlier locations, and outlier eigenvector subspaces. 
We focus, in particular, on the interplay between the bulk and outliers of matrices formed from the loss landscape, with the low-dimensional subspace of summary statistics of the parameter, through which we also describe the typical trajectories taken by SGD. 
Mathematically, this requires refinements of arguments from random matrix theory, to understand the spectral theory of a family of self-coupled empirical matrices, which will be described in Section~\ref{subsec:general-random-matrix-result} below.

\subsection{Our contributions}
Many common models of high-dimensional statistical tasks, in both classification and regression settings, can be described as follows: We are given data $Y\in \mathbb R^d, Y\sim \cP_Y$, possibly class labels $y$, and parameters $\bfx\in\R^{d\times \cC}$, and we perform risk based learning where the loss function, $L$, depends only on the class label and the inner product of the data and parameter matrix $\bfx^\top Y$. 
Consider the empirical Hessian matrix and what we will refer to as the Gradient matrix (the generalization of the Fisher information beyond log-likelihoods), formed out of $n$ i.i.d.\ data points $(Y_\ell)_{\ell=1}^n$, as 
\begin{align}\label{eq:test-Hessian-Gram}
    \widehat H(\mathbf{x}) = \frac{1}{n} \sum_{\ell=1}^n \nabla^2 L(\mathbf{x},Y_{\ell})\,, \qquad \text{and} \qquad \widehat G(\mathbf{x}) = \frac{1}{n} \sum_{\ell=1}^n \nabla L (\mathbf{x},Y_\ell)^{\otimes 2}\,,
\end{align}
at parameter value $\mathbf{x} \in \mathbb R^{d\times \cC}$. In problems of this type, these matrices are highly structured. They are $\cC \times \cC$ block matrices, where each block corresponds to one of the $x^\alpha \in \mathbb R^d$ for $\alpha \in [\cC]$. (For example in classification tasks, this is the block for the $\alpha$'th one-vs-all classifier.) Each block is then of the form
\begin{align}\label{eq:A-D-A^t}
\frac{1}{n} A D A^\top \quad \text{where $A = \left[\begin{array}{ccc} \vert &   &   \vert \\ Y_1 & \cdots &  Y_n \\ \vert &  & \vert \end{array}\right]\in \mathbb R^{d\times n}$}  \quad \text{and} \quad \text{$D = \diag(D_{\ell\ell})\in \mathbb R^{n\times n}$} 
\end{align}
where the $\ell$'th entry $D_{\ell \ell}$ depends only on the $\ell$'th data point's class label and the inner products $Y_\ell^\top\bfx$. This yields a family of self-coupled empirical matrices,  that are of Wishart-type but complicated by the (often non-linear) dependencies between $D$ and $A$ and the potential non-positivity of $D$. In this paper, we will focus on the setting where the data is drawn from a $k$-component Gaussian mixture model where each component has covariance $I/\lambda$, and with $k$ not necessarily equal to $\cC$.  Here $k, \cC,$ and $\lambda$ are fixed numbers, which do not depend on $d, n$, and we call $\lambda>0$ the signal-to-noise ratio.  Such data models occur in many well-studied examples as classifying Gaussian mixture models
with one and two layer neural networks, regression for single and multi-index models, among many others. It is important to note that this model allows for scenarios where the classes may not be linearly classifiable.

Our goal in this paper is to study the high-dimensional (as $n,d,p\to\infty$ together) spectral theory of the Hessian and Gradient matrices of such empirical risks, and in particular, how it varies across the parameter space. We are particularly interested in the spectra of these objects along the training trajectory, $(\bfx_t)_{t\ge 0}$, where the training algorithm is online SGD. (From now on, SGD will only refer to its online form.)
To do so, we prove a general random matrix result expressing the limiting bulk spectrum and outlier locations for matrices of the form ~\eqref{eq:A-D-A^t}: see Theorem~\ref{thm:main-A-D-A^T}--\ref{thm:main-A-D-A^T-outliers}. Notably, these limits, and hence the local loss landscape geometry, depend on the  point $\bfx$ in parameter space \emph{only} through a finite-dimensional family of \emph{summary statistics} of $\bfx$. These summary statistics are given by the $\cC+k\times \cC+k$-dimensional Gram matrix $\mathbf{G}(\bfx) = (\mathbf{x},\bmu)^\top (\bfx,\bmu)$, where $\bmu= (\mu_1,...,\mu_k)$ are the class means.  
Summarizing, 
\begin{itemize}
    \item Despite the dependencies, the bulk spectrum converges to the free multiplicative convolution of a Marchenko--Pastur distribution with the law of the entries of $D$ (which only depends on the summary statistics): Theorem~\ref{thm:main-A-D-A^T};
    \item Outlier locations are given by fixed point equations expressible using only the summary statistics; the number of solutions to the outlier equations are bounded by the number of projections of $Y$ on which the law of the entries of $D$ may depend: Theorem~\ref{thm:main-A-D-A^T-outliers}
\end{itemize}  

In statistical tasks of this form, if $\mathbf{x}_\ell$ follows SGD, then the time-rescaled summary statistic values, $(\mathbf{G}(\mathbf{x}_{\lfloor tn\rfloor}))_{t\in [0,T]}$ are typically asymptotically autonomous, meaning they converge to the solution of an autonomous system of ODEs/SDEs in the high-dimensional limit (in particular this is the case in all the examples mentioned earlier, per e.g.,~\cite{BGJ22}). Together, this gives a dimension-free characterization of the bulk spectrum and outlier locations of the empirical Hessian and empirical Gradient matrix in parameter space in a way that can be tracked along the typical trajectories traced out by SGD. 

With that, returning to our motivation from the more empirical literature, as well as from classic spectral transitions encountered in random matrix theory, our aim is to characterize signal-to-noise (inverse variance) thresholds, and thresholds along the SGD path, at which outliers pop out from the bulk, and/or separate from one another.  Per Theorems~\ref{thm:main-A-D-A^T}--\ref{thm:main-A-D-A^T-outliers}, the evolution of the bulk spectrum and outliers along the typical SGD trajectory are described by evolutions of solutions to finite-dimensional fixed point equations as $\mathbf{G}_t$ evolves by some finite  dimensional ODE system.  Thus, deducing concrete consequences can be analytically very complicated. 

Due to this complexity, we will begin our presentation by carefully analyzing the simplest, and most canonical classification task in high-dimensions,  namely classifying a Gaussian mixture model in high-dimensions via multi-class logistic regression.  We will allow for the number of classes to be different from the number of mixture components and for the classes to be non-linearly separable.
We present the effective spectral theory (bulk and outliers) of high-dimensional logistic regression in Theorems~\ref{mainthm:logistic-regression}--\ref{mainthm:logistic-regression-outliers} as a function only of the summary statistics. We then combine this with ODEs for the evolution of summary statistics under SGD from~\cite{BGJH23} to derive the following consequences on the spectra of the empirical Hessian and Gradient matrices~\eqref{eq:test-Hessian-Gram} generated with test data at fixed variance and sample complexities, at initialization and over the course of order $n$ steps of training. 
\begin{itemize}
    \item\emph{Global spectral transitions}: We find that for any point in parameter space with summary statistic matrix $\bfG$, there is a threshold in signal-to-noise ratio, such that for all signal-to-noise ratios above this, there are outliers. See Corollary~\ref{cor:BBP-at-all-points}. 

    \item \emph{Sharp spectral transitions at initialization}: We give critical signal-to-noise ratios at which different outliers emerge from the bulk at initialization, and demonstrate that these differ between different natural uninformative initializations. See Corollaries~\ref{cor:exact-distribution-at-initialization}--\ref{cor:Gaussian-init-distribution}. 
    
    \item \emph{Dynamical spectral transitions}: The evolution instantaneously induces eigenvalue splitting. More precisely, in the early stages of training (the first $\epsilon n$ steps), to first order in $\epsilon$, the bulk does not change. Over the same period of time, the outlier of multiplicity $k$ at initialization splits into one increasing eigenvalue, with eigenvector correlated with the class mean, and a minibulk of the other $k-1$ which decrease and have eigenspace that is anti-correlated with the class mean. See Theorem~\ref{thm:effective-bulk-along-effective-dynamics}. 
\end{itemize}
% It is also important to note here that despite being related to a likelihood, we find that the empirical matrices $\widehat H$ and $\widehat G$ have completely different spectra in this high-dimensional setting.
Since much of our motivation are these latter interplays between the spectral theory and the SGD trajectory through parameter space, our presentation of the formal results begins with the special case of logistic regression for classifying high-dimensional Gaussian mixtures, where the expressions are more explicit: Section~\ref{subsec:kGMM-main-results}. We then move to the general random matrix theory result in Section~\ref{subsec:general-random-matrix-result}, and explain other high-dimensional statistical tasks to which it can be applied in Section~\ref{subsec:further-applications}.

\subsection{Logistic regression for supervised classification of Gaussian mixtures}\label{subsec:kGMM-main-results}
Let us explore our results in the context of one of the most canonical supervised classification tasks namely, logistic regression for a $k$-component Gaussian mixture model, which is as follows.  In this setting, the limiting spectral theory and training dynamics are more explicit and the interplay between them can be more easily examined analytically. 

Suppose that we are given  $n$ i.i.d.\ samples from a $k$-component mixture of $d$-dimensional isotropic Gaussian distributions with  distinct means $\mu_1,...,\mu_k$ with $\|\mu_i\| = O(1)$\footnote{When we use asymptotic notation like $O(1)$, we mean as compared to $n,d,p$ which will be large together.} and weights $(p_c)_{c\in[k]}$ for some fixed $k$. That is, the data are i.i.d. draws $(Y_\ell)_{\ell=1}$ of the form
\begin{align}\label{eq:data-distribution}
Y \sim \mu_a + Z_\lambda \qquad \text{where} \quad a\sim \sum_{c\in [k]} p_c \delta_c \qquad \text{and} \quad Z_\lambda \sim \mathcal N(0,I_d/\lambda)\,.
\end{align}
Here the parameter $\lambda$ is the inverse variance, or alternatively a \emph{signal-to-noise ratio} (SNR), and we will refer to the random variable $a$ as the \emph{hidden label}. 
In addition, each data point $Y$ comes with a \emph{class label} $y \in [\cC]$, where $y$ is a deterministic function of the corresponding hidden class label. (In the following, we abuse notation and denote this function by $\cC(a)$.) \footnote{In particular, there may be more than one mean corresponding to the same class, but the choice of mean dictates the class fully.} We sometimes naturally identify the label $y$ with a one-hot vector $y\in \{0,1\}^\cC$.  We denote the distribution of the labeled data $(y,Y)$ by $\mathcal P_Y$. 

Given this data, we seek $\cC$ distinct $1$-vs-all hyperplane classifiers, whose normal vectors are encoded by $\bfx = (x^a)_{a\in\cC}$. We find these classifiers by optimizing the cross-entropy loss
\begin{align}\label{eq:loss-function-1-layer}
    L(\mathbf{x},(y,Y)) = - \sum_{c\in [\cC]} y_c x^c\cdot Y  + \log \sum_{c\in [\cC]} \exp( x^c \cdot Y)\,,  \qquad\text{where} \quad \mathbf{x}= (x^a)_{a\in [\cC]}\,,\, x^a \in \mathbb R^d\,.
\end{align}
(Note that since the means are in general position, one may also include a bias parameter for each of the $\cC$ classes. We omit this in our discussion to minimize notational burden, but note that this can be handled by a trivial modification of what follows.)

\subsubsection{Effective spectral theory}
We begin by characterizing the spectra of the Hessian and Gradient matrices~\eqref{eq:test-Hessian-Gram} and their dependence on the parameter $\bfx$ in the high-dimensional regime of $n,d$ getting large proportionately to one another. 
In the following, for a subset $I$ of parameter coordinates, we let $\widehat H_{I,I}$ and $\widehat G_{I,I}$ denote the corresponding blocks of the Hessian/Gradient matrices. When it is clear from context, we will use the shorthand $H$ and $G$ to refer to these matrices.

By computing the derivatives of~\eqref{eq:loss-function-1-layer} in the block $x^\alpha$, it is clear that these matrices are of the form of~\eqref{eq:A-D-A^t} where the diagonal entries only depend on $Y$ through its label $y$ and its inner products with $\bfx=(x^\alpha)_{\alpha=1}^\cC$ and $\bmu=(\mu_i)_{i=1}^k$ (see Section~\ref{sec:logistic-regression} for the relevant calculations). To be explicit, if we define the soft-max function $\varphi_\alpha(\vec{z}) = \frac{\exp(z_\alpha)}{\sum_{j} \exp(z_j)}$, then the $bc$-block of the Hessian and gradient matrices are  
\begin{align*}
            \widehat G_{bc} &=  \frac{1}{n} A D A^\top \,\, \text{where} \,\, A = [Y_1 \cdots Y_n] \,, \quad D= \text{diag}((y_{\ell,b} - \varphi_b(\bfx^\top Y_\ell)  ) (y_{\ell,c} - \varphi_c(\bfx^\top Y_\ell))\,.\\
        \label{eq:Hessian-as-A-D-A^t}
    \widehat H_{bc} &=  \frac{1}{n} A D A^\top \,\, \text{where} \,\, A = [Y_1 \cdots Y_n] \,, \quad D= \text{diag}(\varphi_b(\bfx^\top Y_\ell) \delta_{bc} - \varphi_b(\bfx^\top Y_\ell)\varphi_c(\bfx^\top Y_\ell))\,.
\end{align*}

Therefore, by isotropy of $Y- \mu_a$, the \emph{law} of $D_{\ell\ell}$ only depends on $\mathbf{x}$ through the Gram matrix of the parameters and class means
\begin{equation}\label{eq:define_G}
    \mathbf{G} = (\mathbf{x},\bmu)^\top (\mathbf{x},\bmu) \in \mathbb R^{(\cC+k) \times (\cC+k)}\,
\end{equation}
Following the terminology of ~\cite{BGJ22}, we refer to this as the \emph{summary statistic} matrix. In \cite{BGJH23},  it was shown that $\bfG$ evolves asymptotically autonomously under SGD (see~\cite[Theorem 5.7]{BGJH23}).

We begin our discussion by describing the high-dimensional empirical Hessian and Gradient matrices' spectra (bulk and outliers) in terms only of the summary statistic matrix $\mathbf{G}$. 
Let 
\begin{align}\label{eq:q-qbar}\bar{q}:=\dim(\text{Span}(x^1,...,x^\cC,\mu_1,...,\mu_k))\leq q:=\cC+k\,.
\end{align}
For a summary statistic matrix $\mathbf{G}$, define its unique positive semi-definite square root $\sqrt{\mathbf{G}}\succeq 0$. The map from $\mathbf{G}$ to $\sqrt{\mathbf{G}}$ is continuous, so if we have a sequence $\mathbf{G}^{(d)}\to \mathbf{G}$, then $\sqrt{\mathbf{G}^{(d)}}\to \sqrt{\mathbf{G}}$.

In both the Hessian and Gradient matrix cases, there is a coupled i.i.d.\ family of Gaussians $\vec{g}$ such that the entry $D_{\ell \ell}$ with hidden label $b\in [k]$ is given by a function of the form
\begin{align*}
    f_b(\lambda^{-1/2} \vec{g}; \mathbf{G})\,, \qquad \text{where}\quad \vec{g} = (g_1,...,g_q) \sim \mathcal N(0,I_q) \,.
\end{align*}
To be concrete, define the soft-max function $\varphi_\alpha(\vec{z}) = \frac{\exp(z_\alpha)}{\sum_{j} \exp(z_j)}$, and define
\begin{equation}
\begin{aligned}\label{eq:f_a-kGMM-case}
    f_b^H(\lambda^{-1/2} \vec{g}; \mathbf{G}) &= [\varphi_\alpha(1-\varphi_\alpha)](\bfG_{{[\cC]} b}+\bfG_{[\cC,\cC]}^{1/2}\lambda^{-1/2} \vec  g_{[\cC]})\,,\\ 
    f_b^G(\lambda^{-1/2} \vec{g}; \mathbf{G}) &=  (\mathbf 1_{\alpha = b} - \varphi_\alpha(\bfG_{{[\cC]} b}+\bfG_{[\cC,\cC]}^{1/2}\lambda^{-1/2} \vec  g_{[\cC]}) )^2\,, 
\end{aligned}
\end{equation}
which will correspond to the $(\alpha,\alpha)$-block of the empirical Hessian matrix $\widehat H$ or empirical Gradient matrix $\widehat G$ respectively.
Here we are using the subscript $[\cC]$ to pick out the sub-vector with indices in $[\cC]$ and $[\cC,\cC]$ to pick out the principal sub-matrix with both indices in $[\cC]$.

Recall that the Stieltjes transform of a probability measure, $\nu\in\cM_1(\R)$, is defined by 
\begin{align}\label{eq:Stieltjes-transform}
    S_\nu(z) := \int_{\mathbb R} \frac{1}{t-z} d\nu(t)\,, \qquad z\in \mathbb C\,.
\end{align}
We can now define the \emph{effective bulk}.

\begin{defn}\label{def:effective-bulk}
   Fix $\lambda,\phi>0$. The \emph{effective bulk} for the $(\alpha,\alpha)$-block of the Hessian at summary-statistic matrix $\mathbf{G}\in \mathbb R^{q\times q}$ is the unique probability measure $\nu^{H,\alpha}_{\mathbf{G}}\in \mathscr{M}_1(\mathbb R)$ whose Stieltjes transform $S^{H,\alpha}_\bfG(z)$ solves 
    \begin{align}\label{e:defmz-with-rho}
        1+ zS(z) = \phi \sum_{b=1}^{k} p_b \mathbb E\Big[ \frac{S(z) f_b^H(\lambda^{-1/2}\vec g;\mathbf G)}{\lambda \phi + S(z) f_b^H(\lambda^{-1/2}\vec g;\mathbf G)}\Big]\,,
    \end{align}
    where the expectation on the right-hand side is with respect to the Gaussian vector $\vec g$.
    The \emph{effective bulk} for the $(\alpha,\alpha)$-block of the Gradient matrix is the unique probability meausure, $\nu^{G,\alpha}_\bfG$, whose Stieltjes transform, $S^{G,\alpha}_\bfG$, solves the same equation with $f^G_b$ in place of $f^H_b$.
\end{defn}

To understand this definition, it helps to observe that \eqref{e:defmz-with-rho} defines the free multiplicative convolution of the Marchenko--Pastur distribution with aspect ratio $\phi^{-1}$ and variance $\lambda^{-1}$, which we denote by  $\nu_{\MP}$, with the distribution of $f_J^{H}(\lambda^{-1/2}\vec{g},\bfG)$ for $J\sim \text{Cat}(p_j)_{j=1}^k$, meaning $\mathbb P(J=j) = p_j$ for each $j\in [k]$. That is, if $\rho_\bfG$ denotes the distribution of $f_J^H$ then $\nu_\bfG^{H,\alpha}=\rho_\bfG\boxtimes\nu_{\MP}$. (Here and in the following $\boxtimes$ denotes free multiplicative convolution.) Recall that this is well-defined \cite{marchenko1967distribution}. For $\nu\in\cM_1(\R)$, let $\supp(\nu)$ denote its support, and $\supp^+(\nu) = \sup \{x\in \supp(\nu)\}$ be its \emph{right edge}. 

We now turn to the effective outliers.
To this end, for any $\lambda,\phi>0$, and $M\in\{ H, G\}$ define the matrix valued function $F_\alpha^M(z,\mathbf{G}) \in \mathbb R^{q\times q}$, 
\begin{equation}\label{eq:F-matrix}
     F_{ij,\alpha}^M(z,\bfG)= \lambda\phi\sum_{b=1}^k p_b \mathbb E\left[\frac{f_b^M(\lambda^{-1/2}\vec g; \mathbf{G})}{\lambda \phi+S^M(z)f_b^M(\lambda^{-1/2}\vec g;\mathbf{G})} (\frac{g_i}{\sqrt{\lambda}}+(\sqrt{\mathbf G})_{ib})(\frac{g_j}{\sqrt{\lambda}}+(\sqrt{\mathbf G})_{jb})\right]\,,
\end{equation}
where $S^M(z)=S_{\bfG}^{M,\alpha}(z)$ for $M\in \{H,G\}$ are as in Definition~\ref{def:effective-bulk}. 
We then have the following.

\begin{defn}\label{def:effective-outliers}
    Fix $\lambda,\phi>0$ and $\alpha\in[\cC]$. For the $(\alpha,\alpha)$-block of the Hessian, the \emph{effective outliers} $\mathcal Z_\alpha^H=(z^H_{*,i})_{i}$ are the real roots  (with multiplicity)  of the equation 
\begin{align}\label{eq:outlier-equation-general}
    \det(z I_q - F^H_\alpha(z,\bfG)) =0\,.
\end{align}
For any effective eigenvalue $z\in \mathcal Z_\alpha^H$, call $u\in \mathbb R^q$ a solution to the \emph{effective eigenvector equations} if
\begin{align}
(z I_q - F^H_\alpha(z,\bfG))u &  =0 \label{eq:effective-evect}\\
\norm{u}^2 -\langle u,\partial_z F^H_\alpha(z,\bfG)u\rangle - 1&  =0.\label{eq:effective-evect-proj}
\end{align}
Call $u\in \mathbb R^q$ an \emph{$\epsilon$-approximate solution} if the left-hand sides are at most $\epsilon$ in norm. 
Define the notions for the Gradient matrix analogously with $F^G$ in place of $F^H$.
\end{defn}

Solutions to~\eqref{eq:outlier-equation-general} may or may not exist. Lemma~\ref{l:num-solutions} shows that when they do exist, the number of solutions in each connected component of $\mathbb R\setminus \supp(\nu_{\mathbf{G}}^H)$ is at most $\bar q$ from~\eqref{eq:q-qbar}. The number of such connected components will depend on the spectrum of $D$, and in particular number of connected components  in its limiting support.

Our main theorem says that the high-dimensional behavior of the spectrum of the empirical Hessian and Gradient matrices at point $\mathbf{x}$ with summary statistics $\mathbf{G}$ in parameter space are given by the effective bulk $\nu_{\mathbf{G}}$ and the effective outliers. To state this we need a little more notation.

In the following, let $d_{\text{BL}}$ denote the Bounded Lipschitz metric on $\cM_1(\R)$. 
For each $\alpha\in[\cC]$,  let $\hat\nu^{H,(d)}$ and $\hat\nu^{G,(d)}$ denote the empirical spectral measures of the $(\alpha,\alpha)$-block of the Hessian and Gradient matrix respectively at parameter value $\mathbf{x}\in \mathbb R^{\cC d}$.  (We suppress the $\alpha$ dependence to reduce notation.) 
We then have the following which says that the empirical spectral measure of the Hessian or Gradient matrix can be approximated by the solution of~\eqref{e:defmz-with-rho}, that is, by the effective bulk.

\begin{thm}[Bulk]\label{mainthm:logistic-regression}
Fix $M\in\{H,G\}$, $\alpha\in[\cC]$, and $\lambda,\phi>0$. 
Let $n,d$ get large with $n/d=\phi$. For any $\mathbf x\in \mathbb R^{\cC d}$ with summary statistics matrix $\bfG \in \mathbb R^{q\times q}$, we have that with probability $1-o(1)$, 
\[
d_{\text{BL}}(\hat\nu^{M,(d)},\nu^{M}_\bfG) =o(1)\,.
\]
Furthermore, given a sequence of parameters $\mathbf{x}^{(d)}$ with summary statistic matrices $\bfGd\to \bfG$ as $n,d\to\infty$ at rate $\|\bfGd - \bfG\| = o(1/(\log d))$, we have that $\hat\nu^{M,(d)}\to\nu^M_\bfG$ weakly a.s. 
\end{thm}

Let us now turn to the outliers. To this end, it helps to introduce further notation. Let $\sigma_{\alpha}^{H,(d)}=\sigma(\widehat H_{[\alpha,\alpha]})$ be the spectrum of the Hessian and define $\sigma_\alpha^{G,(d)}$ analogously for the Gradient matrix.  Furthermore, given $[x^1,...,x^\cC,\mu_1,...,\mu_k]$, let $L$ denote any $d\times q$ matrix with orthonormal columns satisfying 
\begin{equation}\label{e:L-defining-property}
L^\top [x^1 \,...\,x^\cC\, \mu_1\, ...\, \mu_k] = \sqrt{\bfG}.
\end{equation}
That such a matrix exists can be shown via the compact singular value decomposition (see Section~\ref{sec:isolating-dependences} for more details).
We then have the following which says, informally, that the collection of  eigenvalues outside of the bulk are close to the effective outliers (with multiplicity) and the projection of their corresponding (unit) eigenvectors into the summary statistic space approximately solve~\eqref{eq:effective-evect}--\eqref{eq:effective-evect-proj}.

\begin{thm}[Outliers]\label{mainthm:logistic-regression-outliers}
Fix $M\in\{H,G\}$, $\alpha\in[\cC]$, and $\lambda,\phi>0$. Recall that $\mathcal Z^M$ is the set of effective outliers. 
For any $\bfG$ and any $[a,b]\subseteq \supp(\nu^M_\bfG)^c$ with $\{a,b\}\cap \mathcal{Z}^M =\emptyset$, we have that for $n,d$ sufficiently large with $n/d=\phi$ and any $\mathbf x$ with $\bfG(x)=\bfG$, with probability $1-o(1)$, 
\[
|\sigma^{M,(d)} \cap [a,b]| = |\mathcal {Z}^M\cap[a,b]|,
\]
Moreover, for every $\delta>0$ small enough and every $L$ satisfying \eqref{e:L-defining-property}, with probability $1-o(1)$, 
for all $z\in\cZ$, if $z^{(d)}\in\sigma^{M,(d)}\cap B_\delta(z)$ then any corresponding eigenvector $v^{(d)}$ is such that $u^{(d)}=L^\top v^{(d)}$  is a $o(1)$-approximate solution to the effective eigenvector equation~\eqref{eq:effective-evect}. If further $z\neq0$, then it also is a $o(1)$-approximate solution \eqref{eq:effective-evect-proj}.  
Finally, given a sequence of parameters, $(\mathbf x^{(d)})$, with summary statistic matrices $\bfGd\to\bfG$ as $d\to\infty$ at rate $\|\bfGd - \bfG\| = o(1/(\log d))$, the analogous convergence statements holds as $n,d\to\infty$ with $n/d= \phi$ almost surely. 
\end{thm}

\begin{figure}
\includegraphics[width=0.32 \textwidth]{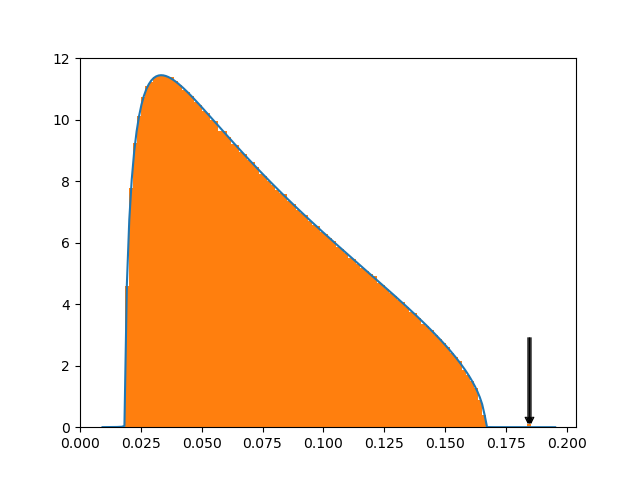}
\includegraphics[width=0.32 \textwidth]{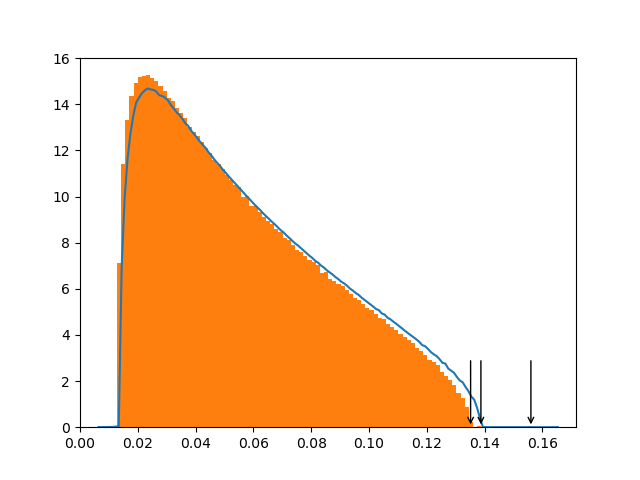}
\includegraphics[width=0.32 \textwidth]{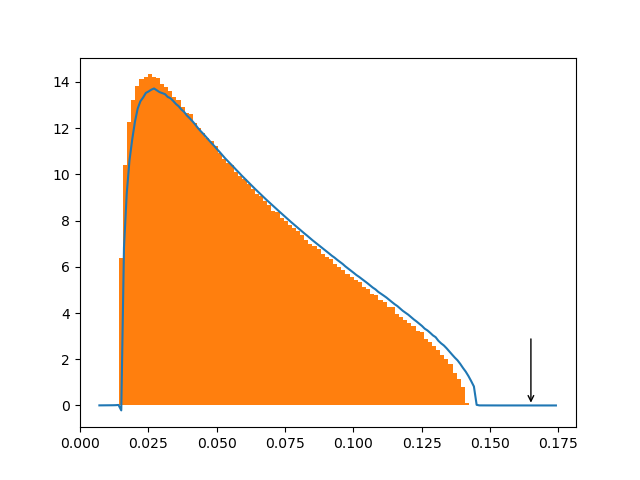}
    \caption{Histogram of the empirical Hessian spectrum (orange) at various points in parameter space in $d=20,000$ $k=3$, $\lambda =3$ and $\phi=4$. The parameter values plotted are, from left to right, at $\mathbf{x} \equiv 0$, at $\mathbf{x} \equiv \mu+  \mathcal N(0,I_d/d)$, and at the optimal classifier $\mathbf {x}^\alpha = (1-\frac{1}{k})\mu^\alpha - \sum_{j\ne \alpha}\frac{1}{k} \mu_j$. Arrows point to the empirical locations of the outlier eigenvalues for visibility, and the blue curve is the theoretically predicted bulk spectrum at those summary statistic values.}\label{fig:bulk-spectra}
\end{figure}

In words this is saying that in any interval outside the support of the limiting bulk, the numbers of eigenvalues (counting multiplicity) approximates the number of effective outlier solutions in that interval. Moreover, the outlier eigenvectors approximately solve the corresponding effective eigenvector equations which described the eigenvectors' projection into summary statistic space.  See Figure~\ref{fig:bulk-spectra} for numerical realizations of the spectrum at $d=20,000$, together with its effective $n,d\to\infty$ limits, at several natural points in parameter space.

 Having fully characterized the limiting spectral theory as the summary statistics vary, we now relate this to the spectra seen over the course of training by following the effective bulk and effective outliers as $\mathbf{G}$ evolves per the limiting evolution of summary statistics under SGD. 

\subsubsection{The effective spectrum at initialization and along the training trajectory}
The (online) SGD with respect to the loss~\eqref{eq:loss-function-1-layer} with $\ell^2$ regularization with parameter $\beta$, and with step-size $\eta = c_\eta/n$, initialized at $\mathbf{x}_0$, is defined recursively by 
\begin{align}\label{eq:SGD-def}
    \mathbf{x}_\ell = \mathbf{x}_{\ell -1} - \eta \nabla L(\mathbf{x}_{\ell-1}, (y,Y)_{\ell}) - \eta \beta \mathbf{x}_{\ell-1}\,.
\end{align}

It was shown in~\cite{BGJH23}, following the general framework of~\cite{BGJ22}, that as $n,d \to\infty$ and $\eta\to0$ proportionately , the evolution of the summary statistics under SGD converge to an autonomous system of ODEs called their \emph{effective dynamics}. More precisely  there is a function $\mathsf{F} = \mathsf{F}_{\lambda ,\phi,\beta,c_\eta}$ such that for each $T>0$, the limit in probability
\begin{align}\label{eq:effective-dynamics}
    (\mathbf{G}_t)_{t\in [0,T]} \stackrel{(\textsc{p})}= \lim_{d\to\infty} (\mathbf{G}(\mathbf{x}_{\lfloor t/\eta\rfloor}))_{t\in [0,T]} \qquad \text{exists and solves the ODE} \qquad d\mathbf{G}_t = \mathsf{F}(\mathbf{G}_t) dt\,.
\end{align}
initialized from $\mathbf{G}_0$ which is the limit (assuming it exists) of $\mathbf{G}(\mathbf{x}_0)$. Moreover, quantitative rates on $\sup_{t\in [0,T]} \|\mathbf{G}(\mathbf{x}_{\lfloor t/\eta\rfloor}) - \bfG_t\|$ could be obtained depending on the number of moments assumed on the loss (which in the logistic regression case are arbitrarily good). 
See Section~\ref{sec:logistic-regression-over-training} for the precise form of the effective drift function $\mathsf{F}$ for logistic regression. Note that the bottom-right $k\times k$ block (corresponding to the Gram matrix of the hidden means) is constant, and the bottom left and top right must be symmetric, so~\eqref{eq:effective-dynamics} is really only a dynamics on the top $(\cC\times q)$-submatrix of $\bfG$.
The approach of finding finite families of observables whose evolutions are autonomous for analyzing high-dimensional limits of SGD dates back at least to~\cite{saad1995line,saad1995dynamics}, and has been fruitful in many similarly structured problems in the past several years: see e.g.,~\cite{goldt2019dynamics,veiga2022phase,paquette2021sgd,pmlr-v195-arnaboldi23a,TanVershynin,BGJ22,DamianLeeSoltanolkotabi-22,mousavi-hosseini2023neural} for a necessarily small sampling. In particular, the other examples we will discuss in Section~\ref{subsec:further-applications} similarly admit their own effective dynamics.

By plugging this evolution of summary statistics $\mathbf{G}_t$ into our main result on the empirical Hessian and Gradient matrices at parameter values with different $\mathbf{G}$, we can study the effective spectrum, in particular the effective bulk and outlier locations, along the trajectory of the effective dynamical system resulting from SGD. This uses the crucial point that the expressions in~\eqref{e:defmz-with-rho}--\eqref{eq:outlier-equation-general} only depend on the value of $\mathbf{x}$ through $\mathbf{G}$. (All our statements are after the limit $d\to\infty$. See Remark~\ref{rem:test-vs-train} for the interpretation in finite $d$.)  From this viewpoint, possible high-dimensional spectral transitions (emergence and splitting of outliers) that occur as one varies the relevant parameters $\lambda, \phi$, and over the course of training, have all been reduced to finite-dimensional  problems.  

It is important to note that all our statements evaluating the effective spectra on the effective dynamics are after the limit $d\to\infty$. One can interpret these results in finite-$d$, by considering the case where the these matrices are computed using test data.  We expand on this more in Remark~\ref{rem:test-vs-train}, where we also include numerical evidence of a similar qualitative picture when the matrices are generated using the same set of training data: see Figures~\ref{fig:Hessian-spectrum-evolution-train}--\ref{fig:G-spectrum-evolution-train}.

Of course, even having reduced the analysis to finite-dimensional fixed point equations, solving these analytically can be very challenging, so we proceed to describe a few concrete implications we can derive. These will be exact descriptions of the bulk and outlier spectra at initialization, and then in the early stages of training (corresponding to the first $\epsilon n$ steps of the SGD for $\epsilon$ small). We will present all results in this subsection under the simplified setting of $\cC = k$ and $\mu_1,...,\mu_k$ orthonormal with equal weights $(p_j)_{j=1}^{k} = (\frac{1}{k},...,\frac{1}{k})$, though in Section~\ref{sec:logistic-regression-over-training} we state and prove the analogous results in greater generality.

We begin with the following description of the effective spectrum at natural initializations for $\mathbf{x}_0$ (e.g., either the all-zero or isotropic Gaussian parameter initialization). In these situations, the bulk distribution and its BBP-type spectral transitions (in $\lambda,\alpha$) are exactly solvable.

\begin{cor}\label{cor:exact-distribution-at-initialization}
    At the initialization $\mathbf{x} \equiv 0$, as $n,d\to\infty$ with $n/d=\phi$, and $\mu_1,...,\mu_k$ orthonormal with $(p_j)_{j=1}^{k} = (\frac{1}{k},...,\frac{1}{k})$, we have the following characterization of the effective spectra: 

    \smallskip
    \noindent \textbf{Hessian matrix}. The empirical Hessian $\widehat H_{\alpha\alpha}$'s effective bulk spectrum $\nu_{\mathbf{G}}^{H}$ is the scaled Marchenko--Pastur, with left and right edges at $z_{\pm} = \frac{k-1}{k^2}\cdot\frac{1}{\lambda}\cdot(1\pm \frac{1}{\sqrt{\phi}})^2$
    Further, if we define the critical $\lambda$, 
    \begin{align}\label{eq:lambda-c-H}\lambda_{c}^{H} := \frac{k}{\sqrt{\phi}}\,,\end{align}
    then there is a spectral transition at $\lambda_c^H$, whereby: 
    \begin{itemize}
        \item If $\lambda<\lambda_c^H$, then there are no effective outliers to the right of $\supp^+(\nu_{\mathbf{G}}^H)$;
        \item If $\lambda>\lambda_c^H$, there is a single effective outlier $z_{*} >\supp^+(\nu_{\mathbf{G}}^H)$ of multiplicity $k$; each of the corresponding $k$ eigenvectors has projection into $\text{Span}(\mu_1,...,\mu_k)$ of the form $c_j \mu_j +o(1)$ for some $j\in [k]$, where $c_j>0$ is given by~\eqref{eq:outlier-eigenvector-size-zero-init}. 
    \end{itemize} 

    \smallskip
    \noindent \textbf{Gradient matrix}. The Gradient matrix $\widehat G_{\alpha \alpha}$ at the $\mathbf{x} \equiv 0$ initialization has effective bulk,   
    \begin{align*}
        \nu_{\mathbf{G}}^G = \nu_{\MP} \boxtimes \Big(\frac{1}{k}\delta_{(k-1)/k^2} + \frac{k-1}{k} \delta_{1/k^2}\Big)\,,
    \end{align*}
    and the possible solutions of the effective outlier equations are $z_{*,1} >z_{*,2}$ (in the $\lambda$ regimes where they both exist), with $z_{*,1}$ being a simple effective outlier whose  prelimit eigenvector has $\Omega(1)$ inner product with $\mu_\alpha$, and $o(1)$ inner product with any of $(\mu_b)_{b\ne \alpha}$, and $z_{*,2}$ having multiplicity $k-1$ and corresponding eigenvectors with $\Omega(1)$ projection in $\text{Span}(\mu_b)_{b\ne \alpha}$ and $o(1)$ projection on $\mu_\alpha$. Explicit formulas defining $z_{*,1},z_{*,2}$ are given in~\eqref{eq:G-matrix-outlier-1}--\eqref{eq:G-matrix-outlier-2}.
\end{cor}

\begin{figure}
\includegraphics[width=0.32 \textwidth]{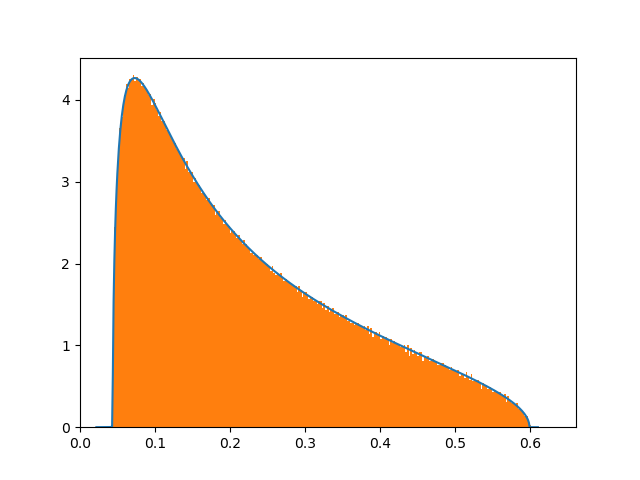}
\includegraphics[width=0.32 \textwidth]{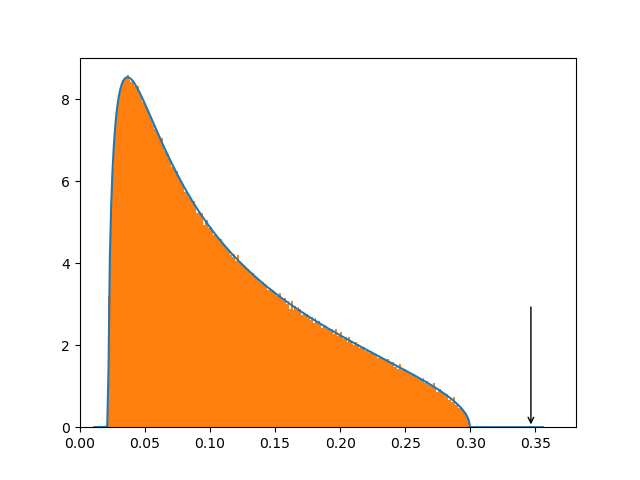}
\includegraphics[width=0.32 \textwidth]{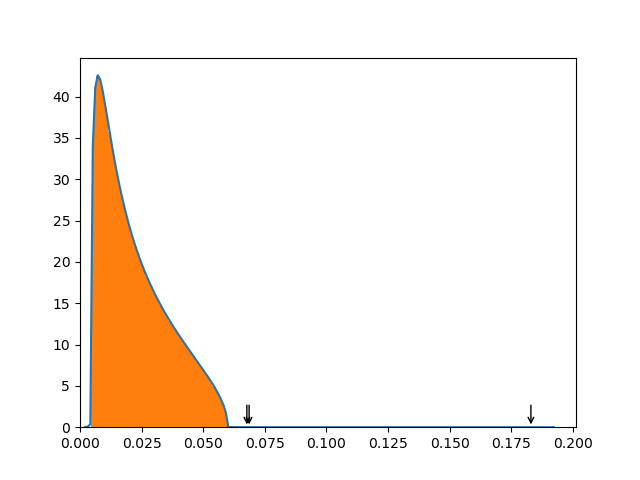}
    \caption{From left to right: the histogram of the empirical Gradient matrix (orange) at $\lambda = 1,2,20$ respectively, in $d= 20,000$ and $k=3$ and $\phi =4$ at initialization, with arrows pointing at the empirical locations of the outlier eigenvalues. This demonstrates the existence of distinct transition SNRs, $\lambda_{c,1}<\lambda_{c,2}$, for existence of one, then $k$ outlier eigenvalues as predicted by Corollary~\ref{cor:exact-distribution-at-initialization}.}\label{fig:Gradient-matrix-multi-BBP}
\end{figure}

See Figure~\ref{fig:Gradient-matrix-multi-BBP} for a depiction of the Gradient matrix at initialization at different SNR values, demonstrating two BBP-type transitions, one where $z_{*,1}$ emerges to the right of the bulk, and another where both $z_{*,2}<z_{*,1}$ are to the right of the bulk.

Another equally natural initialization is a standard Gaussian initialization with unit norm, i.e., $\mathbf{x} \sim \mathcal N(0, I_d/d)$. There, the expressions are more complicated, but can still be worked out fairly explicitly. For readability of the introduction, we state the following corollary with references to the relevant expressions for the fixed point equations defining the effective bulk and effective outliers. Let $\pi_{\REM}$ be a ``Random Energy Model"~\cite{Derrida-REM} on $\cC$ points, which is to say a Gibbs measure on $[\cC]$ with each element being assigned an independent $\mathcal N(0,\lambda^{-1})$ energy.
    
\begin{cor}\label{cor:Gaussian-init-distribution}
    If $\mathbf{x} \sim \mathcal N(0,I_d/d)$, for the Hessian $\widehat H_{\alpha\alpha}$, its effective bulk $\nu_{\mathbf{G}}^H$ will converge to 
    \begin{align*}
        \nu_{\mathbf{G}}^H = \nu_{\MP}\boxtimes \nu_{\REM}\qquad \text{where}\qquad \nu_{\REM} = \text{Law}(\pi_{\REM}(\alpha)(1-\pi_{\REM}(\alpha)))\,.
    \end{align*}
    There can be up to $q = \cC+ k$ many effective outliers. One possible effective outlier is of multiplicity $k$ and solves~\eqref{eq:Gaussian-init-fixed-pt-1} and the others consist of a simple effective outlier solving~\eqref{eq:Gaussian-init-fixed-pt-3} and an effective outlier of multiplicity $\cC-1$ solving~\eqref{eq:Gaussian-init-fixed-pt-2}.
\end{cor}

One comment on Corollary~\ref{cor:Gaussian-init-distribution} is that the two ``uninformative" initializations $\mathbf{x} \equiv 0$ and $\mathbf{x}\sim \mathcal N(0,I_d/d)$ have different effective spectral theories and thresholds for outlier emergence. Notably, the Gaussian initialization induces ``spurious" outliers with eigenvectors correlated with $\text{Span}(x^1,...,x^\cC)$ but orthogonal to all informative directions $\mu_1,...,\mu_k$. 

Though Corollaries~\ref{cor:exact-distribution-at-initialization}--\ref{cor:Gaussian-init-distribution} capture interesting spectral transitions as the SNR varies, the real strength of our results is the ability to track the spectra along the SGD trajectories in the high-dimensional limit, because each only depend on the same family of summary statistics $\mathbf{G}_{t}$ which we recall are the $n,d\to\infty$ limit of $\mathbf{G}(\mathbf{x}_{\lfloor t n\rfloor})$ over training. To demonstrate this, we describe a result on the evolution of the bulk and outliers along the SGD trajectory in the early stages of training.  
 Namely, we will implicitly differentiate the effective bulk and effective outlier equations~\eqref{e:defmz-with-rho}--\eqref{eq:outlier-equation-general} along the trajectory~\eqref{eq:effective-dynamics} at their initialization. Note that the infinitesimal behavior of the effective spectra at initialization captures the evolution of the empirical spectra after use of a macroscopic $\epsilon$ fraction of the data points, and is only infinitesimal in the dimension-free scale. 

Recall $\mathbf{G}_t$ is the summary statistic evolution solving~\eqref{eq:effective-dynamics}. For ease of the linear algebra calculations and to arrive at easily describable conclusions, we take the small-step size limit $c_\eta \to 0$ in the effective dynamics. 

\begin{figure}
\includegraphics[width = .32\textwidth]{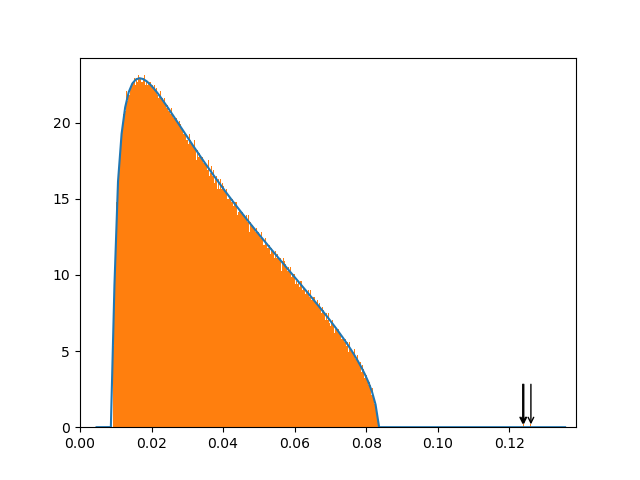}
\includegraphics[width = .32\textwidth]{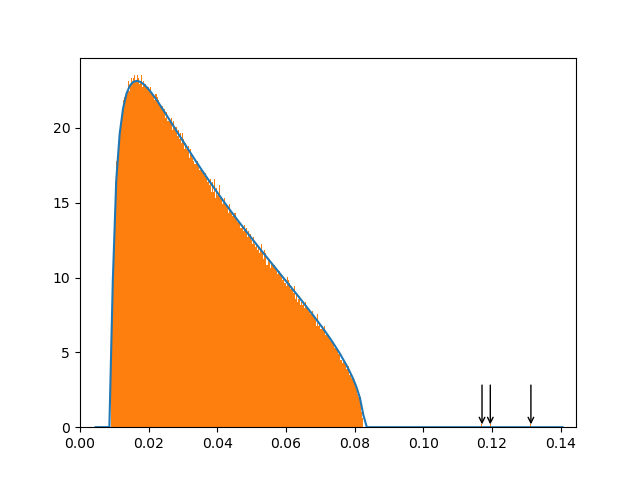}
\includegraphics[width = .32\textwidth]{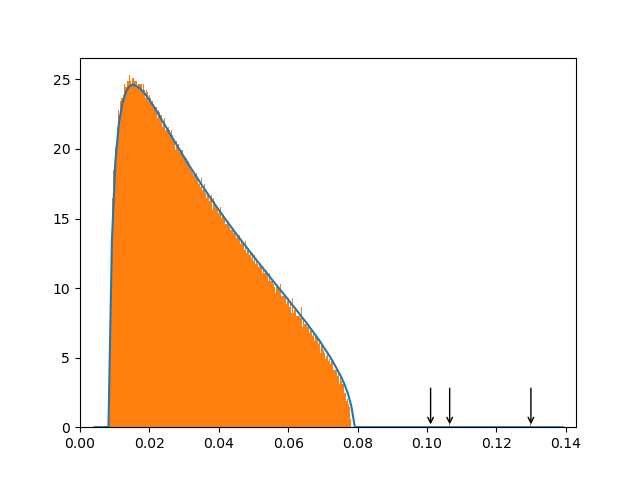}
    \caption{From left to right: spectra of the empirical Hessian computed with test data (orange histogram) overlaid with the predicted bulk spectrum given its summary statistic values (blue curve) after $0$, $n/5$ and $3n/5$ steps of online SGD, in $d=20,000$ with $k=3, \phi=4$. The figures demonstrate the splitting of the multiplicity-$k$  outlier eigenvalue, into a more pronounced outlier corresponding to the mean being learned by the classifier being trained, and the remaining $k-1$, as predicted by Theorem~\ref{thm:effective-bulk-along-effective-dynamics}.}\label{fig:Hessian-spectrum-evolution-test}
\end{figure}

\begin{figure}[h]
\includegraphics[width = .32\textwidth]{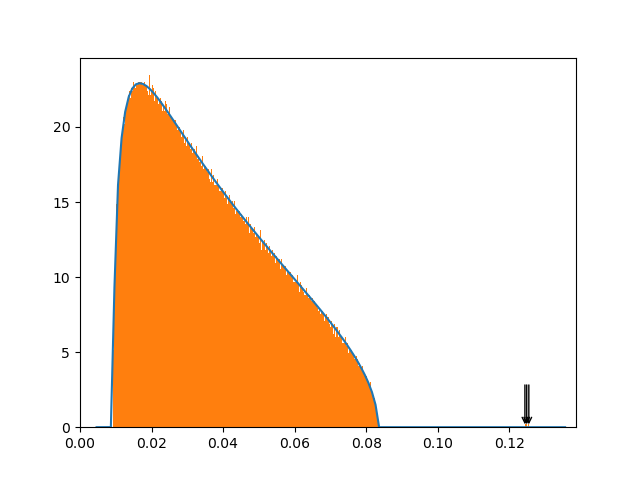}
\includegraphics[width = .32\textwidth]{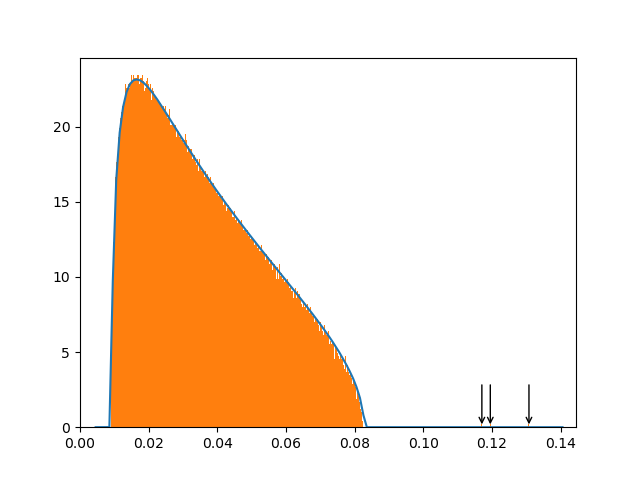}
\includegraphics[width = .32\textwidth]{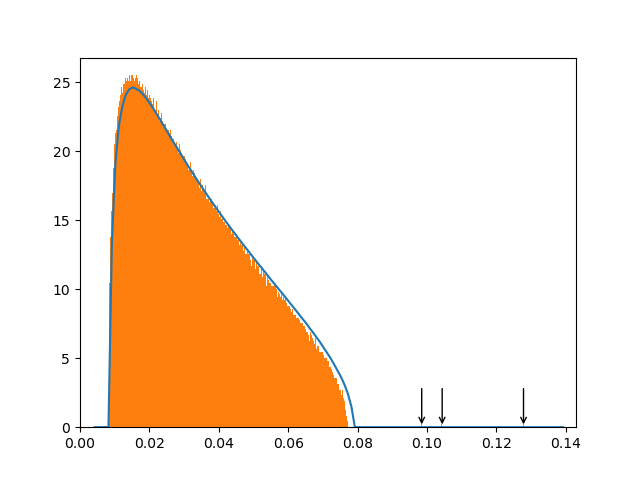}
    \caption{The analogue of Figure~\ref{fig:Hessian-spectrum-evolution-test} with train data used to generate the empirical Hessian matrix rather than test data. The similarity to Figure~\ref{fig:Hessian-spectrum-evolution-test} shows a good match regardless of which data set is used to generate the matrices.}\label{fig:Hessian-spectrum-evolution-train}
\end{figure}

\begin{thm}\label{thm:effective-bulk-along-effective-dynamics}
     Consider the effective bulk spectrum $\nu_{\mathbf{G}_t}^H$ of the empirical Hessian matrix $\widehat H_{\alpha \alpha}$, and the effective outliers $(z_{*,i}(\mathbf{G}_t))_i$ (with multiplicity), where $\mathbf{G}_t$ is as in~\eqref{eq:effective-dynamics} initialized from $\mathbf{x}_0\equiv 0$. 
    \begin{itemize}
        \item For any $\lambda>0$, the time derivative of the bulk is zero:  $\frac{d}{dt}S_{\nu_{\bfG_t}^H}(z)\vert_{t=0}=0$ for $z\in \supp(\nu_{\mathbf{G}_0}^H)^c$. 
        \item When $\lambda>\lambda_c^H$ from~\eqref{eq:lambda-c-H}, one effective outlier infinitesimally grows, i.e., $\frac{d}{dt} z_{*,1}(\bfG_t)\vert_{t=0}>0$\,.
        \item When $\lambda>\lambda_c^H$ from~\eqref{eq:lambda-c-H}, the other $k-1$ effective outliers have negative time derivatives $\frac{d}{dt} z_{*,i}(\bfG_t) \vert_{t=0} <0$ for $i= 2,...,k$. 
    \end{itemize}
\end{thm}

Recall from Corollary~\ref{cor:exact-distribution-at-initialization} that when $\lambda>\lambda_c^H$, there is a single multiplicity $k$ effective outlier of $\widehat H_{\alpha \alpha}$ at initialization. In words, Theorem~\ref{thm:effective-bulk-along-effective-dynamics} is describing a splitting of this outlier spectrum, as a single outlier dictating the direction of greatest curvature quickly points in the direction of the class mean $\mu_\alpha$. Shortly into training, therefore, the Hessian demonstrates a single informative outlier, a minibulk of size $k-1$, followed by a narrower bulk. See Figure~\ref{fig:Hessian-spectrum-evolution-test} for the depiction of the spectra at several time snapshots, demonstrating this splitting of the outlier spectrum, and initial lack of drift for the bulk.  
We note that we also derive exact fixed point equations for the time derivatives of Theorem~\ref{thm:effective-bulk-along-effective-dynamics} in~\eqref{eq:psi-def}. 

Moreover, with continuity of the effective outlier equations' solutions in $\lambda$ and the fact that the derivative bounds of bullet points 2--3 above hold for all $\lambda$ (only they don't actually correspond to an effective outlier when $\lambda\le \lambda_c^H$), there can even exist a regime of SNRs $\lambda<\lambda_c^H$ but sufficiently close to $\lambda_c^H$ such that at initialization, there are no effective outliers; but there exists a small $t_*(\lambda)>0$ such that at the summary statistic values $\bfG_t$ corresponding to $t_*n$ steps of training by SGD, there emerges a single effective outlier eigenvalue, whose prelimit eigenvector has $\Omega(1)$ correlation with~$\mu_\alpha$. 

The above discussions form \emph{dynamical versions} of the BBP-type~\cite{baik2005phase} spectral transitions wherein outliers emerge from a bulk, that traditionally occur from varying the SNR, but here are occurring with fixed SNR as one varies the training time. Emergence of outlier eigenvalues over the course of training should be even more pronounced (with delayed onsets) in more complicated problems, like single-index models whose first non-zero Hermite exponents are at least $2$, described in Subsection~\ref{subsec:further-applications}.

The above types of emergence of effective outliers as the SNR varies and/or SGD moves through the parameter space are a feature of a more general phenomenon that holds all throughout the parameter space. Associated to each point in parameter space with summary statistics $\mathbf{G}$, we expect there to be an associated $\lambda_c(\mathbf{G})$ such that for $\lambda$ below it there are no outlier eigenvalues, and above it there is at least one outlier eigenvalue. The following  establishes that everywhere in parameter space, for every $\phi \in (0,\infty)$, at least at sufficiently large $\lambda$ there are outlier eigenvalues. 

\begin{cor}\label{cor:BBP-at-all-points}
    Fix any $\mathbf{G}$, there exists $\lambda_0(\mathbf{G})$ such that for every $\lambda>\lambda_0$, a largest effective outlier (in either the Hessian or Gradient matrix cases) exists and satisfies $z_{*,1} >\supp^+(\nu_{\mathbf{G}})$. 
\end{cor}

We also comment that Corollary~\ref{cor:BBP-at-all-points} leaves open the possibility of bizarre scenarios like non-monotonicity of the existence of effective right-outliers as one varies $\lambda$: some of the subtlety arises from the fact that the law of entries of $D$ depend on $\lambda$ in complicated, possibly non-monotonic ways. Indeed, the reason we do not have a complementary result that at $\lambda$ small, there are no outliers, is that while small $\lambda$ widens the support of the Marchenko--Pastur from $\frac{1}{n} AA^\top$, it concentrates the soft-max in $D$ towards $\{0,1\}$ (increasing the temperature of the corresponding random energy model) so that the entries of $D$ are simultaneously approaching zero. In the context of single-index models,~\cite[Section 4.3]{lu2017phase} observed such a non-monotonicity and multiple phase transitions as the parameter $\phi$ is varied.

\begin{rem}\label{rem:test-vs-train}
These results are stated in terms of the effective dynamics and effective spectrum which are limit objects. It is natural to ask about the corresponding interpretation at finite-$d$. In this case, one must be precise about whether
these matrices are generated using an independent (e.g., test) data set, or the same data set as used for training. Since we are in the online setting it is not clear which scenario is more relevant than the other. For the first scenario, the results hold as stated on the SGD trajectory itself (with high probability). For the second scenario, however, this is a very interesting mathematical challenge as the SGD trajectory on timescales of order $d$ may have non-trivial correlations with the empirical matrices. That said, we find surprisingly good agreement numerically in this scenario as well, by comparing Figure~\ref{fig:Hessian-spectrum-evolution-train} to Figure~\ref{fig:Hessian-spectrum-evolution-test} above, as well comparing Figure~\ref{fig:G-spectrum-evolution-test} and Figure~\ref{fig:G-spectrum-evolution-train} in Section~\ref{sec:supplementary-figures}. 
\end{rem}

\subsection{Effective spectral theory for random matrices of the form~\eqref{eq:A-D-A^t}}\label{subsec:general-random-matrix-result}
Theorem~\ref{mainthm:logistic-regression}--\ref{mainthm:logistic-regression-outliers}, from which all the conclusions in Section~\ref{subsec:kGMM-main-results} follow, is a special case of a general random matrix result which describes the high-dimensional spectral theory of matrices of the form of~\eqref{eq:A-D-A^t} precisely in terms of the summary statistic matrix. 

We consider general matrices of the form \eqref{eq:A-D-A^t} where the $d\times n$ matrix $A= [Y_1,...,Y_n]$ has columns that are given by $n$  data points $(Y_\ell)$, where  $Y_\ell\sim \mathcal P_Y$ are i.i.d.\ and $\cP_Y$ is a $k$-component Gaussian mixture~\eqref{eq:data-distribution}.  This matrix is coupled to a diagonal matrix $D=\diag(\Dll)$ where $D_{\ell \ell}$ only depends on $Y_\ell$ through its inner products with a fixed $O(1)$ collection of vectors, $\bfx =(x_1,...,x_\cC)$ and the class means $(\mu_1,...,\mu_k)$, whose indices $[k]$ are called the \emph{hidden labels}. 
 Observe that the cases where $D$ is the identity, and more generally, where $D$ is positive semi-definite but independent of $A$, have a long history in the random matrix literature on \emph{empirical covariance matrices}. See e.g., the book of~\cite{BaiSilversteinBook}, in particular the chapter on sample covariance matrices, for more. The key distinction for us is that $D$ is coupled to $A$; also it is real, but we allow it to have negative entries.  
 
 As described in the introduction, matrices of this form arise naturally as empirical Hessian and Gradient matrices of loss landscapes when the loss only depends on the data through its labels and inner products with some $O(1)$ many vectors.  This class of problems is quite broad and extends beyond the setting of classification tasks. 
We give examples of other well-known high-dimensional statistical tasks in Section~\ref{subsec:further-applications} to which this result applies after stating the general result. Let us spell out our exact assumptions on the matrix $D$ and how it is coupled to the data matrix $A$.

\begin{assumption}\label{assumption:meas}
The matrix $D$ is diagonal, and $D_{\ell\ell}$ is a measurable function of the hidden label $b(\ell)$ and the $q= \cC+k$ many inner products, 
$((\langle Y_\ell ,x_a\rangle)_{a\in [\cC]}, (\langle Y_\ell , \mu_j\rangle)_{j\in [k]})$. 
\end{assumption}

\begin{assumption}\label{assumption:D}
The law of $D_{\ell\ell}$ is uniformly (in $d$) compactly supported, and there exists a sequence $a_\varepsilon$ going to $0$ as $\varepsilon \downarrow 0$ such that $\mathbb P(|D_{\ell\ell}| <\varepsilon ) \le a_\varepsilon$ for all $d$ and all $\varepsilon>0$.
\end{assumption}
We observe here that the latter assumption ensures $D$ is  a.s.\ invertible for all $d$, and doesn't accumulate mass at $0$ in the $d\to\infty$ limit.

We may now define the matrix $\mathbf{G}= \mathbf{G}(\mathbf{x})$ as in~\eqref{eq:define_G}.
Note here, that we can work in a probability space in which there are additional i.i.d.\ vectors $(\vec g_\ell)$ which are standard Gaussians in $\R^q$ such that, conditionally on the hidden labels $b(\ell)$, we have for each $\ell\leq n$, 
 \begin{equation}\label{eq:Y-g-coupling}
 ((\langle Y_\ell ,x_a\rangle)_{a\in [\cC]}, (\langle Y_\ell , \mu_j\rangle)_{j\in [k]}) =  \mathbf{G}_{[q]b(\ell)}(\mathbf x)+ \sqrt{\bfG}_{[q]b(\ell )}(\bfx)\lambda^{-1/2}\vec{g}_\ell.
 \end{equation}
In particular, Assumption~\ref{assumption:meas} implies existence of a family of functions $(f_j)_{j\in [k]}$ such that 
\begin{align}\label{eq:f-in-terms-of-G}
\Dll = f_{b(\ell)}(\lambda^{-1/2}\vec g;\bfG).
\end{align}
With this shorthand, we define the matrix $F(z,\bfG)$ as in~\eqref{eq:F-matrix} and
 define the effective bulk, $\nu_\bfG$, and effective outliers, $\mathcal Z_{\bfG}$, as in Definitions~\ref{def:effective-bulk}--\ref{def:effective-outliers} with respect to these quantities. 
As above let $\hat\nu^{(d)}$ denote the empirical spectral measure of $ADA^\top/n$. We then have the following two theorems which are the analogues of Theorems~\ref{mainthm:logistic-regression}--\ref{mainthm:logistic-regression-outliers}. 

\begin{thm}[Bulk]\label{thm:main-A-D-A^T}
Suppose that Assumption~\ref{assumption:meas} holds. Fix $\lambda,\phi>0$. For $n,d$ sufficiently large with $n/d=\phi$, we have that for any $\mathbf x$ with summary statistics matrix $\bfG$, with probability $1-o(1)$,
\[
d_{\text{BL}}(\hat\nu^{(d)},\nu_\bfG) =o(1)\,.
\]
\end{thm}

Let us now turn to the outliers. As above let $\sigma^{(d)}$ denote the full spectrum of the matrix $ADA^\top /n$. and let $L$ be as in \eqref{e:L-defining-property}. We then have the following.

\begin{thm}[Outliers]\label{thm:main-A-D-A^T-outliers}
Suppose that Assumptions~\ref{assumption:meas}-\ref{assumption:D} hold. Fix $\lambda,\phi>0$. 
For any $\bfG$ and any $[a,b]\subseteq \supp(\nu_\bfG)^c$ with $\{a,b\}\cap \mathcal Z_{\bfG} =\emptyset $, we have that for $n,d$ sufficiently large with $n/d=\phi$ and any $\mathbf x$ with $\bfG(x)=\bfG$,
\begin{equation}\label{eq:outlier-count-general}
   |\sigma^{(d)} \cap [a,b]| = |\mathcal Z_{\bfG}\cap[a,b]|\,, 
\end{equation}
with probability $1-o(1)$. 
Moreover, for every $\delta>0$ small enough and every $L$ satisfying \eqref{e:L-defining-property}, with probability $1-o(1)$, for all $z\in\cZ_{\bfG}$, if $z^{(d)}\in\sigma^{(d)}\cap B_\delta(z)$ then any corresponding eigenvector $v^{(d)}$ is such that $u^{(d)}=L^\top v^{(d)}$  is a $o(1)$-approximate solution to the effective eigenvector equation~\eqref{eq:effective-evect}. If further $z\neq0$, then it also is an $o(1)$-approximate solution \eqref{eq:effective-evect-proj}.  
\end{thm}

    We emphasize once again that the defining equations~\eqref{e:defmz-with-rho}--\eqref{eq:outlier-equation-general} of the effective bulk and outliers are dimension independent, only depending on the relevant vectors $\mathbf{x}$ and the mean vectors $\mu_1,...,\mu_k$ through their  summary statistic matrix $\mathbf{G} \in \mathbb R^{q\times q}$ for $q= \cC+k = O(1)$.

    Though in our examples, we mostly discuss  the behavior of right-outliers, we note that left outliers can have interesting implications for optimization. For example, in non-linear least squares \cite{bartlett2021deep,oymak2019overparameterized,chizat2019lazy}, it is well known that in certain regimes the smallest eigenvalue governs the speed of convergence. Since Theorem~\ref{thm:main-A-D-A^T-outliers} can rule out outliers in intervals, one could pursue analogous types of implications in future work.

     We finally introduce a stronger assumption on the entries of $D_{\ell \ell}$ that allows us to pass to limits as $d\to\infty$ in the cases where the summary statistic matrices are not constant, but only converging to some limit $\mathbf{G}$. This is needed to study evolutions along the SGD where there are necessarily $O(d^{-1/2})$ corrections to the effective dynamics. Recall that $\rho_{\mathbf{G}}$ denotes the law of $ f_J(\lambda^{-1/2}\vec{g};\mathbf{G})$ where $J\sim \text{Cat}(p_i)_{i=1}^k$ for $f_j$ from~\eqref{eq:f-in-terms-of-G}, and let $\hat \rho_{\mathbf{G}^{(d)}}$ be the empirical measure of $(D_{\ell\ell})_{\ell=1}^d$. 
    \begin{assumption}\label{assumption:strong-convergence}
        As $\bfGd\to \bfG$, it also holds that $\hat \rho_{\bfGd}\to \rho_{\bfG}$ weakly almost surely, and
        \begin{align*}
            d_H(\supp(\hat \rho_{\bfGd}) , \supp(\rho_\bfG)) \to 0\,,
        \end{align*}
       almost surely, where $d_H$ denotes the Hausdorff distance. 
    \end{assumption}

\noindent In the free probability literature, this is sometimes referred to as strong convergence in distribution of $D\to \rho_{\bfG}$. It can be checked under mild regularity assumptions: see  Lemma~\ref{lem:bounded-differentiable-implies-strong-convergence}, which shows that if the function $h$ from Assumption~\ref{assumption:meas} for which $D_{\ell\ell} = h(\mathbf{x}^\top Y_{\ell})$ is uniformly continuous, and $\|\bfGd - \bfG\| = o(1/\log d)$ then Assumption~\ref{assumption:strong-convergence} holds. 

Under this assumption, we then have the following. 
    \begin{thm}\label{thm:main-A-D-A^T-limit}
        In the setting of Theorems~\ref{thm:main-A-D-A^T}--\ref{thm:main-A-D-A^T-outliers}, if furthermore $\mathbf{x}^{(d)}$ is a sequence of parameters with summary statistic matrices $\mathbf{G}^{(d)} \to \mathbf{G}$ as $n,d\to\infty$ which satisfies Assumption~\ref{assumption:strong-convergence} then 
        \begin{itemize}
            \item We have that $\hat \nu^{(d)} \to \nu_{\mathbf{G}}$ weakly a.s.\
            \item Eq.~\eqref{eq:outlier-count-general} holds eventually a.s., and every limit point of $u^{(d)}$ is a solution of the effective eigenvector equation~\eqref{eq:effective-evect}, and also a solution to \eqref{eq:effective-evect-proj}, provided the corresponding $z\neq0$.
        \end{itemize}
    \end{thm}

The first bullet point in Theorem~\ref{thm:main-A-D-A^T-limit} only requires the first part of Assumption~\ref{assumption:strong-convergence} on weak convergence. 
    Recall that these limiting consequences for convergent sequences of summary statistic matrices were already present in the logistic regression case's main theorems. That is because the specific forms of the $f_j$ from~\eqref{eq:f_a-kGMM-case} are checked to satisfy Assumption~\ref{assumption:strong-convergence}.

\subsection{Further applications}\label{subsec:further-applications}
We now describe some example statistical tasks (data distributions together with loss functions) whose Hessian and Gradient matrices also fall under the purview of Theorems~\ref{thm:main-A-D-A^T}--\ref{thm:main-A-D-A^T-limit}. Notably, these include first-layer Hessians in bounded width multi-layer neural networks for classifying Gaussian mixture data, and for learning single and multi-index models, under mild conditions on the activation and link functions. In all of these cases, under the SGD training, the same set of summary statistics solve asymptotically autonomous ODEs in the high-dimensional limit, so in principle one can track the effective bulk and outliers along the effective dynamics trajectory $\mathbf{G}_t$ as we did for the logistic regression task in Section~\ref{subsec:kGMM-main-results}.

\subsubsection{Multi-layer GMM classification}
When the class labels corresponding to the $k$-GMM distribution are such that the class means are not linearly separable, optimizing the (empirical or population) loss function~\eqref{eq:loss-function-1-layer} will not result in a good classifier. A well-studied example of this is the XOR data distribution~\cite{minsky1969introduction} where the four means are $\pm \mu,\pm \nu$, and one class corresponds to hidden label $\pm \mu$, while the other class corresponds to hidden label $\pm \nu$. 
In general, consider the case that the means are $\mu_1,...,\mu_k$ (i.e., the data distribution is still $\mathcal P_Y$ of~\eqref{eq:data-distribution} and the labels are assigned to subsets of $[k]$ which are the classes, $I_1,...,I_\cC$ partitioning $[k]$ the distribution of $Y$ is the same as above, and the label $y$ is the class among $I_1,...,I_\cC$ that the mean belongs to. 

In such tasks where the means from different classes are not linearly separable, one needs a multi-layer neural network to even express a good classifier.  Let us consider a simple two-layer architecture, with $K= O(1)$ hidden neurons, activation function $g$ on the hidden neurons, and sigmoid activation at the output layer. That is, the loss function takes the form 
\begin{align}\label{eq:loss-function-2-layer}
    L(\mathbf{x},(y,Y)) = -\sum_{c\in [\cC]} y_c v^c \cdot g(\mathbf{W}^c Y) + \log \sum_{c\in [\cC]} \exp(v^c \cdot g(\mathbf{W}^c Y))\,.
\end{align}
For each $c\in [\cC]$ parameter $\mathbf{x}^c = (\mathbf{W}^c,v^c)$ generates the $c$'th one-vs-all classifier, where $\mathbf{W}^c\in \mathbb R^{K\times d}$ is the first layer weights $(\mathbf{W}_1^c,...,\mathbf{W}_K^c)$, and $v^c\in \mathbb R^K$ is the second layer weights in the $c$'th one-vs-all classifier being trained. The activation $g$ is applied entry-wise.

Let us introduce a non-degeneracy assumption on how a function composes with a Gaussian random variable, that will come up throughout our following examples. 
\begin{assumption}\label{assumption:Gaussian-non-degeneracy}
    We say a function $h: \mathbb R^m\to \mathbb R$ satisfies the non-degeneracy assumption if, whenever $Z \sim \mathcal N(0,I_m)$, the random variable $h(Z)$ has bounded density (w.r.t.\ Lebesgue). 
\end{assumption}

\begin{cor}\label{cor:multi-layer-GMM}
    Suppose the network activation function $g \in C^2_b(\mathbb R)$ is such that any $h\in \{g, g', g''\}$ satisfies the non-degeneracy condition of Assumption~\ref{assumption:Gaussian-non-degeneracy}. 
        
    If the neural network weights $\mathbf{x} = (\mathbf{W},v)$ have $v^\alpha, W_i^\alpha \ne 0$ and have $(q=K+k)$-by-$q$ summary statistic matrix $\mathbf{G} = (\mathbf{W},\bmu)^{\otimes 2}$ then, Assumptions~\ref{assumption:meas}--\ref{assumption:D} apply. Therefore Theorems~\ref{thm:main-A-D-A^T}--\ref{thm:main-A-D-A^T-outliers} are applicable to the $(i,\alpha)(i, \alpha)$-blocks of both the empirical Hessian and Gradient matrix of the multi-layer $k$-GMM classification task.
    Furthermore, if $g\in C^3_b$ and $\mathbf{G}^{(d)}$ is the sequence of summary statistic values for $\mathbf{x}$, if $\|\bfG^{(d)} - \bfG\| = o(1/\log d)$, then Assumption~\ref{assumption:strong-convergence} applies and Theorem~\ref{thm:main-A-D-A^T-limit} is applicable. 
\end{cor}

\begin{figure}
\includegraphics[width = .32\textwidth]{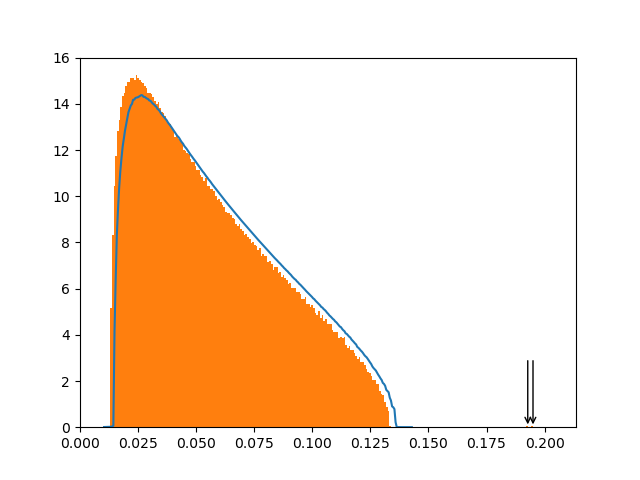}
\includegraphics[width = .32\textwidth]{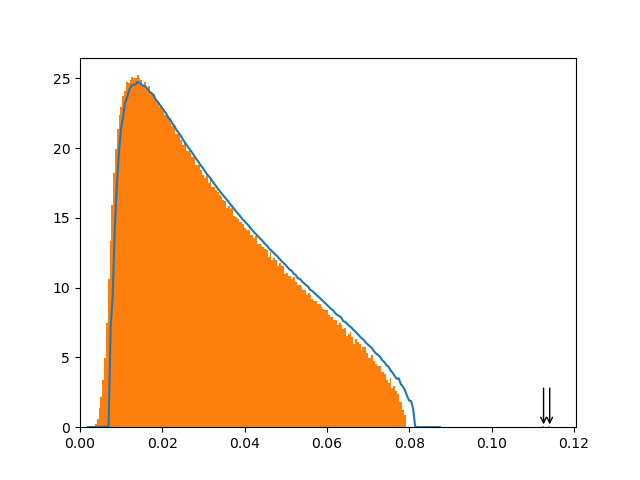}
\includegraphics[width = .32\textwidth]{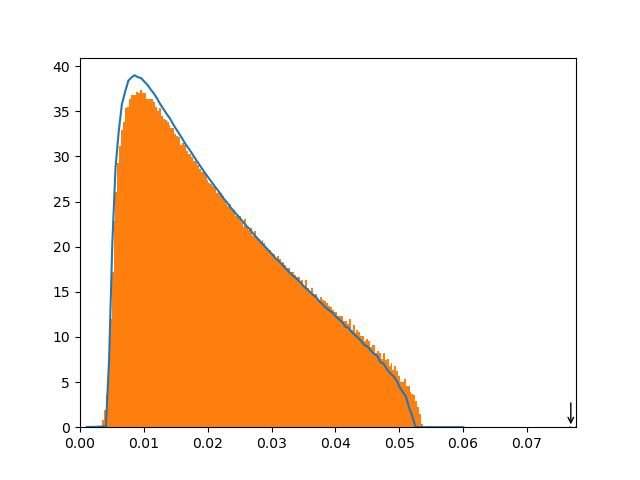}
    \caption{The evolution of the empirical (first-layer) Hessian spectrum (orange histogram) for the XOR GMM task, together with the predicted bulk spectrum (blue curve) and outlier locations (arrows) after $0$, $n/3$ and $2n/3$ steps of online SGD. $d=20,000$, $k=3$, and $\lambda = \phi =4$.}\label{fig:Hessian-spectrum-evolution-XOR}
\end{figure}

In the specific case of the XOR data distribution with four means $\pm \mu, \pm \nu$ and the high-dimensional training dynamics in this XOR case were studied in detail in~\cite{refinetti2021classifying,BGJ22}. In particular, the latter showed that the summary statistics $(v^b)_{b\in [\cC]}, \mathbf{G}^{\alpha}$ solve a limiting autonomous effective dynamics that has multiple attractors. Figure~\ref{fig:Hessian-spectrum-evolution-XOR} shows snapshots of the first layer empirical Hessian in the XOR task over the course of training, together with the predicted effective bulk at the induced summary statistic values. 
In that task, there is uniformly positive probability (with respect to a Gaussian initialization $\mathbf{x} \sim \mathcal N(0,I_d/d)$) of success/failure at classification. It was shown in~\cite{BGJH23} that at sufficiently large SNR $\lambda \gg 1$, the number of outlier eigenvalues after a burn-in period of training serves as a distinguishing statistic for success/failure at the classification task. We believe it would be of much interest to understand the sharp SNR thresholds and points at which effective outliers emerge from the effective bulk along the effective dynamics path for this problem. 

\begin{rem}\label{rem:regularity-on-activation}
Let us say about Assumption~\ref{assumption:Gaussian-non-degeneracy} for $h\in \{g,g',g''\}$ that since these are functions from $\mathbb R \to \mathbb R$, it is well-known that strict monotonicity, or piecewise strict monotonicity are sufficient to ensure the non-degeneracy condition. Thus, the sigmoid/softmax and tanh activations satisfy it, though  ReLU for instance will not due to the zero derivative on the negative half line. To extend the result to settings like ReLU activation, the zero-entries in $D$ would have to be handled separately since the invertibility of $D$ is used at several points in the proofs. 
\end{rem}

\medskip
In the fixed-width $K=O(1)$ regime, the second layer matrices are  $K\times K$ in the $n,d\to\infty$ limit, so there are not good notions of bulk and outliers. The following lemma 
says that still, the asymptotic dependence of this $K\times K$ matrix on the parameters is through their summary statistics.  

\begin{lem}\label{lem:second-layer-only-summary-statistic}
    In the setting of Corollary~\ref{cor:multi-layer-GMM}, if we consider the $v^\alpha v^\alpha$-blocks of both the empirical Hessian and Gradient matrix, their $n,d\to\infty$ limiting law only depends on the second layer weights $(v^\alpha)_{\alpha}$, and the limiting summary statistic matrix $\mathbf{G} = \lim_{d\to\infty} \mathbf{G}^{(d)}$. 
\end{lem}

These same summary statistics $(v^\alpha)_{\alpha},\mathbf{G}$ were shown in~\cite{BGJ22} to evolve autonomously under an effective dynamics as the $n,d\to\infty$ limit of SGD. Thus, the limiting $K\times K$ second layer matrix can also be tracked along the effective dynamics. In the $\lambda,\alpha$ large regimes of the XOR problem,~\cite{BGJ22,BGJH23} showed connections between the rank (deficiency) of this matrix and how many of the hidden means have been learned. 

This also raises the question of what can be said about the limiting spectral theory of the output layer in the \emph{narrow overparameterized} regime, where the width $K\to\infty$ after $n,d\to\infty$. We leave this to future investigation, noting that this limiting matrix may exhibit serious degeneracies.

\subsubsection{Parametric regression for the multi-index model}

A quite different set of extensively studied statistical learning problems come from single-index models (sometimes generalized linear models)~\cite{hastie2009elements,bishop}. This family include phase retrieval as a central example, studied in high dimensions, see e.g., in~\cite{CCFM,CandesLS14,LeSuNIPS16,lu2017phase,TanVershyninKaczmarz,ALBZZ19,JeongGunturk}. More broadly, the family of multi-index models (multi-variate generalizations of single index models) have recently seen much attention in the context of success of SGD training, with a particular focus on what types of functions can be learned with finite v.s.\ diverging sample complexity. For a sampling of this literature for learning with single and multi-index models, see e.g.,~\cite{pmlr-v75-dudeja18a,LALspectral18,BGJ21,NEURIPS2023_02763667,NEURIPS2023_e21955c9,abbe2023sgd,simsek2024learninggaussianmultiindexmodels,bietti2023learninggaussianmultiindexmodels,Paquettes-high-dimensional-notes}.

In particular, several of the aforementioned works studied the spectral theory of matrices like those in~\eqref{eq:A-D-A^t} in the context of optimal spectral estimators in the single-index setting~\cite{mondelli2018fundamental,lu2017phase,zhang2022precise,LALspectral18}. Concurrent to this paper,~\cite{defilippis2025optimalspectraltransitionshighdimensional,mondelli-spectral-estimators} investigated the multi-index settings. The general framework of our results closely overlaps with these latter papers. We allow for non-centered data matrices (in particular with random class means) and are especially interested in the changes to the spectrum as one varies the parameter values $\mathbf{X}$ through space. They were more interested in the optimal processing function giving their diagonal matrix $D$, with no spatial dependence.

The multi-index model is the data model where one gets access to observations of $k$ ground truth vectors through a non-linear \emph{link function} of inner products of the ground truth with certain \emph{sensing vectors} or \emph{covariates}. To be a bit more precise, the ground truth is 
$\Theta_* \in \mathbb R^{k\times d}$, the sensing vector $Y\in \mathbb R^{d}$ is typically standard Gaussian, 
the link is $g: \mathbb R^k \to \mathbb R$, and the observation is 
\begin{align}\label{eq:multi-index-regression-data}
y = g(\Theta_* Y) = g(\Theta_*^1 \cdot Y,...,\Theta_*^k\cdot Y)\,. 
\end{align}
In the corresponding regression problem, one tries to optimize the following $L^2$ loss function over the parameter $\mathbf{X} \in \mathbb R^{k\times d}$:
\begin{align}\label{eq:multi-index-regression-loss}
    L(\mathbf{X}; (y,Y))= |g(\mathbf{X}Y) - y|^2\,.
\end{align}
Evidently, if $\mathbf{X} = \Theta_*$ then the loss is zero. 

\begin{cor}\label{cor:multi-index-regression}
    Suppose the link function $g$ is in $C^2_b(\mathbb R^k)$ and that $h\in \{g,\partial_\alpha g, \partial_{\alpha \alpha } g\}$ satisfy the non-degeneracy condition of Assumption~\ref{assumption:Gaussian-non-degeneracy}.  
    If the ground truth and parameters have $2k\times 2k$ summary statistic matrix $\mathbf{G} = (\mathbf{X},\Theta_*)^{\otimes 2}$ that is full rank, then, Assumptions~\ref{assumption:meas}--\ref{assumption:D} apply, and therefore Theorems~\ref{thm:main-A-D-A^T}--\ref{thm:main-A-D-A^T-outliers} is applicable to the $(\alpha, \alpha)$-blocks of both the empirical Hessian and Gradient matrix of the parametric regression problem described above.     Furthermore, if $g\in C^3_b$ and $\|\bfG^{(d)} - \bfG\| = o(1/\log d)$, then Assumption~\ref{assumption:strong-convergence} applies and Theorem~\ref{thm:main-A-D-A^T-limit}. 
\end{cor}

\begin{rem}\label{rem:boundedness-of-link-or-activation}
    The boundedness requirement on the link function may seem strong (in particular it rules out the example of phase retrieval, where $g (x)  = |x|$ or more smoothly $g(x) = |x|^2$), but without it discussion of outlier eigenvalues may not be especially pertinent at high enough $d$, since $\frac{1}{n}ADA^\top$ with unbounded entries in $D$ can have an effective bulk whose support is unbounded. Indeed, if we drop the uniform boundedness from Assumption~\ref{assumption:D}, our effective bulk proof and result would still apply, but the proof of the outlier equations (which assume things like being to the right of the support of the bulk) would cease to apply due to this unbounded-ness. In practice, it is common to truncate the link function at some large threshold. 
\end{rem}

\subsubsection{Multi-layer learning for multi-index models}

An extension of the multi-index model task is in the situation where even the link function $g$ is unknown to the statistician, and is to be learned using a multi-layer neural network. We follow the formulation found for example in the paper~\cite{bietti2023learninggaussianmultiindexmodels}, which studied the effective dynamics for this problem.  
In this case, the goal is to learn the function $h_*:Y \mapsto y = g(\Theta_* Y)$ by  a two-layer neural network with bounded width in its hidden layer. Namely, the loss function will now be given by 
\begin{align*}
    L(\mathbf{x};(h_*(Y),Y)) = \big|h_{\NN}(Y) - h_*(Y)\big|^2 \qquad\text{where} \qquad h_{\NN}(Y) = v\cdot \sigma(\mathbf{W}Y)
\end{align*}
where the parameters $\mathbf{x} = (\mathbf{W},v)$ are the first and second layer weights of a fixed width neural network, $\mathbf{W} \in \mathbb R^{K\times d}$ and $v\in \mathbb R^K$, and where $\sigma$ is the \emph{activation function} applied entrywise. 

\begin{cor}\label{cor:multi-layer-multi-index}
    Suppose the link function $g$ is in $C_b(\mathbb R^k)$, and the activation function $\sigma \in C^2_b(\mathbb R)$ is such that any $h\in \{\sigma, \sigma', \sigma''\}$ satisfies the non-degeneracy condition of Assumption~\ref{assumption:Gaussian-non-degeneracy}. 
        
    If the neural network weights $\mathbf{x} = (\mathbf{W},v)$ and ground truth vectors $\Theta_*$ are such that $v_c\ne 0$,  
    \begin{align}\label{eq:not-ground-truth-requirement}
        \mathbb P(|h_{\NN}(Y) - h_*(Y)|<\epsilon) = o_\epsilon(1)\,,
    \end{align}
    and it has full rank $(q=K+k)$-by-$q$ summary statistic matrix $\mathbf{G} = (\mathbf{W},\Theta_*)^{\otimes 2}$, then, Assumptions~\ref{assumption:meas}--\ref{assumption:D} apply. Therefore Theorems~\ref{thm:main-A-D-A^T}--\ref{thm:main-A-D-A^T-outliers} is applicable to both the $(\alpha, \alpha)$-blocks of the empirical Hessian and Gradient matrices of the multi-layer learning of multi-index model described above. Furthermore, if $\sigma \in C_b^3$ and $\|\bfG^{(d)} - \bfG\| = o(1/\log d)$, then Assumption~\ref{assumption:strong-convergence} applies and Theorem~\ref{thm:main-A-D-A^T-limit}. 
\end{cor}

Let us briefly comment on~\eqref{eq:not-ground-truth-requirement} which is needed to ensure that the diagonal matrix in the Gradient matrix is invertible. E.g., if $h_*$ is perfectly learned, $D$ would be the zero matrix. Therefore, similar to the earlier comment on the ReLU activation, if Assumption~\ref{assumption:D} does not hold and $D$ can have zero entries, these entries need to be handled separately, and then once the matrix is full rank our methods can apply.

\subsection*{Acknowledgments}
The authors are grateful to the anonymous referees for their careful reading of the manuscript and helpful suggestions. 
The research of G.B.A.\ is supported in part by NSF grant 2134216. The research of R.G.\ is supported in part by NSF CAREER grant 2440509 and NSF DMS grant 2246780.  
The research of J.H.\ is supported by NSF grants DMS-2331096 and DMS-2337795, and the Sloan Research Award. 
A.J.\ acknowledges the support of the Natural Sciences and Engineering Research Council of Canada (NSERC), the Canada Research Chairs programme, and the Ontario Research Fund. Cette recherche a \'et\'e enterprise gr\^ace, en partie, au 
soutien financier du Conseil de Recherches en Sciences Naturelles et en G\'enie du Canada (CRSNG),  [RGPIN-2020-04597, DGECR-2020-00199], et du Programme des chaires de recherche du Canada.

\section{Spectral theory for self-coupled empirical matrices} 
In this section we establish our general random matrix theory result which will imply Theorems~\ref{thm:main-A-D-A^T}--\ref{thm:main-A-D-A^T-limit}. Rather than working directly with $(1/n)ADA^\top$ from~\eqref{eq:A-D-A^t}, it helps to work instead with the rescaling $(\lambda/d)ADA^\top$.

Throughout this section, we consider the matrix model $(\lambda/d)ADA^\top$ and denote its empirical spectral measure by $\hat\nu^{(d)}$. Furthermore, given a sequence $[x_1^{(d)},...,x_\cC^{(d)},\mu_1^{(d)},...,\mu_k^{(d)}]$ with summary statistic matrix $\bfG^{(d)}$, and square root $\sqrt{\bfGd}$, let $L^{(d)}$ denote the $d\times q$ matrix with orthonormal columns satisfying 
\begin{equation}\label{e:L-defining-property-2}
L^{(d)\top} [x_1^{(d)},...,x_\cC^{(d)},\mu_1^{(d)},...,\mu_k^{(d)}] = \sqrt{\bfGd}.
\end{equation}
That such a matrix exists can be shown via the compact singular value decomposition. See \eqref{e:decomp} and the surrounding discussion  below.
For concreteness, let us also define the analogues of the quantities defining the effective bulk and outliers. To this end, let $\nu=\nu_\bfG$ be the probability measure whose Stieltjes transform $\tS(z)$ solves 
    \begin{align}\label{e:ADA_eig}
        \frac{1}{ \tS(z)}+ z= \phi\sum_{a=1}^{k} p_a \mathbb E\Big[ \frac{f_a(\lambda^{-1/2}\vec{g};\mathbf G)}{1 +  \tS(z) f_a(\lambda^{-1/2}\vec{g};\mathbf G) }\Big].
    \end{align}
where $f_a$ is as in \eqref{eq:f-in-terms-of-G}. (In the following, we drop the subscript $\bfG$ when it is clear from context.) Define the matrix $\wt F$ by
\begin{equation}\label{eq:general-f-def}
\wt F(z, \bfG) =\lambda \phi\sum_{a=1}^k p_a \mathbb E\left[\frac{f_{a}(\lambda^{-1/2}\vec g, \mathbf G)}{1+ \tS(z)f_a(\lambda^{-1/2}\vec g, \mathbf G)}(\frac{\vec{g}}{\sqrt{\lambda}}+(\sqrt{\bfG}_{\cdot,a}))(\frac{\vec{g}}{\sqrt{\lambda}}+(\sqrt{\bfG}_{\cdot,a}))^\top\right],
\end{equation}
with these definitions we have the notion of effective bulk, effective outlier, and effective eigenvector equations for $(\lambda/d) A D A^\top$. In particular let $\cZ_{\bfG}$ be the collection of effective outliers.

\begin{thm} 
    \label{thm:deterministic-equivalent}
Consider the matrix model $(\lambda/d)ADA^\top$. Suppose that Assumptions~\ref{assumption:meas}-\ref{assumption:strong-convergence} holds. Fix $\lambda,\phi>0$ and suppose the summary statistic matrices of $(\bfx,\bmu)$ form the sequence $\bfGd\to\bfG$. Then as $n,d\to\infty$ with $n/d=\phi$, the empirical spectral measure satisfies
\begin{equation}\label{eq:main-prop-bulk}
\hat\nu^{(d)}\to\nu_\bfG,
\end{equation}
weakly a.s., where $\nu_\bfG$ is the measure whose Stieltjes transform solves the effective bulk equation \eqref{e:ADA_eig}.
Furthermore, for any $[a,b]\in \supp(\nu_\bfG)^c$ such that $\{a,b\}\cap \cZ_{\bfG}=\emptyset$, we have that,
\begin{equation}\label{eq:main-prop-outliers}
|\sigma^{(d)}\cap[a,b]| = |\cZ_{\bfG}\cap [a,b]|    
\end{equation}
eventually a.s. Moreover, if $z\in\cZ_{\bfG}$ and $z^{(d)}$ is a sequence of eigenvalues with $z^{(d)}\to z$, then for any $L$ satisfying \eqref{e:L-defining-property} and $u^{(d)}=L^\top v^{(d)}$, we have that every limit of point of  $u^{(d)}$ is a solution of the effective eigenvector equations
\begin{align}\label{eq:main-prop-eigenpairs-1}
(z-\wt F(z)) u &= 0\\
\norm{u}^2 - \langle u , \partial_z\wt F(z) u\rangle &= 1\label{eq:main-prop-eigenpairs-2}
\end{align}
almost surely.
\end{thm}

Theorems~\ref{thm:main-A-D-A^T}--\ref{thm:main-A-D-A^T-limit} follows immediately from Theorem \ref{thm:deterministic-equivalent} by standard rescaling identities for Stieltjes transforms. For the reader's convenience this is outlined in the following subsection.

\subsection{Proofs of Theorems~\ref{thm:main-A-D-A^T}--\ref{thm:main-A-D-A^T-limit}}
We begin by deducing Theorem~\ref{thm:main-A-D-A^T-limit} from Theorem~\ref{thm:deterministic-equivalent} as both assume the same Assumptions~\ref{assumption:meas}--\ref{assumption:strong-convergence}. To see the equivalence of the two theorems, we make the following observation regarding rescaling the effective bulk, outlier, and eigenvector equations. 
Denote the Stieltjes transform of the limiting empirical eigenvalues of $ADA^\top/n$ by $S(z)$ and that of $(\lambda/ d)ADA^\top = \lambda \phi (1/n)ADA^\top$ by $ \wt S(z)$. Then by definition~\eqref{eq:Stieltjes-transform},  $\wt S(z)= (1/\lambda \phi)S(z/\lambda\phi)$. 
If we let  $w= z/\lambda \phi$, we can rewrite \eqref{e:ADA_eig} as 
    \begin{align}
        \frac{1}{ (1/\lambda \phi) S(w)}+ \lambda \phi w= \phi\sum_{a=1}^{k} p_a \mathbb E\Big[ \frac{f_a(\lambda^{-1/2}\vec{g};\mathbf G)}{1 +  (1/\lambda \phi) S(w) f_a(\lambda^{-1/2}\vec{g};\mathbf G) }\Big]
    \end{align}
By \eqref{eq:main-prop-eigenpairs-1} ($zI_q-\widetilde F(z)$ is not full rank, so its determinant is zero), the effective outliers of $ADA^\top/n$ satisfy 
    \begin{align}
    \det ( wI_q - \wt F(\lambda \phi w)/\lambda \phi) =0 \,.
\end{align}
which can then be rearranged as  
\begin{align*}
\frac{\wt F(\lambda \phi w)}{\lambda \phi}
&=\sum_{a=1}^k p_a \mathbb E\left[\frac{f_{a}(\lambda^{-1/2}\vec g, \mathbf G)}{1+ \wt S(\lambda \phi w)f_a(\lambda^{-1/2}\vec g, \mathbf G)}(\frac{\vec{g}}{\sqrt{\lambda}}+\sqrt{\bfG}_{\cdot,a})(\frac{\vec{g}}{\sqrt{\lambda}}+\sqrt{\bfG}_{\cdot,a})^\top\right]\\
&= \lambda\phi\sum_{a=1}^k p_a \mathbb E\left[\frac{f_a(\lambda^{-1/2}\vec g; \mathbf{G})}{\lambda \phi+S(w)f_a(\lambda^{-1/2}\vec g;\mathbf{G})} (\frac{\vec{g}}{\sqrt{\lambda}}+\sqrt{\bfG}_{\cdot,a})(\frac{\vec{g}}{\sqrt{\lambda}}+\sqrt{\bfG}_{\cdot,a})^\top\right]=F(w) 
\end{align*}
to give~\eqref{eq:F-matrix}, and therefore~\eqref{eq:effective-evect}. The equivalence of~\eqref{eq:main-prop-eigenpairs-2} and ~\eqref{eq:effective-evect-proj} follows from the above by $\partial_w F(w) =  \tilde F'(\lambda \phi w) = \partial_z \tilde F(z)$. 

To deduce Theorems~\ref{thm:main-A-D-A^T}--\ref{thm:main-A-D-A^T-outliers}, notice that if $\bfGd \equiv \bfG$ for all $d$, then $D_{\ell\ell}$ are i.i.d.\ over $\ell$ and $d$ by Assumption~\ref{assumption:meas}. Therefore, the almost sure weak convergence follows from the strong law of large numbers for the empirical spectral measure. The strong convergence for the empirical measure for i.i.d.\  samples is similarly straightforward. 
\qed

\subsection{Estimates for the Stieltjes transform of $D$ and $\wt S$}
In this section, we show that the Stieltjes transform of the empirical spectral measure of $D$ and its limit satisfy the $\tau$-regularity condition and thus the estimates from the preceding section. To this end, let $\hat \rho^{(d)}=\frac{1}{n}\sum \delta_{D_{\ell\ell}}$, and $\rho^{(d)}$ the law of the entries $D_{\ell \ell}$.
Let $\tSd$ denote Stieltjes transform of 
\begin{equation}\label{eq:}
    \nu_d = \hat\rho^{(d)}\boxtimes\nu_{\MP}.
\end{equation} 
Observe that this satisfies the self-consistent equation, 
\begin{align}\label{eq:defmz-sec2}
\frac{1}{\widetilde S^{(d)}(z)}=-z+\phi \int \frac{x}{1+\widetilde S^{(d)}(z) x}\hat\rho^{(d)} (d x),
\end{align}
for $\phi=\frac{n}{d-q} \asymp \frac{n}{d}$. 

Suppose that $\bfG^{(d)}\to\bfG$,  let $\rho$ denote the law of $f_J(\lambda^{-1/2}\vec g,\bfG)$ where $J\sim \text{Cat}((p_i)_{i\in [k]})$ and $\vec g\sim N(0,I_q)$ as above. Observe that by construction, $\tS$ is the Stieltjes transform of the measure $\nu=\nu_\bfG$ with 
\[
\nu_\bfG=\rho\boxtimes\nu_\MP.
\]

Let us now introduce the following definition, which will be important in the coming sections.
\begin{defn}
    We say that $\mu\in \cM_1(\R)$ is \emph{$\tau$-regular} if $\mu([-\tau^{-1},\tau^{-1}]^c)=0$ and 
    \[
    \mu(\abs{x}\leq \tau) \leq 1-\tau.
    \]
\end{defn}
We now observe the following, whose proof is elementary.

\begin{lem}\label{lem:tau-reg}
   If $D$ satisfies Assumption~\ref{assumption:D}, then there is a $0<\tau_0\leq 1$ such that the following holds:
   \begin{enumerate}
       \item $D$ is a.s. invertible.
       \item For each $d\geq 1,0<\tau<\tau_0$,
   $\prob(\hat\rho^{(d)}\text{ is } \tau\text{-regular})\geq 1-C\exp(-c(\tau,\phi)d).$ 
   \item For each $0<\tau<\tau_0$, $\rho$ is $\tau$-regular.
   \end{enumerate}
\end{lem}

\subsection{Regularity of Stieltjes transforms of free multiplicative convolutions with the Marchenko-Pastur law }
We start by developing deterministic estimates on the Stieltjes transform of the free multiplicative convolution of $\tau$-regular measures with the Marchenko-Pastur law. In particular, in light of Lemma~\ref{lem:tau-reg}, all of the following estimates will apply to $\nu$ and $\nu_d$.

Let $\mu\in \cM_1(\R)$ be a probability measure with compact support.  Recall that its free multiplicative convolution with the Marchenko-Pastur law is the probability measure $\mu\boxtimes\nu_{\MP}$ whose Stieltjes transform is the solution to the self-consistent equation
\begin{equation}\label{eq:def-FMC}
    \frac{1}{S_{\mu}(z)}=-z+\phi \int \frac{x}{1+S_\mu(z) x}d\mu(x)\,,
\end{equation}
for some $\phi>0$. (Here and in this section, $\nu_{\MP}$ is the Marchenko-Pastur distribution with parameters $\phi$ and $\lambda=1$.) That these objects are well-defined is classical, going back to the original work of Marchenko-Pastur \cite{marchenko1967distribution}. It is also easy to see that since $\mu,\nu_\MP$ are compactly supported, so is $\mu\boxtimes\nu_\MP$. Note that we are suppressing the dependence of $S_\mu$ on $\phi$ for readability. 

We will assume throughout this section that $\mu$ is $\tau$-regular for some $0<\tau<1$. Let us note that in this case, one immediately obtains (e.g., by using the probabilistic interpretation of the free multiplicative convolution)
\begin{equation}\label{e:support-bound}
    \max\{|z|:z\in \supp (\mu\boxtimes\nu_\MP )\} \lesssim_\phi \tau^{-1}\,.
\end{equation}
In the following, it helps to define
\begin{equation}\label{eq:E-eps}
E_\eps(\mu )=\{ z\in \C: \dist(z,\supp (\mu\boxtimes \nu_{\MP}))\geq \eps\}.
\end{equation}
(Let us pause to note here that in the subsequent, if we write $E_\eps$ we specifically mean $E_\eps=E_\eps(\nu)$ where $\nu$ solves~\eqref{e:ADA_eig}.)
We then have the following. (We use the convention $0!=1$.t)
\begin{lem}\label{l:partial_S_bound} 
    For any $K,\eps, k_0>0$ and $0<\tau<1$, there is an $\eta=\eta(K,\eps,k_0,\tau)$ such that the following holds: For any $\tau$-regular $\mu$, the derivatives of $S_\mu(z)$ are upper and lower bounded as 
   \begin{equation}\label{e:dSzbound}
   k!\Big(\frac{\tau}{1+K}\Big)^{k+1}\lesssim_{\phi} \abs{\partial^{k}_z S_\mu(z)}\lesssim_\phi \frac{k!}{\eps^{k+1}}
   \end{equation}
   uniformly for all $z\in ([-K,K]\times[-\eta,\eta]) \cap E_\eps(\mu)$ for all $0\leq k\leq k_0$. In particular, the upper bound holds on $z\in B(K)\cap E_\eps(\mu)$, where $B(K)= \{z\in \mathbb C: |z|\le K\}$ is the ball of radius $K$ in $\mathbb C$.

\end{lem}
\begin{proof}
Observe that 
    \begin{align*}
        \partial_z^{(k)} S_\mu(z) = \int \frac{ k!}{(t-z)^{k+1}} \mu\boxtimes\nu_\MP (dt).
    \end{align*}
    By our assumption that $\dist(z,\supp( \mu\boxtimes \nu_{\MP}))\geq \epsilon$, the above derivative is bounded as
     \[
        |\partial_z^{(k)}  S_\mu(z)| \leq  \int \frac{ k!}{\epsilon^{k+1}} \mu\boxtimes\nu_\MP(dt)\leq \frac{k!}{\epsilon^{k+1}}\,,
    \]
    as desired. For the lower bound, observe that by \eqref{e:support-bound}, we have that for $z\in B(K)\cap \R$ and $t\in\supp \mu$,
    \[
    |t-z|\lesssim_\phi \tau^{-1}+K\leq (1+K)/\tau,
    \] from which it follows that 
    \[
        |\partial_z^{(k)} S_\mu(z)|  \geq  \frac{ k!\tau^{k+1}}{(1+K)^{k+1}} \,.
    \]
    Then we can extend the above estimate off the real line by observing that 
    \begin{align*}
    |\partial_z^{(k)}  S_\mu(z)| 
    &\geq 
    |\partial_z^{(k)}  S_\mu(\Re[z])| 
    -|\Im[z]|\max_{w\in[\Re[z], z]]}|\partial_z^{(k+1)}  S_\mu(w)|\\ 
    &\geq \frac{k!}{(\tau^{-1}+K)^{k+1}}-|\Im[z]|\frac{(k+1)!}{\eps^{k+2}}\,,
    \end{align*}
    where $[\Re[z],z]$ is the line segment connecting $\Re[z]$ and $z$, and taking $\Im[z]$ small enough.
\end{proof}

Throughout the following, it helps to have the following set
\begin{equation}\label{eq:E-eps-eta}
E_{\epsilon,K,\tau,k}(\mu)= E_\eps(\mu)\cap ([-K,K]\times[-\eta,\eta])
\end{equation}
where $E_\eps(\mu)$ is as in \eqref{eq:E-eps} and $\eta$ is chosen as in Lemma~\ref{l:partial_S_bound}.

\begin{lem}\label{l:Sd_est}
 For any $K,\eps,\tau,k_0>0$, and any $\tau$-regular $\mu$, we have that  $| S_\mu(z)|\geq \tau/2(1+K+\phi)$, and for $2\leq k\leq k_0$, 
    \begin{align}\label{e:high_moment}
        \Big|\phi \int \frac{x^k}{(1+ S_\mu(z) x)^k} \mu(dx)\Big| \lesssim_{\phi,k} \left(\frac{1+\phi}{\tau}\right)^k
   +\frac{(1+K)^{2k}}{(\tau \epsilon^2)^{2k}} 
    \end{align}
    uniformly for all $z\in E_{\eps,K,\tau,k}(\mu)$ 
\end{lem}
\begin{proof}
We prove $| S_\mu(z)|\geq \tau/2(1+K+\phi) $ by contradiction. Suppose for contradiction that $|S_\mu^{(d)}(z)|< \tau/2(1+K+\phi) $. Then for $x\in \supp(\mu)$, $|x|\leq \tau^{-1}$  and $|S_\mu(z)||x|< 1/(2(1+K+\phi))\leq 1/2$, and
 \begin{align*}
        \frac{2(1+K+\phi)}{\tau}&< \left|\frac{1}{S_\mu(z)}\right|\leq |z|+\phi \int \frac{|x|}{|1+ S_\mu(z) x|}\mu(dx)\\
        &< K+\int 2\phi|x|\mu(dx)\leq K+\frac{2\phi}{\tau}\leq 2(K+\phi)/\tau ,
    \end{align*}
which is impossible. This proves the first claim. 

We now turn to the second. By taking a $z$ derivative on both sides of \eqref{eq:def-FMC} we conclude
\begin{align*}
    \frac{\partial_z  S_\mu(z)}{( S_\mu(z))^2}=1+\phi \int \frac{\partial_z  S_\mu(z) x^2}{(1+ S_\mu(z) x)^2}\mu(dx),
\end{align*}
It follows by rearranging and recalling $|\partial_z ( S_\mu(z))|\gtrsim_\phi \tau^2/(1+K)^2$ from Lemma \ref{l:partial_S_bound}, that 
\begin{align*}
     \phantom{{}={}}\left|\phi\int \frac{ x^2}{(1+S_\mu(z) x)^2}\mu(dx)\right|
     & = \left|\frac{1}{(S_\mu(z))^2} -\frac{1}{\partial_z S_\mu(z)}\right|\\
     &\leq \left|\frac{1}{(S_\mu(z))^2}\right|+\left|\frac{1}{\partial_z S_\mu(z)}\right|
     \lesssim \frac{(1+\phi+K)^2}{\tau^2}+(\frac{\tau}{1+K})^2.
\end{align*}
We can take higher derivatives, and repeat the above computations to get
\begin{align*}
   \left|\phi\int \frac{ x^k}{(1+S_\mu(z) x)^k}\mu(dx)\right| 
   & \lesssim \left|\frac{1}{S_\mu(z)^k}\right|+\sum_{p=1}^{2k-1} \sum_{q_1+q_2+\cdots+q_{r}=p-1}\left|\frac{\partial_z^{q_1} S_\mu(z)\cdots \partial_z^{q_r}S_\mu(z)}{|\partial_z S_\mu(z)|^{p}}\right|\\
   &\lesssim_\phi \left(\frac{2(1+\phi+K)}{\tau}\right)^k+\sum_{p=1}^{2k-1} \sum_{q_1+q_2+\cdots+q_{r}=p-1} \frac{q_1!}{\epsilon^{q_1+1}}\cdots \frac{q_r!}{\epsilon^{q_r+1}}\frac{1}{p!}\Big(\frac{1+K}{\tau}\Big)^p\\
   &\lesssim_k \left(\frac{1+\phi}{\tau}\right)^k
   +\frac{(1+K)^{2k}}{(\tau \epsilon^2)^{2k}},
\end{align*}
and the claim \eqref{e:high_moment} follows. 
\end{proof}

\begin{lem}\label{l:num-small}
For any $K,\eps,\tau>0$, there is an $\eta=\eta(K,\eps,\tau)$ such that the following holds: For any $\tau$-regular $\mu$ and $\delta<1/2$, we have that uniformly over all $z\in E_{\eps,K,\tau,8}(\mu)$,
    \[
    \mu(\{x:\abs{1+S_\mu(z) x}\leq \delta\}) \lesssim_{\eps,\tau,K,\phi} \delta^2.
    \]
\end{lem}
\begin{proof}
Fix $z\in E_\eps$, with $\abs{z}\leq K$.
Note that by Lemma~\ref{l:Sd_est}, we have that
\[
\int \frac{x^2}{(1+S_\mu(z) x)^2 }\mathbf{1}{}(|1+S_\mu(z)x|\leq \delta )\mu(dx) \lesssim_{\eps,\tau,\phi,K} 1 
\]
Using the lower bound on $S_\mu$ from Lemma~\ref{l:Sd_est}, we see that for $|1+S_\mu |x||\leq \delta$ we have
$|x|\gtrsim_{\phi,\tau,K} 1-\delta$
so that
\[
\frac{(1-\delta)^2}{\delta^2}\mu(\{x:\abs{1+S_\mu(z) x}\leq \delta\})\lesssim_{\tau,\phi,K} 1.
\]
which yields the desired for $\delta<1/2$.
\end{proof}

\subsection{Limiting spectral measure and a ``No Outliers'' theorem}
In this section and the next, we will study properties of random matrices of the form $WDW^\top/n$, where $W$ is a $d\times n$ matrix with i.i.d. standard Gaussian entires,  $D$ is a matrix satisfying Assumptions~\ref{assumption:meas}--\ref{assumption:strong-convergence}. and $W$ and $D$ are taken to be \textbf{mutually independent}.   Let $\sigma_d=\sigma(WDW/n)$ denote this spectrum. The goal of this section is to prove that in this setting the limiting empirical spectral distribution is $\nu$ and that eventually a.s. there are no outliers in $\sigma_d$.

Let us start with the first claim. The following is a classical result of Marchenko--Pastur \cite{marchenko1967distribution}, specifically the version proved by Bai--Silverstein \cite{silverstein1995empirical}. 
\begin{prop}\label{p:free-mult-conv-wishart}
Suppose that $D= \diag(D_{\ell\ell})$ has independent entries that satisfy Assumptions~\ref{assumption:meas} and \ref{assumption:strong-convergence} and  let $W$ be an independent $d \times n$ random matrix with i.i.d. $\cN(0,1)$ entries.  Then if $\hat\mu_d$ denotes the the empirical spectral measure of $(1/n) WDW$, we have that
\[
\hat\mu_d \to \rho \boxtimes \nu_{\MP}
\]
weakly almost surely.
\end{prop}
\begin{proof}
    In \cite{silverstein1995empirical}, this is proved under the assumption that $\hat\rho^{(d)}\to \rho$ weakly a.s. This is the exactly what is guaranteed by Assumption~\ref{assumption:strong-convergence}.
\end{proof}

We now turn to showing that $\sigma_d$ has no outliers.  Here and in the rest of this section, let $E_\eps=E_\eps(\nu_\bfG)$. We also have here the following result whose proof follows a modification of a classical result of Bai--Silverstein~\cite{bai1998no}. 
\begin{prop}\label{prop:no-outliers}
Let $W$ be a $d \times n$ matrix with i.i.d. $\cN(0,1)$ entries and suppose that $D= \diag((D_{\ell\ell})_{\ell \le n})$ has independent entries which satisfy Assumptions~\ref{assumption:meas} and \ref{assumption:strong-convergence}. Suppose that $W$  and $D$ are independent, $n/d\to\phi$, and $\bfGd\to\bfG$. Then if $\sigma_d=\sigma(WDW/n)$ denotes the the spectrum of $(1/n) WDW$, we have that for any $\eps>0$
    \[
    \mathbb P( \sigma_d\cap E_\eps =\emptyset \text{ eventually }) = 1
    \]
    \end{prop}
\begin{proof}
In \cite{bai1998no}, the concentration of the extreme eigenvalues is proven under an extra assumption that $D$ is non-negative definite, which is the matrix $T_n$ in \cite{bai1998no}. The proof in \cite{bai1998no} also carries out when $D$ is not non-negative definite, by taking $D^{1/2}=\diag\{\sqrt{D_{\ell \ell}}\}_{1\leq \ell\leq n}$, which is a complex valued matrix. There are two places that non-negativity is assumed in \cite{bai1998no}. The first place is \cite[Lemma 2.11]{bai1998no}. But in fact its proof as given in \cite[Lemma 2.3]{silverstein1995strong} does not require that the matrix is non-negative definite. 
The second place is the claim $\Im[r_j^*((1/z)B_{(j)}-I)^{-1}r_j]\geq 0$ above \cite[Equation (3.4)]{bai1998no}. If we translate that statement to our setting, we have
\begin{align*}
    &\phantom{{}={}}\frac{1}{n}\Im\left[ W_{i\cdot} D^{1/2}((1/z)D^{1/2}\sum_{j\neq i} W^\top _{j\cdot} W_{j\cdot} D^{1/2}-I)^{-1} D^{1/2} W^\top_{i\cdot}\right]\\
    &=\frac{1}{n}\Im\left[ W_{i\cdot} ((1/z)\sum_{j\neq i} W^\top _{j\cdot} W_{j\cdot} -D^{-1})^{-1} W^\top_{i\cdot}\right]\\
    &
  =-\frac{1}{n} W_{i\cdot}((1/z)\sum_{j\neq i} W^\top _{j\cdot} W_{j\cdot} -D^{-1})^{-1} \Im\left[(1/z)\sum_{j\neq i} W^\top _{j\cdot} W_{j\cdot}\right]\overline{((1/z)\sum_{j\neq i} W^\top _{j\cdot} W_{j\cdot} -D^{-1})^{-1}} W^\top_{i\cdot}\\
  &=\frac{\Im[z]}{|z|^2n} W_{i\cdot}((1/z)\sum_{j\neq i} W^\top _{j\cdot} W_{j\cdot} -D^{-1})^{-1} \left(\sum_{j\neq i} W^\top _{j\cdot} W_{j\cdot}\right)\overline{((1/z)\sum_{j\neq i} W^\top _{j\cdot} W_{j\cdot} -D^{-1})^{-1}} W^\top_{i\cdot}\geq 0.
\end{align*}

It remains to check the conditions of that theorem. Specifically, it remains to show that, for any interval $[a,b]\in E_\tau$, we have that $\nu(E_\tau)=0$ and eventually almost surely $\nu^{(d)}([a,b])=0$ where $\nu^{(d)}=\hat\rho^{(d)}\boxtimes\nu_\bfG$. This will follow immediately from Assumption~\ref{assumption:strong-convergence}.

Indeed, first observe that by Assumption~\ref{assumption:strong-convergence} we have
\begin{equation}\label{eq:strong-convergence}
\max_{\ell\leq n} d_H(\Dll,\supp(\rho))\to 0.
\end{equation}
Now, fix $[a,b]\subseteq \supp(\rho)^c$ and let
\begin{align*}
    \psi_d(y) &= y+\phi^{-1}y\int(y-t)^{-1}t\hat\rho^{(d)}(dt)
\end{align*}
and define $\psi(y)$ analogously with $\rho$ in place of $\hat\rho^{(d)}$. Observe that $\psi(y)=\tS^{-1}(-1/y)$ (and $\psi_d(y)=(\tSd)^{-1}(-1/y)$) so that $\psi$ is well-defined on $[a,b]$ and by \eqref{eq:strong-convergence}, $\psi_d$ is as well for $d$ sufficiently large. 
Recall now that, by standard properties of Stieltjes transforms (see, e.g., \cite{silverstein1995analysis}), $[a,b]$ is a gap of $\tilde S$ if and only if $\psi'>0$ on that interval and the same holds for $\psi_d$ and $\tSd$ respectively. Thus it suffices to prove that $\psi_d'>0$ on $[a,b]$

To see this, observe that the integrands are uniformly bounded there as well. Thus, $\psi_d\to \psi$ pointwise by the bounded convergence theorem. Since  $\psi_d$ and $\psi$ are twice differentiable and convex there, we further more have $\psi_d\to\psi$ and $\psi_d'\to\psi'$ uniformly. Thus eventually $\psi_d'>0$ as desired.
\end{proof}

\subsection{Deterministic equivalent}
For the proof of our main theorem, we will need to provide a ``deterministic equivalent'' of a certain important matrix. To this end, we need to prove a modification of the Knowles--Yin anisotropic local law \cite{knowles2017anisotropic} to the setting of matrices of the form $WDW/n$, where $D$ can be negative and satisfies weaker regularity assumptions. 

We start first by proving a coarse estimate on the Green's function which follows from the arguments of \cite[Section 5]{knowles2017anisotropic}, but with some important modifications to account for the new setting. As this proof is somewhat lengthy, we only describe here what changes in the argument. In the remainder of this section, we then extend this bound all the way up to $\R$.

\begin{thm}\label{t:local_law-0}
Fix any $\tau>0$ independent of $n,d$. Suppose that $D$ is invertible and its empirical spectral measure is $\tau$-regular. Let $W$ denote an independent $d \times  n$ matrix with i.i.d.\ standard Gaussian entries.
Denote the Green's function and its approximation as
\begin{align}\label{eq:Pi-def}
 G(z)=\left[
\begin{array}{cc}
-z &  W\\
 W^\top  & -D^{-1}
\end{array}
\right]^{-1},
\quad
\Pi(z)=\left[
\begin{array}{cc}
\widetilde S^{(d)}(z)I_d & 0\\
0  & -D(I+\widetilde S^{(d)}(z)D)^{-1}
\end{array}
\right],
\end{align}
where $\widetilde S^{(d)}(z)$ solves \eqref{eq:defmz-sec2}.
For any unit vectors $\bm v, \bm u\in \mathbb R^{d+n-q}$, the following holds with probability $1-O(n^{-100})$ 
\begin{align}\label{eq:deterministic-equivalent}
\langle \bm v, (G(z)-\Pi(z))\bm u \rangle\leq \frac{C}{d^{1/4}\Im[z]^5}, 
\end{align}
uniformly over all $|z|\leq \tau^{-1}$.
We denote this by ${G}(z)\approx \Pi(z)$.
\end{thm}

\begin{proof}
The paper \cite{knowles2017anisotropic} studies sample covariance matrices with population covariance matrix $D=TT^\top$, which are by definition positive semi-definite. Under some extra assumption on the limiting empirical eigenvalue distribution of $D$, they prove a much stronger version of \eqref{eq:deterministic-equivalent} down to optimal scale, i.e., inside the bulk one can take $z$ such that $\Im[z]=O(d^{-1+o(1)})$, and close to the edge, $\Im[z]=O(d^{-2/3+o(1)})$. In particular they obtain a much tighter bound than \eqref{eq:deterministic-equivalent}. As we are only interested in this weaker bound, we can do away with the assumptions of positivity and, more importantly, regularity of the spectral edge. As \eqref{eq:deterministic-equivalent} is trivial on the real axis, we will work with $\Im[z]=\eta(d)>0$.

To prove Theorem~\ref{t:local_law}, we follow the same argument as in \cite[Section 5]{knowles2017anisotropic}, where a weak version of entry-wise local law is proven. The key step in this argument is to prove the stability of the self-consistent equation for the Stieltjes transform of the empirical spectral distribution. 
In our setting a self-consistent equation for the Stieltjes transform of the empirical eigenvalue distribution of $ W D  W^{\top}$
\begin{align}
    m_d(z):= \frac{1}{d}{\rm Tr}[( W D  W^{\top}-z)^{-1}]
\end{align}
can be derived as follows. Introduce a control parameter
\begin{align}
    \Lambda=\max_{1\leq i,j\leq d} \frac{| G_{ij}-\Pi_{ij}|}{|d_i d_j|}\,,
\end{align}
and the event $\Xi=\{\Lambda\leq (\log N)^{-1}\}$. By the same argument as in \cite[Lemma 5.3]{knowles2017anisotropic}, for any $C>0$ and small $a>0$, the following holds with probability $1-O(d^{-C})$ 
\begin{align}\label{e:conditional-stability}
   \bm1\{\Xi\}\left(-\frac{1}{m_d(z)}+\phi \int \frac{x}{1+m_d(z) x}\hat\rho^{(d)} (x)d x -z\right)\leq \bm1\{\Xi\}d^a \sqrt{\frac{(1/\Im[z])+\Lambda}{d\Im[z]}}\,,
\end{align}
provided $d$ is large enough. On the event $\Xi$, the above estimate can be viewed as an approximate self-consistent equation for $m_d(z)$ of the form
\begin{align}\label{e:self_consistent}
\left|
-\frac{1}{m_d(z)}+\phi \int \frac{x}{1+m_d(z) x}\hat\rho^{(d)} (x)d x -z\right|\leq \delta(z),
\end{align}
where $\delta(z)$ is the right hand side of \eqref{e:conditional-stability}.
(We pause here to note that all of the assumptions used in \cite{knowles2017anisotropic} on the empirical spectral distribution of $D$ were used to obtain a stronger estimate on this self-consistent equation, that we do not require here.)

When $z$ is bounded away from the real axis, the stability of \eqref{e:self_consistent} is also studied in \cite[Eqs 3.20, 3.21]{bai1998no}. By taking ($(m_n, m_n^0, \omega)$ there is our $(m_d, \wt S^{(d)}, \delta(z))$, they obtain that when $|\delta(z)|\leq \Im[z]$, then there exists some constant $C>0$
\begin{align}
    \Im[z]^2|m_d(z)-\widetilde S^{(d)}(z)|
    \leq C m_d(z)\widetilde S^{(d)}(z)\delta(z)\,,
\end{align}
for the Stieltjes transforms, $|\widetilde S^{(d)}(z)|, |m_d(z)|\leq 1/\Im[z]$, then it follows that 
\begin{align}\label{e:stability2}
    |m_d(z)-\widetilde S^{(d)}(z)|\leq \frac{C\delta(z)}{\Im[z]^4}\,.
\end{align}
Using the stability \eqref{e:stability2} as input, together with \cite[Lemma 5.3]{knowles2017anisotropic}, by the same argument as in \cite[Proposition 5.1]{knowles2017anisotropic} one can show that with probability at least $1-O(d^{-100})$
\begin{align}\label{eq:deterministic-equivalent1} 
  \max_{ij}|G_{ij}(z)-\Pi_{ij}(z)|,  |m_d(z)-\widetilde S^{(d)}(z)|\leq \frac{C}{d^{1/4}\Im[z]^5}.
\end{align}
Then the deterministic equivalent \eqref{eq:deterministic-equivalent} follows from the entry-wise deterministic equivalent \eqref{eq:deterministic-equivalent1} and  \cite[Lemma 6.2]{knowles2017anisotropic}.
\end{proof}
We now extend this result to the real line for our choice of $D$. We start with the following basic estimates. 

Recall that we let $\sigma_d$ denote spectrum of $(1/n) WDW$ and $B(K)$ is the ball of radius $K$ in $\mathbb C$. 

\begin{lem}\label{l:everyone-is-lipschitz}
 Assume that $\sigma_d\cap E_{\eps/2}(\nu)=\emptyset$ and that the empirical spectral measure of $D$ is $\tau$-regular. Then the functions $\wt S^{(d)}(z)$, $G_{ij}(z), \Pi_{ij}(z)$ are all Lipschitz with constant $C(\epsilon)$ for $z\in E_\eps$. 
\end{lem}
\begin{proof}
    We start with $\wt S^{(d)}(z)$. By our assumption on the spectrum, for $z,z'\in E_\eps $ and $x\in \sigma_d$, we have $|z-x|, |z'-x|\geq \eps/2$ so that 
    \begin{align*}
        \left|\wt S^{(d)}(z)-\wt S^{(d)}(z')\right|\leq \int \frac{|z-z'|\hat\rho^{(d)}(dx)}{|z-x||z'-x|}
        \leq |z-z'| \int \frac{\hat\rho^{(d)}(d x)}{\eps^2}\leq \frac{|z-z'| }{\eps^2}.
    \end{align*}
    Thus $\wt S^{(d)}(z)$ is Lipschitz with Lipschitz constant bounded by $1/\eps^2$. The quantity $\Pi_{ij}(z)$ as in \eqref{eq:Pi-def} is defined in terms of $\wt S^{(d)}(z)$, so is also Lipschitz. 

    For the Green's function $G_{ij}(z)$, we check the case that $1\leq i,j\leq d$, and the other cases can be proven in the same way
    \begin{align*}
        G_{ij}(z)=((1/n)WDW^\top-z)_{ij}^{-1}.
    \end{align*}
   For $z,z'\in E_\eps$, we have $|z-x|, |z'-x|\geq \eps/2$ for $x\in \sigma_d$, and $\|(-z-(1/n)WDW^\top)^{-1}\|\leq 2/\eps$, and 
\[
|\partial_z G_{ij}(z)|=|((1/n)WDW^\top-z)_{ij}^{-2}|\lesssim 4/\eps. \qedhere
\]
\end{proof}

\begin{lem}\label{l:greens-function-lipschitz}
Assume that $\sigma_d\cap E_{\eps/2}=\emptyset$ and that the empirical spectral measure for $D$ is $\tau$-regular. Then for $z\in E_\eps$
 the following holds: 
if for any unit vectors $\bm v, \bm u\in \mathbb R^{d+n-q}$ we have
 \begin{align}\label{eq:prelim-bound}
      \langle \bm v, (G(z)-\Pi(z))\bm u \rangle\lesssim \frac{1}{d^{1/4}\Im[z]^5}
  \end{align}
then in fact, we have
  \begin{align}
      \langle \bm v, (G(z)-\Pi(z))\bm u \rangle\lesssim_\eps \frac{1}{d^{1/24}}
  \end{align}
\end{lem} 

\begin{proof}
Take small $\eps'>0$, which will be optimized later.
    Let $z'=z+i \eps'$, then \eqref{eq:prelim-bound} gives
    \begin{align}\label{e:t1}
\langle \bm v, (G(z')-\Pi(z'))\bm u \rangle\leq \frac{C}{d^{1/4}(\eps')^5}.
\end{align}
By the Lipschitz property of $\langle \bm v, (G(z)-\Pi(z))\bm u \rangle$ from Lemma~\ref{l:everyone-is-lipschitz},
\begin{align}\label{e:t2}
    |\langle \bm v, (G(z')-\Pi(z'))\bm u \rangle-\langle \bm v, (G(z)-\Pi(z))\bm u \rangle|\lesssim_\eps |z'-z|\lesssim_\eps \eps'.
\end{align}
Combining \eqref{e:t1} and \eqref{e:t2}, we conclude 
\begin{align*}
   |\langle \bm v, (G(z)-\Pi(z))\bm u \rangle|\lesssim_\eps  \frac{1}{d^{1/4}(\eps')^5}+\eps'\lesssim_\eps \frac{1}{d^{1/24}},
\end{align*}
provided we take $\eps'=1/d^{1/24}$. 
\end{proof}

As an immediate corollary of  the above we obtain the following.
\begin{thm}\label{t:local_law}
Fix any $\tau>0$ independent of $n,d$. Suppose that $D$ has independent entries which satisfy Assumptions~\ref{assumption:meas}-\ref{assumption:strong-convergence}, $\bfGd\to\bfG$, and that $W$ is an independent $d \times  n$ matrix with i.i.d. standard Gaussian entries. Denote the Green's function $G(z)$ and its approximation $\Pi(z)$ as in~\eqref{eq:Pi-def}. 
For any unit vectors $\bm v, \bm u\in \mathbb R^{d+n-q}$
\begin{align}\label{eq:deterministic-equivalent-1}
\langle \bm v, (G(z)-\Pi(z))\bm u \rangle\lesssim_{\eps,\tau} \frac{1}{d^{1/24}}, 
\end{align}
uniformly over all $z\in E_\eps\cap B(\tau^{-1})$  eventually a.s.  We denote this by ${G}\approx \Pi$. In particular, we may take $\bm{v},\bm{u}$ to be random, provided they are independent of $W$.
\end{thm}
\begin{proof}
    This follows by combining Proposition~\ref{prop:no-outliers} with Theorem~\ref{t:local_law-0} and Lemmas \ref{lem:tau-reg} and \ref{l:greens-function-lipschitz}.
\end{proof}

We also obtain the following useful lemma.
\begin{lem}\label{lem:Stieltjes-ucas}
    For each $\eps>0$, We have that $\tilde{S}^{(d)}(z)\to \tS(z)$ uniformly on compacta of $E_\eps$  a.s. where $\tS$ is the Stieltjes transform of $\rho\boxtimes\nu_{\MP}$. 
\end{lem}
\begin{proof}
By Assumption~\ref{assumption:strong-convergence}, we have $\hat\rho^{(d)}\to\rho$ weakly a.s. Thus $\tSd\to \wt S$ pointwise on $\C_+$ so that in fact, $\nu_d\to\nu$ weakly a.s. We use here the classical fact that convergence of Stieltjes transforms is equivalent to their pointwise convergence on $\C_+$  \cite[Corollary 1]{geronimo2003necessary}. In fact, pointwise convergence of Stieltjes transforms is equivalent to their uniform convergence on compacta. (To see this, simply note that they must be locally bounded on compacta due to the usual bound $\abs{S(z)}\leq |\Im[z]|^{-1}$ for any Stieltjes transform $S$  so that the convergence follows from Montel's theorem.) Thus the convergence is uniform on compacta of $\C_+$. It remains to extend this to $E_\eps$. By combining Proposition~\ref{prop:no-outliers} with Lemma~\ref{l:everyone-is-lipschitz}, we see that these objects are uniformly lipschitz on $E_\eps$ from which the result follows.
\end{proof}

\subsection{Isolating dependencies}\label{sec:isolating-dependences}
In the following, it helps to decompose the matrix $A$ into two parts, one of which is correlated to $D$ and one of which is independent of it. 
To this end, let $L=[l_1 \cdots l_q]$ be as in \eqref{e:L-defining-property}. That is, let $L$ be a matrix whose columns are orthonormal
and satisfy the property that
\begin{equation*}
L^\top[x^1 \cdots x^\cC \mu_1 \cdots \mu_k] = \sqrt{\bfGd}.   
\end{equation*}
To see that such a matrix exists, simply compute the compact singular value decomposition of this matrix as
\begin{align}\label{e:decomp}
[x^1 \cdots x^\cC \,\,\mu_1 \cdots \mu_k] = U \Lambda V^\top,    
\end{align}
where $U$ is a $d \times q$ matrix with orthonormal columns, $\Lambda$ is diagonal with non-negative entries, and $V^\top$ is a $q \times q$ orthogonal matrix. The matrix $L= UV^\top$ satisfies the desired properties. It is important to note that in this construction, if $q>\bar q$ there is an arbitrary choice of the basis vectors $\{l_i\}_{i>\barq}$.
We can then decompose the matrix $A$  as a sum of its projection on ${\rm col}(L)$ and its orthogonal complement via
\begin{align}\label{e:Adecomp}
\sqrt{\lambda/d} A = LR^\top + UW\,,
\end{align}
where $U\in \R^{d\times d}$ is an orthogonal matrix whose first $q$ columns are given by $L$, and $W$ is an $d\times n$ matrix whose first $q$ rows are $0$, and other entries i.i.d.\ Gaussian $\mathcal N(0,1/d)$, and $R$ is $n \times q$ with $i$-th row, $r_i$, given by 
\begin{align}\label{e:ri}
r_i=
\widetilde r_i+ \frac{1}{\sqrt{d}}
\left[
\begin{array}{c}
g_1^{(i)}\\
\vdots\\
g_q^{(i)}
\end{array}
\right]\,
\qquad \text{where} \qquad 
\widetilde r_i
=
\frac{\sqrt{\lambda}}{\sqrt{d}}\left[
\begin{array}{c}
\langle l_1, \mu_{y_i}\rangle\\
\vdots\\
\langle l_q, \mu_{y_i}\rangle
\end{array}
\right], 
\end{align}
and where $g_i^{(q)}$ are independent standard Gaussian. Note that Assumption~\ref{assumption:meas} implies that $W$ and $D$ are independent. Also note that by construction, 
\begin{align}\label{e:UL}
U^\top L=
\left[
\begin{array}{c}
I_q\\
0
\end{array}
\right], 
\quad
W=\left[
\begin{array}{c}
0\\
\widetilde W
\end{array}
\right],
\end{align}
where $\widetilde W\in \mathbb R^{(d-q)\times n}$ is an Gaussian matrix with i.i.d.\ $\mathcal N(0,1/d)$ entries. Observe that \eqref{eq:Y-g-coupling} follows from this decomposition after recalling the above defining property of $L$, \eqref{e:L-defining-property}. Finally observe that by Gaussianity, $R$ is full rank a.s., so $LR^\top$ is rank $q$ a.s.

\subsection{Concentration of a reduced matrix}
With this decomposition in hand, we now turn to estimates on certain proxy matrices that will be useful in the following.
In particular, recall $\wt F$ from \eqref{eq:general-f-def} and define
\begin{equation}
    \label{eq:f-tilde-def}
    \widetilde F^{(d)}(z) =R^\top D(I+\widetilde S^{(d)}(z)D)^{-1}R,
\end{equation}
where $R$ is as in \eqref{e:Adecomp}. Here and in the remaining subsections, we let, $E_\eps=E_\eps(\nu_\bfG)$ and $E_{\eps,K,\tau}=E_{\eps,K,\tau,8}(\nu_\bfG)$.
The main goal of this subsection is to prove the following.
\begin{thm}\label{t:tfd-to-f}
    For every $\eps,K,\tau>0$ there is an $\eta$ such that 
    $\tFd\to \wt F$ uniformly
    over $E_{\eps,K,\tau}$ a.s.
\end{thm}

To prove this, we start by introducing the following cut-off.
Let $\chi_\delta$ be a smooth function on $\R$ which is supported on $(-\delta/2,\delta/2)^c$ and such that $\mathbf{1}\{(-\delta,\delta)^c\} \leq \chi_\delta \leq \mathbf{1}\{{(-\delta/2,\delta/2)^c}\}$, and let $\chi_\delta^c= 1-\chi_\delta$ which satisfies $\mathbf{1}\{[-\delta/2,\delta/2]\}\leq \chi_\delta^c \leq \mathbf{1}\{[-\delta,\delta]\}$.
Define
\begin{align}\label{eq:lambda-tilde-def}
    \Lambda^{(d)}(z) &=D(I+\widetilde S^{(d)}(z)D)^{-1}\\
    \Lambda^{(d)}_\delta(z) &=  D_\delta(I+\tS(z)D)^{-1} \label{eq:lambda-check-def} 
\end{align}
where $ D_\delta = \diag(\Dll\chi_\delta(1+\tS(z)\Dll))$. Then we have the following.

\begin{lem}\label{l:lambda-control}
    For every $\eps,K>0$, $0<\tau<1$ and $\delta<1/2$, we have that 
    \[
    \frac{1}{n}\norm{\Lambda^{(d)}(z)- \Lambda^{(d)}_\delta(z)}_F^2 \lesssim_{K,\epsilon,\tau,\phi} \delta
    \]
    uniformly on $\Eekt$ eventually a.s..
\end{lem}

\begin{proof}
Let $\tilde{D}_\delta$ be such that $\tilde D_{\ell\ell,\delta} = \Dll\cdot \chi_\delta(1+\tSd(z)\Dll)$. Let $\tilde{\Lambda}^{(d)}_\delta=\tilde{D}_\delta (I+\tSd D)^{-1}.$ We see that 
\begin{equation}\label{eq:lambda-control-1}
\frac{1}{n}\norm{\Lambda^{(d)}(z)-\tilde\Lambda^{(d)}_\delta(z)}^2\leq
\frac{1}{n}\sum_\ell \frac{\Dll^2}{(1+\tSd\Dll)^2}\mathbf{1}\{(|1+\tSd\Dll|\leq \delta)\}
\lesssim_{\eps,\phi,\tau} \delta
\end{equation}
on $\Eekt$ where the last inequality follows by Cauchy-Schwarz, Lemma~\ref{l:Sd_est}, and  Lemma~\ref{l:num-small}.

Let $\bar\Lambda^{(d)}_\delta= \tilde D_\delta(I+\tS(z)D)^{-1}$. Then
$\frac{1}{n}\norm{\tilde\Lambda^{(d)}-\bar\Lambda^{(d)}}^2$ is bounded by
\[
(\tS-\tSd)^2\cdot \frac{1}{n}\sum \frac{\Dll^4}{(1+\tSd\Dll)^2(1+\tS\Dll)^2}\mathbf{1}\{\abs{1+\tSd\Dll}\geq\delta/2\}\,.
\]
By Lemma~\ref{lem:Stieltjes-ucas}, $\tSd\to \tS$ uniformly on $E_\eps$ a.s. Thus this vanishes uniformly on $\Eekt$ provided we can show that the second term is bounded there a.s. To this end, observe that we can bound the summand by 
\[
\frac{\Dll^4}{(1+\widetilde S\Dll)^2(1+\tSd\Dll)^2} \lesssim
\frac{\Dll^4}{1+\tSd\Dll^4}+\frac{D^4}{(1+\widetilde S\Dll)^4}\,,
\]
by Young's inequality. Summing in $\ell$, and dividing by $n$, we have that the first term is bounded on $\Eekt$ by Lemma~\ref{l:Sd_est}. For the second term, note that we have that $\abs{\tS-\tSd}\leq \delta\cdot \tau /4$  on $\Eekt$ eventually almost surely. Thus eventually a.s., $\mathbf{1}\{\abs{1+\tSd\Dll}\geq \delta/2\}\leq\mathbf{1}\{\abs{1+\tS\Dll}\geq\delta/4\}$, so the second term satisfies, for  $z\in \Eekt$,
\[
\frac{1}{n}\sum \frac{D_{\ell\ell}^4}{(1+\tS\Dll)^4} \mathbf{1}\{(\abs{1+\tS\Dll}\geq \delta/4)\}\lesssim_{K,\eps,\phi,\tau} \frac{1}{\delta^4}
\]
eventually almost surely.
Combining this with the above we see 
\begin{equation}\label{eq:lambda-control-2}
\frac{1}{n}\norm{\tilde\Lambda^{(d)}_\delta-\bar \Lambda^{(d)}_\delta }^2\to 0    
\end{equation}
uniformly on $\Eekt$ for each $K>0$ almost surely. 

Now we relate this to  $\Lambda^{(d)}_\delta$. To this end, first observe that by definition of $\chi_\delta$, we have
\begin{align*}
\abs{\chi_\delta(1+\tSd\Dll)-\chi_\delta(1+\tS\Dll)} &\leq \mathbf 1\{\{\abs{1+\tSd\Dll}\geq\delta/2\}\triangle\{|1+\tS\Dll|\geq\delta/2\}\}
\end{align*}
Using again the convergence $\tSd\to \tS$ and that $\Dll \leq 1/\tau$,  we see that for $\delta\leq \delta_1$ for some $\delta_1$ non-random depending on $\tau$, this indicator satisfies 
\[
\frac{1}{n}\sum_\ell\mathbf 1\{\delta/4 \leq \abs{1+\tSd(z)\Dll}\leq\delta\} \lesssim_{\tau,\eps,\phi,K} \delta^2
\]
on $\Eekt$ eventually a.s., where the last equality follows again by Lemma~\ref{l:num-small}.
Thus by a similar argument as before 
\begin{equation}\label{eq:lambda-control-3}
\frac{1}{n}\norm{\bar\Lambda^{(d)}_\delta  - {\Lambda}^{(d)}_\delta}^2_F \lesssim_{\tau,\eps,\phi,K} \delta
\end{equation}
on $\Eekt$
eventually a.s. The result then follows by combining \eqref{eq:lambda-control-1}-\eqref{eq:lambda-control-3}.
\end{proof}

In the below, recall $\widetilde F^{(d)}$ from~\eqref{eq:f-tilde-def} and let 
\begin{align}
    \widetilde{F}^{(d)}_\delta &= R^\top  D_\delta(I+\tS(z)D)^{-1}R \label{eq:f-check-def} \,.
\end{align}

\begin{lem}\label{l:tfd-control}
    For every $\tau,\eps>0$ there is a $\delta_0$ such that for all $K>0$ and $\delta<\delta_0$,we have
    \[
    \limsup_{d\to\infty}\norm{\tFd-\tFd_\delta}\lesssim_{\eps,\tau,K} \delta
    \]
    uniformly on $\Eekt$ almost surely
\end{lem}
\begin{proof}
    By Cauchy--Schwarz we have
    \[
    \norm{\tFd-\tFd_\delta}^2_F \leq \frac{n}{d}\frac{1}{n}\norm{\Lambda^{(d)}-\Lambda_\delta^{(d)}}^2_F\cdot \sum_{i,j}\frac{1}{d}\sum_\ell (\sqrt{d}R_{\ell i})^2(\sqrt{d}R_{\ell j})^2
    \]
    The first term satisfies the bound we desire by Lemma~\ref{l:lambda-control} and the second term is bounded eventually a.s. by the strong law of large numbers for triangular arrays. 
\end{proof}
We now show that $\tilde{F}^{(d)}_\delta$ concentrates about its mean. 
\begin{lem}\label{l:tfd-delta-conc}
    For every $\tau,\eps>0$ there is an $\delta_0$ such that for all $\delta<\delta_0$, we have that 
    \[
    \tilde{F}^{(d)}_\delta-\E\tilde{F}^{(d)}_\delta\to 0\,,
    \]
    uniformly on compacta of $E_\eps$ a.s. In particular, $\E\tilde F^{(d)}_\delta$ is equicontinuous on compacta a.s.
\end{lem}

\begin{proof}
Fix $i,j$, and suppress the dependence when clear from context.
By definition,
\[
\tilde F^{(d)}_\delta(z)_{ij}=\frac{n}{d}\frac{1}{n} \sum_{\ell} \frac{\Dll}{1+\tS(z)\Dll}(\sqrt{d}R_{\ell i})(\sqrt{d}R_{\ell j})\chi_\delta(1+S(z)\Dll)\,.
\]
Observe that this is a triangular array of i.i.d. sums and that, since $\Dll\leq 1/\tau$ a.s., and $\chi_\delta(t) \leq\mathbf{1}\{(-\delta/2,\delta/2)^c\}(t)$ each summand has all moments finite. Thus by the strong law of the large numbers, 
\[
\norm{\tilde{F}^{(d)}_\delta-\E\tilde F^{(d)}_\delta}\to0\,,
\]
almost surely for each $z\in E_\eps$. It remains to upgrade this to local uniform convergence. To see this, note that it suffices to prove that the family
$(\tilde{F}^{(d)}_\delta-\E\tilde F^{(d)}_\delta)_{d}$
is equi-continuous in $d$ a.s. for fixed $\delta$. Observe that 
\begin{align*}
\frac{d}{dz}\frac{\Dll}{1+\tS(z)\Dll}\chi_\delta(1+\tS(z)\Dll)
&= -\frac{\Dll \tS'(z)}{(1+\tS(z)\Dll)^2}\chi_\delta(1+\tS(z)\Dll)\\
&\qquad+\frac{\Dll}{1+\tS(z)\Dll}\chi'_\delta(1+\tS(z)\Dll) \tS'(z)\Dll\,.
\end{align*}
Note that $\tS'(z)$ is locally uniformly bounded on $E_\eps$ by Lemma~\ref{l:partial_S_bound},
and that $\chi_d$ is Lipschitz.  Thus each term is  uniformly bounded by a similar argument to before upon appealing to Lemma~\ref{l:Sd_est}. Call this upper bound $M(K,\delta)$ for $z\in E_\eps$ with $|z|\leq K$. Then, for such $z$, since $n/d\to \phi$
\[
|\frac{d}{dz}(\tilde{F}^{(d)}_\delta(z)-\E\tilde F^{(d)}_\delta(z))|\leq C(\phi)M(K,\delta)\frac{1}{n}\sum_\ell \abs{\sqrt{d}R_{\ell i}}\abs{\sqrt{d}R_{\ell j}}.
\]
The last summand is bounded eventually almost surely by the strong law of large numbers. Thus this family is  locally uniformly equicontinuous a.s. Taking expectations, yields the same bound. The result then follows by the Arzela--Ascoli theorem.
\end{proof}

\begin{lem}
We have that 
\begin{equation}
    \E\wt F^{(d)}_\delta(z) =\lambda\phi \sum_{a} p_a \E \left[\frac{f_a(\lambda^{-1/2}\vec g, \bfG^{(d)})}{1+\tS(z)f_a(\lambda^{-1/2}\vec g, \bfG^{(d)})}(\frac{\vec g}{\sqrt{\lambda}} +\sqrt{\bfGd}_{\cdot,a})^{\tensor 2}\chi_\delta(1+\tS(z)f_a(\lambda^{-1/2}\vec{g},\bfGd))
\right]\label{eq:F^d-1}
\end{equation}
Here $\sqrt{\bfGd}_{\cdot,a}$ refers to the $\cC+a$th column of $\sqrt{\bfGd}$,  containing inner products with $\mu_a$.
\end{lem}
\begin{proof}
Recall from Assumption~\ref{assumption:meas}, that $D$ is a diagonal matrix whose entries equal, conditionally on  the hidden label $a(\ell)$ to a function of i.i.d. Gaussians and the summary statistic matrix $\bfG^{(d)}$, 
\begin{align}\label{e:defDll}
   D_{\ell \ell}= f_{a(\ell)}\left( 
\vec g_\ell /\sqrt{\lambda}; \mathbf G^{(d)}\right),
\end{align}
where $\vec g_{\ell} \sim \cN(0,I_q)$ and $f_y$ is as in \eqref{eq:f-in-terms-of-G}. Moreover, recalling \eqref{e:Adecomp}, we have that, conditionally on the hidden labels, $\sqrt d R_{\ell }= \vec g_\ell +\sqrt{\lambda}(L^\top \mu_{a(\ell)})$, for the same $\vec g$. Notice that $\sqrt d R_{\ell}$ is a Gaussian random vector with mean $\sqrt{\lambda}(L^\top \mu_{a})$ and variance $I$. Thus we have the desired:
\begin{align*}
\E\wt F^{(d)}_\delta(z) &= \phi \E\left[ \frac{D_{11}}{1+\tS(z)D_{11}} (\sqrt d R_{1 j})(\sqrt{d} R_{1 i})\chi_\delta(1+\tS(z)D_{11})\right]\nonumber\\
&= \phi \sum_{a} p_a \E \left[\frac{f_a(\lambda^{-1/2}\vec g, \bfG^{d})}{1+\tS(z)f_a(\lambda^{-1/2}\vec g, \bfG^{d})}(\vec g +\sqrt{\lambda}L^{\top}\mu_a)^{\tensor 2}\chi_\delta(1+\tS(z)f_a(\lambda^{-1/2}\vec{g},\bfGd))
\right]\\
&=\phi \lambda \sum_{a} p_a \E \left[\frac{f_a(\lambda^{-1/2}\vec g, \bfG^{d})}{1+\tS(z)f_a(\lambda^{-1/2}\vec g, \bfG^{d})}(\frac{\vec g}{\sqrt{\lambda}} +\sqrt{\bfGd}_{\cdot,a})^{\tensor 2}\chi_\delta(1+\tS(z)f_a(\lambda^{-1/2}\vec{g},\bfGd))
\right]
\end{align*}
where in the last line we have used \eqref{e:L-defining-property}. 
\end{proof}
Let
\[
\wt F_\delta(z)=\phi \sum_{a} p_a \E \left[\frac{f_a(\lambda^{-1/2}\vec g, \bfG)}{1+\tS(z)f_a(\lambda^{-1/2}\vec g, \bfG^{d})}(\vec g +\sqrt{\bfGd}_{\cdot,a})^{\tensor 2}\chi_\delta(1+\tS(z)f_a(\lambda^{-1/2}\vec{g},\bfG))
\right]\,.
\]

\begin{lem}\label{l:uniform-compacta}
For every $\eps,\tau,\delta>0$,  $\wt F^{(d)}_\delta(z)\to \wt F_\delta (z)$
uniformly on compacta of $E_\epsilon$  a.s.
\end{lem}
\begin{proof}
By Lemma~\ref{lem:tau-reg}, $f_a(\lambda^{-1/2}\vec g, \bfGd)$ is essentially bounded by $\tau^{-1}$. This combined with the weak convergence from Assumption~\ref{assumption:strong-convergence}, implies that 
\[
f_a(\lambda^{-1/2}\vec g, \bfGd)\to f_a(\lambda^{-1/2}\vec g, \bfG)
\]
not only in distribution but in fact in probability. 
Furthermore, this boundedness implies that 
\[
\frac{f_a(\lambda^{-1/2}\vec g, \bfG^{d})}{1+\tS(z)f_a(\lambda^{-1/2}\vec g, \bfG^{d})}\chi_\delta(1+\tS(z)f_a(\lambda^{-1},\bfGd)) 
\]
is essentially bounded by a constant depending on $\tau,\eps,\delta$. Similarly the sequence $\sqrt{\bfGd}$ is bounded as it is convergent. Thus the integrand is dominated by an integrable random variable. Using this convergence in probability, we have
$\wt F^{(d)}_\delta \to \wt F_\delta$ pointwise by the dominated convergence theorem.  Recall from Lemma~\ref{l:tfd-delta-conc}, that this family is equicontinuous on compacta of $E_\eps$. Thus this convergence is in fact uniform. 
\end{proof}

\begin{proof}[Proof of Theorem~\ref{t:tfd-to-f}]
Combining Lemmas~\ref{l:tfd-delta-conc}--\ref{l:uniform-compacta}, we have that
\[
\tFd_\delta - \wt F_\delta \to 0
\]
uniformly on compacta of $E_\eps$ a.s. Combining this with Lemma~\ref{l:tfd-control}, we see that for each $K$ and for all $\delta<1$,  
\[
\limsup_{d\to\infty}\norm{\tFd-\wt F}\leq \limsup_{d}\norm{\tFd-\tFd_\delta}+\norm{\wt F_\delta-\wt F}
\]
uniformly on $\Eekt$. The second term vanishes as $\delta\to0$ by the dominated convergence theorem. Similarly the first term is $O(\delta)$ by Lemma~\ref{l:tfd-control}. Taking $\delta\to0$ on the righthand side yields the desired.
\end{proof}

\subsection{Proof of Theorem~\ref{thm:deterministic-equivalent}}
The proof of this theorem  has three parts: proving the bulk, computing the outlier eigenvalues, and studying the outlier eigenpairs. To make this proof more readable we will break up the proof in to three subsections. We prove the bulk equation \eqref{eq:main-prop-bulk} in Section~\ref{sec:bulk-proof}; we compute the outlier eigenvalues and study their convergence in Section~\ref{sec:outlier-eval-proof}; we then end by showing that the corresponding eigenvectors are such that their projections converge to solutions of the effective eigenvector equations \eqref{eq:main-prop-eigenpairs-1}--\eqref{eq:main-prop-eigenpairs-2} in Section~\ref{sec:eff-evect-proof}. With these the proof will be complete. \qed

\subsection{Proof of \eqref{eq:main-prop-bulk}}\label{sec:bulk-proof}
We begin by identifying the ``bulk" distribution. Recall the decomposition 
$A = LR^\top + UW$ 
from \eqref{e:Adecomp}, where $LR^\top$ has rank $q$  a.s. Then
$$
\frac{\lambda}{d} A D A^\top 
= \bigl(LR^\top + UW\bigr)\,D\,\bigl(LR^\top + UW\bigr)^\top
= U W \,D\,W^\top U^\top + \text{(low rank)},
$$
where the low-rank part has rank at most $O(q)$.

By the Cauchy interlacing theorem, a rank-1 perturbation of a symmetric matrix results in a new matrix whose eigenvalues interlace with those of the original. Repeating this argument shows that a rank-$O(q)$ perturbation does not change the ``bulk" of the eigenvalues. Namely, the cumulative distribution function of the empirical spectral distribution shifts by at most $O(q/n)$ and therefore the limit of the empirical spectral distribution is the same as that of 
$\tfrac{1}{d}\,U W D W^\top U^\top$.
Since $U$ is an orthogonal matrix, 
$\tfrac{1}{d}\,U W D W^\top U^\top$ 
has the same eigenvalues as 
$\tfrac{1}{d}\,W D W^\top.$

As shown in Proposition~\ref{p:free-mult-conv-wishart}, the limit of the empirical eigenvalue distribution of 
$\tfrac{1}{d}\,W D W^\top$ 
is given by $\nu$, whose Stieltjes transform solves \eqref{e:ADA_eig}. This establishes the ``bulk" distribution \eqref{eq:main-prop-bulk}. \qed

\subsection{Proof of \eqref{eq:main-prop-outliers}}\label{sec:outlier-eval-proof}
We now turn to studying the outliers. 
We begin by expressing the characteristic polynomial as a product of two polynomials, one of which will give the ``bulk'' and the rest which will characterize the outliers. 
We follow a linearization argument.

By Lemma~\ref{lem:tau-reg}, $D$ is invertible. Thus by Schur's formula, the eigenvalues of $(\lambda/d)ADA^\top$ are given by the roots of 
\begin{align}\label{c:eq}
\det
\left[
\begin{array}{cc}
z & \sqrt{\lambda/d}A\\
\sqrt{\lambda/d}A^\top & D^{-1}
\end{array}
\right].
\end{align}
Throughout the following we refer to the matrix within this determinant as \emph{the linearization} of $\lambda ADA^\top /d$, or the linearization for short.
Decomposing $A$ as in \eqref{e:Adecomp}, we see that this is of the form  
\begin{align}
\det
\left(\left[
\begin{array}{cc}
z & UW\\
W^\top U^\top & D^{-1}
\end{array}
\right]
+\left[
\begin{array}{cc}
0 & LR^\top\\
RL^\top & 0
\end{array}
\right]\right).
\end{align}
As determinants are unchanged by orthogonal transformations, let us conjugate both matrices in the above by $U\oplus I$ so that, upon recalling \eqref{e:UL},
this is  equal to 
\begin{align}\label{eq:next-step}
\det
\left[
\begin{array}{ccc}
zI_q & 0 & R^\top\\
0& z I_{d-q} & \wt W\\ 
R & \wt W^\top & D^{-1}
\end{array}
\right].
\end{align}
If we define the matrix 
\[
M(z) = \left[
\begin{array}{cc}
zI_{d-q} & \wt W\\
\wt W^\top & D^{-1}
\end{array}
\right]
\]
then, provided $M$ is invertible, we have that \eqref{eq:next-step} is equal to
$\mfrp_n(z)\cdot \mfrq_n(z)$,
where $\mfrp_n(z)=\det M(z)$ and
\[
\mfrq_n(z)=\det(zI_q - \bar R^\top M^{-1}\bar R),
\]
where $\bar R^\top = [0_{q\times d-q}\, R^\top ]$ which is a $q\times (d-q+n)$ matrix.
Observe that $\mfrp_n(z)$ is the characteristic polynomial of $\wt WD\wt W/d$ times $\det(D)$. Thus, by Proposition~\ref{prop:no-outliers}, $\mfrp_n(z)$ is non-zero on $E_\eps$ and $M$ is invertible there. Thus it suffices to study the zeros of $\mfrq_n$ on this set.

Recalling Theorem~\ref{t:local_law} (and the Borel-Cantelli lemma), we see that the inner block matrix satisfies 
\begin{equation}\label{e:invert}
\left[
\begin{array}{cc}
z I_{d-q} & \widetilde W\\
\widetilde W^\top  & D^{-1}
\end{array}
\right]^{-1}
\approx 
\left[
\begin{array}{cc}
-\widetilde S^{(d)}(z)I_{d-q} & 0\\
0  & D(I_n+\widetilde S^{(d)}(z)D)^{-1}
\end{array}
\right]
\end{equation}
for $z\in E_\eps\cap B(\tau^{-1})$ eventually a.s.
where $\widetilde S^{(d)}(z)$ solves \eqref{eq:defmz-sec2} and where the precise definition of $\approx$ is as stated in Theorem~\ref{t:local_law}. Since $R$ is independent  of $\wt W$ and has finitely many columns, this implies that $\mfrq_n(z)$ is  uniformly close on compacta of $E_\eps$ to 
\begin{align}\label{e:mastereq3}
\mathfrak{h}_n(z) =\det(z-\widetilde F^{(d)}(z)).
\end{align}
where $\widetilde F^{(d)}=R^\top D(I+\widetilde S^{(d)}(z)D)^{-1}R$ is as in \eqref{eq:f-tilde-def}. That is, $\mfrq_n-\mathfrak h_n \to 0$ uniformly on compacta  of $E_\eps$ a.s.
Observe that by Theorem~\ref{t:tfd-to-f} we have that for each $K>0$
$\tFd  \to \wt F$
uniformly on compacta of $\Eekt$ almost surely. Thus, $\mfrq_n\to\mfrq$ with
\[
\mathfrak{q}(z)=\det(z-\wt F(z)).
\]
uniformly on $\Eekt$ for each $K>0$. 
Since $\mfrq_n$ and $\wt\mfrq_n$ are holomorphic there, the roots and number of roots on $\Eekt$ converge for each $K>0$ a.s.\qed

\subsection{Proof of \eqref{eq:main-prop-eigenpairs-1} }\label{sec:eff-evect-proof}
Next we characterize the eigenvectors corresponding to the outliers, specifically we are interested in understanding the projection in to the top space. 

We begin by recalling the following property of linearizations: We have that $v,w$ solve
\begin{align}
\left[
\begin{array}{cc}
-z^\prime & A/\sqrt d\\
A^\top/\sqrt d & -D^{-1}
\end{array}
\right]
\left[
\begin{array}{c}
v\\
w
\end{array}
\right]
=0
\end{align}
if and only if $z^\prime$ is an eigenvalue of $(1/d) ADA^\top$, $v$ is an associated eigenvector, and $w,v$ satisfy
\begin{align}
\frac{1}{d} ADA^\top v=z^\prime v,\quad w=D\frac{A^\top }{\sqrt d}v.
\end{align}
Now recall that if $A$ is a matrix and $P$ is orthogonal, then  $v$
is an eigenvector of $A$ with eigenvalue $z^\prime$ iff $Pv$ is an eigenvector of $PAP^\top$ with the same eigenvalue.  Thus, conjugating the linearization by $U\oplus I$ as above, it suffices to understand  solutions  $z^\prime,(v,w)$ with $v=(v_1,v_2)$ to

\begin{align}\label{e:eigrelation-1}
\left[
\begin{array}{ccc}
-z^\prime I_q & 0 & R^\top\\
0& -z^\prime I_{d-q} & \wt W\\ 
R & \wt W^\top & -D^{-1}
\end{array}
\right].
\left[
\begin{array}{c}
v_1\\
v_2 \\
u
\end{array}
\right]
=0
\end{align}
Here $v_1$ is the vector we are interested in, corresponding to the projection of the eigenvector into the relevant space $\text{Span}(x_1,...,x_\cC,\mu_1,...,\mu_k)$. To see this observe that if $u$ is the eigenvector of interest then $v=Uu$ and, by construction of $U$ from \eqref{e:Adecomp}, we have $v_1 =L^\top u$.

Let us now take $z^\prime = z^{(d)}$ be the eigenvalue as in the statement of the theorem and let us denote $v^{(d)}=(v_1^{(d)},v_2^{(d)})$ the corresponding eigenvector and $w^{(d)}$ the corresponding vector from the linearization. Observe that if we let $\wt w^{(d)}=(v_2^{(d)},w^{(d)})$, and recall $\bar R^\top = [0_{q\times d-q}\, R^\top ]$, then $(v_1^{(d)},\wt w^{(d)})$ solve
\begin{equation}
\begin{aligned}\label{eq:evect-relation}
    z^{(d)} v_1 &= \bar R^\top \wt w^{(d)}\\
    \bar R v_1^{(d)} & = M\wt w^{(d)}\,.
\end{aligned}
\end{equation}
Re-arranging this, we see that
\begin{equation}\label{e:evect-relation-reduced}
(z^{(d)}-\bar R^\top M^{-1}\bar R)v_1^{(d)}= 0.
\end{equation}
Applying the deterministic equivalent as in \eqref{e:invert} and using that $\tFd\to\wt F$  uniformly on $\Eekt$ by Theorem~\ref{t:tfd-to-f}, 
we get  
\[
\lim_{d\to\infty}  (z^{(d)}-\tFd(z^{(d)}))v_1^{(d)} = 0
\]
a.s.
In particular, any limit point of the sequence $v^{(d)}_1$ satisfies $
(z-\wt F(z))v_1 = 0$ as claimed. \qed

\subsection{Proof of \eqref{eq:main-prop-eigenpairs-2}}
It remains to control the norm of $L^\top u$. We consider only the case $z\neq 0$. In this case, if we let $\bar R$ as above, then we may rewrite \eqref{e:eigrelation-1} with $z^\prime =z^{(d)}$ as

\begin{align}
0=
\left(\left[
\begin{array}{cc}
-z^{(d)} I_q & 0\\
0  & M(z^{(d)})
\end{array}
\right]+\left[
\begin{array}{cc}
0 & \bar R^\top\\
\bar R & 0
\end{array}
\right]\right)\left[
\begin{array}{c}
v^{(d)}\\
u^{(d)}
\end{array}
\right].
\end{align}

As $z^{(d)}\to z\neq 0$, eventually $z^{(d)}\neq 0$. Since also $M$ is invertible on $E_\eps$, we have
\begin{align}
\left[
\begin{array}{c}
v^{(d)}\\
u^{(d)}
\end{array}
\right]=
\left[
\begin{array}{cc}
-z^{(d)}I_q & 0\\
0  & M(z^{(d)})
\end{array}
\right]^{-1}\left[
\begin{array}{cc}
0 & \bar R^\top\\
\bar R & 0
\end{array}
\right]\left[
\begin{array}{c}
v^{(d)}\\
u^{(d)}
\end{array}
\right].
\end{align}
Since we are working with unit eigenvectors, we have that $1=\norm{v^{(d)}}^2$ which is also equal to
\begin{align*}
 [(v^{(d)})^\top \,\, (u^{(d)})^\top ]
\left[\begin{array}{cc}
0 & \bar R^\top\\
\bar R & 0
\end{array}
\right]
&\left[
\begin{array}{cc}
-z^{(d)}I_q & 0\\
0  & M(z^{(d)})
\end{array}
\right]^{-1}
\left[
\begin{array}{cc}
I_d  & 0\\
0  & 0
\end{array}
\right]
\left[
\begin{array}{cc}
-z^{(d)}I_q & 0\\
0  & M(z^{(d)})
\end{array}
\right]^{-1}\left[
\begin{array}{cc}
0 & \bar R^\top\\
\bar R & 0
\end{array}
\right]\left[
\begin{array}{c}
v^{(d)}\\
u^{(d)}
\end{array}
\right]\\
&
 =\frac{1}{(z^{(d)})^2}\norm{R^\top u^{(d)}}^2+ (Rv_1^{(d)})^\top M^{-1}I_{d+q}\oplus M^{-1}
Rv_1^{(d)}
\end{align*}

The first term is $\norm{v_1}^2$ by \eqref{eq:evect-relation}. 
For the second term, note that it is equal to $-v_1^{(d)}\partial_z \psi_d(z^{(d)})v_1^{(d)}$ 
where
\[
\psi_d(z^\prime)= (z^\prime I_q-\bar R^\top M^{-1}(z^\prime)\bar R).
\]
which we view as a $q\times q$ matrix valued function of $z^\prime$. Observe that $\psi_d(z^\prime)$ is holomorphic on $\Eekt$. Furthermore, applying the deterministic equivalent \eqref{e:invert} as above we have that 
\[
\psi_d(z^\prime)- \wt F^{(d)}(z^\prime)\to 0,
\]
uniformly on $\Eekt$ eventually almost surely. As these are both holomorphic, their derivatives converge uniformly there as well, i.e.,  
\[
\partial_z\psi(z^\prime)-\partial_z\tFd(z^\prime)\to0
\]
uniformly on $\Eekt$ eventually a.s. Working on this probability 1 event, passing again to a subsequence along which $v_1^{(d)}\to u$, and using that $v_1^{(d)}$ is bounded in norm by 1 since it is a block of a unit vector, we have that 
\[
(v^{(d)}_1)^\top \psi_d^\prime(z^{(d)})v^{(d)}_1 \to u^\top \partial_z\wt F(z) u.
\]
eventually a.s. 
Combining these results, we see that  a.s. any limit point $u$ of $v_1^{(d)}$ must satisfy
\[
1 = \norm{u}^2- u^\top \partial_z \wt F(z)u
\]
as desired.  \qed

\subsection{Proof of invariance under choice of basis}\label{sec:invariant-under-change-of-basis}
In the above, it is natural to ask whether the functions above, specifically the effective outlier equations, truly depend only on $\bfG$, or also on the choice of basis $L$ in which $F$ is expressed. Indeed, while $\sqrt{\bfG}$ is uniquely defined, it is to some extent an arbitrary choice and there are multiple alternatives for this square root corresponding to points in the orbit of $\sqrt{\bfG}$ under the action of the orthogonal group $\text{O}(q)$. In particular, one may naturally ask about how the solutions to the equations \eqref{e:defmz-with-rho} and \eqref{eq:outlier-equation-general}-\eqref{eq:effective-evect-proj} vary along this orbit. As one might expect, the first two equations are invariant and the second two transform in the natural way. 

To see this, fix any $\mathbf O\in \text{O}(q)$, and consider $\bfO\sqrt{\bfG}$. Then $\bfG$ is of course unchanged,
\begin{align}
    \bfG=(\mathbf O \sqrt{\bfG})^\top (\mathbf O \sqrt{\bfG})\,.
\end{align}
We can define the $F$ matrix analogous to \eqref{eq:F-matrix} using $(\mathbf O \sqrt{\bfG})$:
\begin{align*}
     F_{ij}(z;\bfO)= \lambda\phi\sum_{a=1}^k p_a \mathbb E\left[\frac{f_a(\lambda^{-1/2}\mathbf O^\top\vec g; \mathbf{G})}{\lambda \phi+S(z)f_a(\lambda^{-1/2}\mathbf O^\top\vec g;\mathbf{G})} \left(\frac{g_i}{\sqrt{\lambda}}+(\bfO\sqrt{\bfG})_{ia}\right)\left(\frac{g_j}{\sqrt{\lambda}}+(\bfO\sqrt{\mathbf G})_{ja}\right)\right]
\end{align*}
Since the law of $\cN(0,I_d)$ is rotation invariant, namely $\bfO^\top \vec g$ and $\vec g$ have the same law, we have 
\begin{align*}
    (\bfO F(z)\bfO^\top)_{ij}=F_{ij}(z;\bfO)\,.
\end{align*}
In particular, rotating $\sqrt{\bfG}$ by $\bfO$ does not change the solutions of \eqref{eq:outlier-equation-general}
\begin{align}
    \det(zI_{q}-F( z))=\det(zI_q- \bfO F(z)\bfO^\top)=\det(zI_q- F(z;\bfO))=0
\end{align}
and transforms the solutions to \eqref{eq:effective-evect}-\eqref{eq:effective-evect-proj} in the natural way: if we let $U$ denote the space of solutions to this pair of equations for $\sqrt{\bfG}$ and $\tilde U$ denote the space of solutions for $\bfO\sqrt{\bfG}$, then $\bfO U =\tilde U$.  

An immediate consequence of this observation is the following lemma, regarding the number of solutions to \eqref{eq:outlier-equation-general}.
\begin{lem}\label{l:num-solutions}
    Recalling $\bar q \le q$ from~\eqref{eq:q-qbar}, the equation $\det(zI_{q}-F( z))=0$ has at most $\bar q$ real solutions in any connected component of $\mathbb R \setminus \supp(\nu_{\bfG})$. 
\end{lem}
\begin{proof}
By definition, $\sqrt{\bfG}$ has rank $\bar q$. In the following we take the orthogonal matrix $\mathbf O$ such that 
\begin{align}
    (\mathbf O\sqrt{\bfG})_{ia}=0, \text{ for } i=\bar q+1, \bar q+2,\cdots, q,
\end{align}
More crucially, $f_a(\lambda^{-1/2}\mathbf O^\top\vec g; \mathbf{G})$ does not depend on $g_i$ for $i=\bar q+1, \bar q+2,\cdots, q$.
In this way, the $q\times q$ matrix has two blocks:
for $i\neq j$ and $\max\{i,j\}> \bar q$ 
\begin{align}
     F_{ij}(z;\bfO)=0,
\end{align}
for $i=j>\bar q$ we have
\begin{align}
    F_{ii}(z;\bfO)
    &=\lambda\phi\sum_{a=1}^k p_a \mathbb E\left[\frac{f_a(\lambda^{-1/2}\vec g; \mathbf{G})}{\lambda \phi+S(z)f_a(\lambda^{-1/2}\vec g;\mathbf{G})} \frac{(\bfO \vec g)_i}{\sqrt{\lambda}}\frac{(\bfO \vec g)_j}{\sqrt{\lambda}}\right]\\
    &=\phi\sum_{a=1}^k p_a \mathbb E\left[\frac{f_a(\lambda^{-1/2}\vec g; \mathbf{G})}{\lambda \phi+S(z)f_a(\lambda^{-1/2}\vec g;\mathbf{G})} \right]
    =z+\frac{1}{S(z)} 
\end{align}
where in the last equality we used \eqref{e:defmz-with-rho}.
In summary, we have
\begin{align}\label{e:blockform}
     F(z;\bfO)
     =\left[
     \begin{array}{cc}
      F(z;\bfO)_{[{\bar q},\bar q]} & 0\\
      0 & (z+1/S(z))I_{q-\bar q}
\end{array}
\right]
\end{align}
By plugging \eqref{e:blockform} into \eqref{eq:outlier-equation-general}, and noticing that $1/S(z)\neq 0$, we can reduce \eqref{eq:outlier-equation-general} to the following equation,
\begin{align}\label{e:reducedeq}
    \det ( z I_{\bar q} -F(z;\bfO)_{[{\bar q},\bar q]}) =0 \,.
\end{align}
This has at most $\bar q$ solutions  in each connected component of $\mathbb R\setminus \supp(\nu_{\bfG})$. To see this, $B(z) = z I_{\bar q} -F(z;\bfO)_{[{\bar q},\bar q]}$ and observe that in every connected component of $\mathbb R \setminus \supp(\nu_{\bfG})$ this forms a 1-parameter family of matrices that are strictly increasing (in $z$) in the PSD sense. 
Indeed, differentiating in  $z$  
\begin{align}
    \partial_z B(z) = I_{\bar q} + \lambda \phi \sum_{a} p_a \mathbb E\Big[ \frac{S'(z) f_a^2}{(\lambda \phi + S(z) f_a)^2} (\frac{\bfO \vec{g}}{\sqrt{\lambda}} + \sqrt{\bfG})_{[\bar q\bar q]}^{\otimes 2}\Big]
\end{align}
is a strictly positive definite matrix. 
This implies each eigenvalue of $B(z)$ can cross $0$ at most for one value of $z$, so there are at most $\bar q$ solutions of $\det(B(z)) =0$.  
\end{proof}

\section{Detailed analysis of high-dimensional logistic regression}\label{sec:logistic-regression}
In this section, we begin by establishing Theorem~\ref{mainthm:logistic-regression}--\ref{mainthm:logistic-regression-outliers} as a direct consequence of Theorem~\ref{thm:main-A-D-A^T} by verifying that Assumptions~\ref{assumption:meas}--\ref{assumption:strong-convergence} apply. We then proceed to explicitly analyze the effective bulk and effective outlier equations at different points in the parameter space, with a particular focus on evaluating them in the early stages of training along the trajectory taken by the SGD as it learns a good classifier.

\subsection{Empirical matrices for high-dimensional logistic regression}\label{subsec:logistic-regression-empirical-matrices}
We begin with some computations giving explicit expressions for the empirical Hessian and Gradient matrices of the $k$-GMM logistic regression loss function. 
Recall $\varphi_\alpha$ \eqref{eq:f_a-kGMM-case}.
Differentiating~\eqref{eq:loss-function-1-layer} observe that the $(x^b ,x^c)$-block of the Gradient matrix and Hessian are, respectively,   
\begin{align}\label{eq:G-matrix-as-A-D-A^t}
        \widehat G_{bc} &=  \frac{1}{n} A D A^\top \qquad \text{where} \quad A = [Y_1 \cdots Y_n] \,, \quad D= \text{diag}((y_{\ell,b} - \varphi_b(\bfx^\top Y_\ell)  ) (y_{\ell,c} - \varphi_c(\bfx^\top Y_\ell))\,.\\
        \label{eq:Hessian-as-A-D-A^t}
    \widehat H_{bc} &=  \frac{1}{n} A D A^\top \qquad \text{where} \quad A = [Y_1 \cdots Y_n] \,, \quad D= \text{diag}(\varphi_b(\bfx^\top Y_\ell) \delta_{bc} - \varphi_b(\bfx^\top Y_\ell)\varphi_c(\bfx^\top Y_\ell))\,.
\end{align}
Note that the diagonals are, in law,  equivalent to \eqref{eq:f_a-kGMM-case}.

\subsection{Proof of Theorem~\ref{mainthm:logistic-regression}--\ref{mainthm:logistic-regression-outliers}}\label{subsec:applicability-to-logistic-regression}

We have to establish Assumptions~\ref{assumption:meas}--\ref{assumption:D}, and in the case where $\|\bfGd - \bfG\|\le o(1/(\log d))$ we have to show that Assumption~\ref{assumption:strong-convergence} also holds. Assumption~\ref{assumption:meas} is immediate upon recalling  the functions $(f_j^{M})$ for $M\in\{G,H\}$ from \eqref{eq:f_a-kGMM-case}. Since $\varphi_\alpha$ is smooth with bounded derivatives, Lemma~\ref{lem:bounded-differentiable-implies-strong-convergence} below, together with the assumption $\|\bfGd - \bfG\| = o(1/\log d)$ yields Assumption~\ref{assumption:strong-convergence}.

    Towards Assumption~\ref{assumption:D}, compact support follows from~\eqref{eq:f_a-kGMM-case} and the fact that $\varphi_\alpha\in [0,1]$. 
  To complete Assumption~\ref{assumption:D},  it remains to establish the uniform $O(\epsilon)$ bound on the probability of $f_j$ falling within $\epsilon$ of $0$. Fix any $\lambda$ and any $\mathbf{G}$ and consider $\mathbb P(|f_j(\lambda^{-1/2}\vec{g}, \mathbf{G})|<\epsilon)$ which since $f_j$ is non-negative a.s., is the same as $\mathbb P(f_j(\lambda^{-1/2}\vec{g}, \mathbf{G})\in [0,\epsilon))$. Beginning with the Hessian case, 
    \begin{align*}
        \mathbb P(f_j^H (\lambda^{-1/2}\vec{g} , \mathbf{G}) \in [0,\epsilon]) & \le \mathbb P(\varphi_\alpha (\bfG_{{[\cC]} j} +\bfG_{[\cC,\cC]}^{1/2}\lambda^{-1/2} \vec  g_{[\cC]}) \in [0,2\epsilon)) \\
        & \qquad+ \mathbb P(\varphi_\alpha (\bfG_{{[\cC]} j}+\bfG_{[\cC,\cC]}^{1/2}\lambda^{-1/2} \vec  g_{[\cC]}) \in (1-2\epsilon,1])\,.
    \end{align*}
    In turn, letting $\vec{z} = \bfG_{{[\cC]} j}+\bfG_{[\cC,\cC]}^{1/2}\lambda^{-1/2} \vec  g_{[\cC]}$, 
    \begin{align}\label{eq:softmax-prob-of-near-zero}
        \mathbb P(\varphi_\alpha (\vec{z}) \in [0,\epsilon))
       \le  \mathbb P(z_\alpha - \max_{b \ne \alpha} z_b < \log \epsilon) \le k \max_b \mathbb P(|z_b| > \tfrac{1}{2}|\log \epsilon|)\,.
    \end{align}
    This is a tail bound for a Gaussian with a finite mean and variance (depending only on $\mathbf{G} ,\lambda$) and therefore goes to zero as $e^{ - c_{\mathbf{G},\lambda} (\log \epsilon)^2} = \epsilon^{c_{\mathbf{G},\lambda} |\log \epsilon|}$ which is $O(\epsilon)$ as $\epsilon \downarrow 0$. An analogous bound applies to $\mathbb P(\varphi_\alpha(\vec{z}) \in [1-2\epsilon,1])$. 

    In the case of the gradient matrix, the function $f_j$ is given by $(\mathbf 1_{\alpha=j}- \varphi_\alpha(\vec{z}))^2$. The probability in this case satisfies for small $\epsilon$ that 
    \begin{align*}
        \mathbb P(f_j^G(\lambda^{-1/2} \vec{g};\mathbf{G})  \in [0,\epsilon)) \leq \mathbb P(\varphi_{\alpha}(\vec{z})\in [0,2\sqrt{\epsilon})) + \mathbb P(\varphi_{\alpha}(\vec{z}) \in (1-2\sqrt{\epsilon},1])\,.
    \end{align*}
    By the same reasoning as in the Hessian case, the Gaussian tails are bounded as $e^{ - c_{\mathbf{G},\lambda} (\log \sqrt{\epsilon})^2} = \epsilon^{-c_{\mathbf{G},\lambda}|\log \epsilon|/4}$, which is $O(\epsilon)$ as $\epsilon \downarrow 0$.  \qed

        \begin{lem}\label{lem:bounded-differentiable-implies-strong-convergence}
            Suppose that the function of Assumption~\ref{assumption:meas}, is of the form $D_{\ell \ell} = h_{J_\ell}(\mathbf{x}^\top Y)$ and that $(h_j)_{j=1}^{k}$ are  uniformly continuous. If $\|\bfG^{(d)} - \bfG\|\le o(1/\log d)$, then  Assumption~\ref{assumption:strong-convergence} holds. 
        \end{lem}
        \begin{proof}
            Since $h$ is continuous, the first part of Assumption~\ref{assumption:strong-convergence} is immediate. 
            For the ``strong convergence in law'', we begin by writing 
            \[
            (\langle\bfx,Y_\ell\rangle)_\ell \eqdist (\bfGd+\sqrt \bfGd \vec g_\ell )_\ell ,
            \]
            where $\vec g_\ell \sim \cN(0,I_q)$ are iid.
            By standard Gaussian tail bounds, eventually almost surely, 
            \[
            \max_{\ell\leq n} \norm{g_\ell}\lesssim \sqrt{\log n}.
            \]
            Thus
            \[
            \sup_\ell \norm{G+\sqrt G g_\ell -(\bfGd+\sqrt\bfGd g_\ell)} \lesssim \norm{\bfGd-\bfG}^{1/2}\sqrt{\log n}\to 0,
            \]
            by assumption on the rate of convergence. Thus
            \[
            \limsup \sup_\ell \abs{h_{J_\ell}(\bfGd + \sqrt\bfGd\lambda^{-1/2}g_\ell)-h_{J_\ell}(\bfG + \sqrt\bfG\lambda^{-1/2}g_\ell)}=0\,,
            \]
            by uniform continuity of h, so that
            \[
            \sup_\ell d(\Dll,\supp\rho)\to 0\,,
            \]
            as desired. 
        \end{proof}

\subsection{Effective bulk of the Hessian at initialization}
We consider two different, natural, initializations for the SGD algorithm, under which we give exact expressions for the effective bulk and outliers. The first natural initialization, under which the calculations can be done most explicitly, is the case where all the training parameters $\mathbf{x}$ are initialized exactly at $0$. 

\begin{lem}\label{lem:bulk-Hessian-0-init}
            Consider the $k$-GMM with loss function~\eqref{eq:loss-function-1-layer}, with orthonormal means, at the initialization point given by $\mathbf{x} \equiv 0$. The effective bulk $\nu_{\mathbf{G}}$ has Stieltjes transform 
            \begin{align}\label{eq:S(z)-Hessian-0-initialization}
    S(z) 
    & = \frac{- (\lambda \phi z + \xi(1-\phi)) + \sqrt{(\lambda \phi z - \xi(1+\phi))^2 - 4 \xi^2 \phi}}{2 \xi z}\,.
\end{align}
That is, the effective bulk is a Marchenko--Pastur distribution 
whose  left and right edges are at $$z_{\pm} = \frac{k-1}{k^2}\cdot \frac{1}{\lambda} \cdot \Big(1 \pm \frac{1}{\sqrt{\phi}}\Big)^2\,.$$ In particular, \begin{equation}\label{eq:right-edge-value}
    \lim_{z\downarrow z_{*,+}} S(z)  = - \frac{\lambda \phi}{\xi (1+\sqrt{\phi})}\,.
\end{equation}
\end{lem}

\begin{proof}
    By Theorem~\ref{mainthm:logistic-regression} applied to the empirical Hessian~\eqref{eq:Hessian-as-A-D-A^t}, we have 
the functional equation for the Stieltjes transform of the effective bulk, 
\[
1+zS(z)=\phi\sum_{j=1}^k p_j \mathbb E \Bigg[\frac{S(z) f_j(\lambda^{-1/2} \vec{g}; \mathbf{G})}{\lambda \phi + S(z) f_{j}(\lambda^{-1/2}\vec{g}; \mathbf{G})}\Bigg]\,,
\]
where $f_j= f_j^H$ was defined as~\eqref{eq:f_a-kGMM-case}. Since $\mathbf{x} \equiv 0$, the argument of $\varphi_\alpha$ is simply $\vec{0}$, and therefore, it is almost surely equal to $\frac{1}{k}$. In particular, for every $j\in [k]$, $f_j(\lambda^{-1/2}\vec{g},\mathbf{G})=\xi$ a.s. where
\begin{align*}
\xi= \frac{1}{k}(1-\frac{1}{k}) = \frac{k-1}{k^2}\,.
\end{align*}
 With that choice, the equation we get for the Stieltjes transform becomes 
\begin{align}\label{eq:0-init-bulk-equation}
1  + z S(z)  = \phi \frac{ \xi S(z)}{\lambda \phi+ \xi S(z)}\,,
\end{align}
which is the functional relation for the Stieltjes transform 
of a Marchenko--Pastur law with variance $\sigma^2=\xi/\lambda$ and aspect ratio $1/\phi$. In particular, we obtain the desired equations.
\end{proof}

We next examine the case where the initialization is a unit-norm Gaussian, i.e., $\mathbf{x} \sim \mathcal N(0, I_d/d)$. Here, the calculation will be less explicit because the law of $f_j$ will no longer be a delta-mass. 

\begin{lem}\label{lem:bulk-Hessian-normal-init}
    Consider the $k$-GMM with loss function~\eqref{eq:loss-function-1-layer}, with orthonormal means, equal weights $p_c \equiv \frac{1}{k}$, and initialization $\mathbf{x}\sim \mathcal N(0,I_d/d)$. The empirical spectral distribution of the Hessian at initialization converges to $\nu_{\MP}\boxtimes \nu_{\REM}$ where $\nu_{\REM}$ is the law of $\pi_{\REM} (1-\pi_{\REM})$ for 
    \begin{align}\label{eq:pi-rem}
        \pi_{\REM}(\alpha) = \varphi_\alpha(\vec{z}) \qquad \text{where} \qquad \text{$\vec{z} \sim \mathcal N(0,\lambda^{-1} I_\cC)$ i.i.d.}
    \end{align}
    Note this is just a \emph{random energy model (REM)} on $k$ points at inverse temperature $\lambda^{-1/2}$. 
\end{lem}
\begin{proof}
We apply Theorem~\ref{mainthm:logistic-regression}, upon expressing the limiting $\mathbf{G} = \lim_{d\to\infty} \mathbf{G}^{(d)}$ for $\mathbf{x} \sim \mathcal N(0,I_d/d)$. In particular, by standard concentration arguments  $\bfGd\to I_{\cC+k}=:\bfG$ a.s.
In that case, evidently $\varphi_\alpha(\bfG_{{[\cC]} j} +\bfG_{[\cC,\cC]}^{1/2}\lambda^{-1/2} \vec  g_{[\cC]}) = \varphi_\alpha(\lambda^{-1/2} \vec{g}_{[\cC]}) = \pi_{\REM}(\alpha)$. By Theorem~\ref{mainthm:logistic-regression}, the effective bulk's Stieltjes transform solves 
\begin{align}\label{eq:effective-bulk-Gaussian-initialization}
    \frac{1}{S(z)} + z = \phi \mathbb E\Big[\frac{\pi_{\REM}(\alpha)(1-\pi_{\REM}(\alpha))}{\lambda \phi + S(z) \pi_{\REM}(\alpha)(1-\pi_{\REM}(\alpha))}\Big]\,.
\end{align}
The claim follows from the definition of free multiplicative convolution \eqref{eq:def-FMC}.
\end{proof}

\subsection{Effective outliers of the Hessian at initialization}
We now describe the outlier locations, and their ``BBP-type" transitions as one varies the signal-to-noise $\lambda$ and sample complexity $\phi$, under the same two natural initializations. 

\medskip
\noindent \textbf{Hessian outliers at the all-$0$ initialization}: 

\begin{lem}\label{lem:Hessian-effective-outliers-0-init}
    If the means are orthonormal, then at the point $\mathbf{x} \equiv 0$, there are at most $k$ effective outliers, $(z_{*,j})_{j\in [k]}$, with the $j$'th one solving the equation 
    \begin{align}\label{eq:effective-outlier-locations-zero-init}
        S(z) = -\frac{1}{\xi} \frac{1}{p_j + \frac{1}{\lambda \phi}}\,,
    \end{align}
     which has a solution to the right of the effective bulk for $\lambda>\lambda_{c,j}^H:= \frac{1}{p_j \sqrt{\phi}}$. 
\end{lem}
\begin{proof}
Let us recall the functional equation for the effective outliers from 
\eqref{eq:outlier-equation-general}, namely
$0 = \det(z I_q - F( z))$ 
where $F$ is defined per~\eqref{eq:F-matrix}. Since  $\mathbf{x} \equiv 0$, we have $\mathbf{G} = 0_\cC \oplus I_k$ and $\barq = k$. Recall from the bulk analysis that $f_j(\lambda^{-1/2}\vec{g},\mathbf{G})=\frac{1}{k}(1-\frac{1}{k})=: \xi$ a.s. (in the Hessian case).
Since the means are orthonormal, by Section~\ref{sec:invariant-under-change-of-basis}, we may choose to work in the basis where $l_i = \mu_i$ for $i=1,...,k$ and therefore $\sqrt{\bfG}_{[k,k]} = I_k$. 
The matrix $F$ from~\eqref{eq:F-matrix} reduces to   
\begin{align*}
   F(z) =\lambda \phi \sum_{j=1}^{k} p_j \frac{\xi}{\lambda \phi +  S(z) \xi} \Big(\frac{I_k}{\lambda}  + (\sqrt{\mathbf{G}}_{[k,k]})_{\cdot,j}^{\otimes 2}\Big)\,.
\end{align*}
Plugging in the effective bulk equation~\eqref{eq:0-init-bulk-equation} and using that $\sqrt{\bfG}_{[k],j} = e_j$, this becomes 
\begin{align}\label{eq:F-at-0-initialization}
    F(z)  = \big(z + \frac{1}{S(z)}\big) I_k + \lambda \phi \sum_{j=1}^{k} p_j \frac{\xi}{\lambda \phi + \xi  S(z)} e_{j}^{\otimes 2}\,.
\end{align}
Observe that this matrix is diagonal.
Setting the determinant equal to zero (the multiplication by $-S(z)$ not changing anything since $S(z) \ne 0$ for fixed $z$ to the right or left of the support), we get that the effective outliers $(z_{*,j})_{j=1}^{k}$ 
are the $k$ solutions to 
\begin{align}\label{eq:zero-init-outlier-equation}
    -1= p_j \lambda \phi \frac{\xi S(z)}{\lambda \phi + \xi S(z)} 
\end{align}
which gives~\eqref{eq:effective-outlier-locations-zero-init} after rearranging. 

It remains to understand when these solutions exist.
Comparing~\eqref{eq:effective-outlier-locations-zero-init} to the value of $S(z)$ as $z$ approaches the right edge of the effective bulk as in
~\eqref{eq:right-edge-value} 
we see that the $j$'th effective outlier is to the right of the effective bulk provided $\lambda$ is above the critical value 
\begin{align}\label{eq:lambda-crit-zero-init}
    \lambda_{c,j}^H  = \frac{1}{p_j \sqrt{\phi}}\,,
\end{align}
as claimed.
\end{proof}

\begin{rem}\label{rem:outlier-eigenvectors-zero-init}
In the $\mathbf{x} \equiv 0$-initialization case of the Hessian with orthonormal means, the outlier eigenvectors can be described fairly explicitly as well. 
Suppose effective outlier $z_{*}=  z_{*,j}$ solves~\eqref{eq:effective-outlier-locations-zero-init}, and the vector $v$ satisfies $(z_* I_k - F(z_*)) v=0$. Since $z_* I_k - F$ was a diagonal matrix, a choice (if the $p_j$'s are distinct, the only choice) would be for $v$ to be the vector $e_j$. (In the case where $p_a = p_b$ for two classes $a,b$ then any linear combinations of $e_a, e_b$ of course will be in the kernel of $z_* I - F(z_*)$, but the choice of eigenvectors $e_a$ and $e_b$ can be made.) This is the choice in the basis $l_1,...,l_k$, but these were taken to be exactly the mean vectors $\mu_1,...,\mu_k$. Therefore, by Theorem~\ref{mainthm:logistic-regression-outliers}, at finite $d$, any outlier that is $o(1)$ away from $z_*$ has a choice of eigenvector whose projection into $\text{Span}(\mu_1,...,\mu_k)$ is $o(1)$ away from $c_j \mu_j$ for a constant $c_j$ governed by~\eqref{eq:effective-evect-proj}. 

To calculate how large that constant $c_j$ is, per~\eqref{eq:effective-evect-proj}, it equals
\begin{align*}
    c_j^2 ( 1- \langle e_j, \partial_z F(z_*) e_j\rangle)=1
\end{align*}

Then by Theorem~\ref{mainthm:logistic-regression-outliers}, the size of that outlier's projection in $\text{Span}(\mu_1,...,\mu_k)$ is equal to  
\begin{align*}
    c_j = \Big(\frac{1}{1-\langle e_j,\partial_z F(z_*)e_j\rangle}\Big)^{1/2}\,.
\end{align*}
Differentiating~\eqref{eq:F-at-0-initialization} in $z$, we get 
\begin{align*}
    \partial_z F(z) = \Big(1 - \frac{S'(z)}{S(z)^2}\Big)I_k  - \lambda \phi \sum_{j} p_j \frac{\xi^2 S'(z)}{(\lambda \phi + \xi S(z))^2} e_j^{\otimes 2}\,.
\end{align*}
Thus the eigenvector for eigenvalue $z_*$ has form $c_j \mu_j + w+o(1)$ for $w\in \text{Span}(\mu_1,...,\mu_k)^{\perp}$ with  
\begin{align}\label{eq:outlier-eigenvector-size-zero-init}
    c_j =  \Big(\frac{S'( z_*)}{S( z_*)^2}  + \lambda \phi p_j \frac{\xi^2 S'(z_*)}{(\lambda \phi + \xi S(z_*))^2}\Big)^{-1/2}\,.
\end{align}
\end{rem}

\medskip
\noindent \textbf{Hessian outliers at the Gaussian initialization}: 
Recall from Lemma~\ref{lem:bulk-Hessian-normal-init}, under the Gaussian initialization, the bulk spectrum is already quite complicated to characterize, being the free multiplicative convolution of $\nu_{\MP}$ with $\nu_{\REM}$. The following lemma gives a family of individual fixed point equations for the effective outlier locations.

\begin{lem}\label{lem:Hessian-effective-outliers-Gaussian-init}
    At the Gaussian initialization $\mathbf{x} \sim \mathcal N(0, I_d/d)$, with orthonormal means and equal weights $(p_j)_j \equiv \frac{1}{k}$, the effective outlier solutions consist of a multiplicity $k$ solution to\begin{align}\label{eq:Gaussian-init-fixed-pt-1}
    zS(z) = - \frac{1+\lambda}{\lambda}\,.
\end{align}
    If we let $\Pi = \varphi_{\alpha}(\lambda^{-1/2}\vec{g}_{[\cC]})$, the other $\cC$ solutions consist of solutions to   
\begin{align}
    z &= \phi \mathbb E\Big[  \frac{\Pi(1-\Pi)}{\lambda \phi + S(z) \Pi (1-\Pi)} (g_1^2 - g_1 g_2)\Big]\,, \label{eq:Gaussian-init-fixed-pt-2}\\
    z & = \phi \mathbb E\Big[  \frac{\Pi(1-\Pi)}{\lambda \phi + S(z) \Pi (1-\Pi)} (g_1^2 + (\cC-1) g_1 g_2)\Big]   + \cC\phi^2 \lambda \frac{\mathbb E[\frac{\Pi (1-\Pi)}{\lambda \phi + S(z) \Pi(1-\Pi)}g_1]^2}{z-\phi(1+\lambda)\mathbb E[\frac{\Pi (1-\Pi)}{\lambda \phi + S(z) \Pi(1-\Pi)}]}\,, \label{eq:Gaussian-init-fixed-pt-3}
\end{align}
with the former having multiplicity $\cC-1$ and the latter having multiplicity $1$. 
\end{lem}

\begin{proof}

\medskip
Suppose that $\mathbf{x}\sim \mathcal N(0, I_d/d)$. Recall from the proof of Lemma~\ref{lem:bulk-Hessian-normal-init} that almost surely $\mathbf{G}^{(d)}$ converges to $\mathbf{G} = I_q$, and that by Section~\ref{sec:invariant-under-change-of-basis}, we can freely take the orthonormal basis vectors $l_1,...,l_q$ to be $x_1+o(1),...,x_\cC+o(1), \mu_1,...,\mu_k$. 
Recall the function
\begin{align*}
    f_j(\lambda^{-1/2} \vec{g}, \mathbf{G}) =[ \varphi_\alpha (1-\varphi_\alpha )] (\mathbf{G}_{[\cC]j}+ \lambda^{-1/2} \mathbf{G}_{[\cC\cC]}^{1/2}\vec{g}_{[\cC]}) = [ \varphi_\alpha (1-\varphi_\alpha )] (\lambda^{-1/2} \vec{g}_{[\cC]})\,,
\end{align*}
and notice that it is independent of $\vec{g}_{[k]}$ and does not depend on $j$.

If we write $\Pi = \varphi_\alpha(\lambda^{-1/2}\vec{g}_{[\cC]})$, and expand the tensor product in the definition of $F$ from \eqref{eq:F-matrix}, we see that F splits as 
$$F = I + II + III + IV$$ 
where these are the following four $q\times q$ matrices: 
\begin{align*}
    I_{ij} & =  \phi \sum_{c=1}^k p_c   \mathbb E\Big[\frac{\Pi(1-\Pi)}{\lambda \phi + S(z) \Pi(1-\Pi)} g_i g_j\Big]\,,\\ 
    II_{ij} = III_{ji} & = \sqrt{\lambda } \phi \sum_{c=1}^k p_c   \mathbb E\Big[\frac{\Pi(1-\Pi)}{\lambda \phi + S(z) \Pi(1-\Pi)} g_i\Big]\delta_{j,c}\,,\\
    IV_{ij} & = \lambda \phi \sum_{c=1}^k p_c   \mathbb E\Big[\frac{\Pi(1-\Pi)}{\lambda \phi + S(z) \Pi(1-\Pi)}\Big] \delta_{i,c}\delta_{j,c}\,.
\end{align*}
Let us consider $I$--$IV$ individually. In term $I$, by independence of $\Pi$ from $\vec{g}_{[k]}$,
\begin{align*}
    I_{ij} = \begin{cases}\phi \mathbb E\Big[\frac{\Pi(1-\Pi)}{\lambda \phi + S(z) \Pi (1-\Pi)} g_i g_j\Big]  & i,j\in \{1,...,\cC\} \\
    \phi \mathbb E\Big[\frac{\Pi(1-\Pi)}{\lambda \phi + S(z) \Pi (1-\Pi)}\Big] \mathbf{1}_{i=j} & i,j\in \{\cC+1,...,q\} \\ 0 & \text{else}\end{cases}\,.
\end{align*}
Notice that in the bottom-right block, this is diagonal, and in the top-left block it has the form of equaling one explicit constant if $i= j$ and another if $i\ne j$. 
We now turn to the matrices $II$--$III$. These are only non-zero if $j = \cC+c$, and (by independence of $\Pi$ and $\vec{g}_{[k]}$, the matrix $II$ is only non-zero in the top-right block where $i\in [\cC]$ and $j\in [k]$. In that block, since we assume $(p_c)_c \equiv \frac{1}{k}$, in that block it takes a constant value.  
The same is true for the bottom-left block of $III$. 
Finally, in block $IV$ it only has a non-zero component in the $i,j\in [k]$ block where again assuming $(p_c)_c \equiv \frac{1}{k}$, it takes some other constant times the identity. 

Adding these together, we get the matrix 
\begin{align*}
    F(z) = \left(\begin{matrix}
         c_z\mathbf{1}_{\cC\times \cC} + d_z I_{\cC}  & b_z\mathbf 1_{\cC\times k} \\ 
        b_z\mathbf{1}_{k\times \cC} & a_z I_k
    \end{matrix}\right)\,,
\end{align*}
where $a_z,b_z,c_z,d_z$ are the following explicit Gaussian integrals,
\begin{align*}
    a_z &= \phi(1+\lambda) \mathbb E\Big[ \frac{\Pi(1-\Pi)}{\lambda \phi + S(z) \Pi (1-\Pi)}\Big]\,, \qquad \qquad 
    b_z  = \phi \sqrt{\lambda} \mathbb E\Big[ \frac{\Pi(1-\Pi)}{\lambda \phi + S(z) \Pi (1-\Pi)}g_1\Big]\,, \\ 
        c_z & = \phi  \mathbb E\Big[ \frac{\Pi(1-\Pi)}{\lambda \phi + S(z) \Pi (1-\Pi)}g_1 g_2\Big]\,, \qquad \qquad 
        d_z  = \phi  \mathbb E\Big[ \frac{\Pi(1-\Pi)}{\lambda \phi + S(z) \Pi (1-\Pi)}(g_1^2 - g_1 g_2)\Big]\,,
\end{align*}
capturing different correlations between the REM-measure $\Pi$ and its constituent Gaussians. 

Using Schur complement, the solutions to $\det(zI- F)=0$ can be found to be at solutions of 
\begin{align*}
    \det((z- a_z) I_k) = 0\,, \qquad \text{or} \qquad 
    \det \Big((z -d_z)I_\cC - (c_z + \frac{b_z^2}{z-a_z})\mathbf 1_{\cC\times \cC}\Big) =0\,.
\end{align*}
The first one gives solutions of multiplicity $k$ of  
\begin{align}
    z=\phi(1+\lambda) \mathbb E\Big[\frac{\Pi(1-\Pi)}{\lambda \phi + S(z) \Pi(1-\Pi)}\Big]\,,
\end{align}
which recalling the bulk equation~\eqref{eq:effective-bulk-Gaussian-initialization} is the same as~\eqref{eq:Gaussian-init-fixed-pt-1}.
The second determinant equaling zero gives equations~\eqref{eq:Gaussian-init-fixed-pt-2}--\eqref{eq:Gaussian-init-fixed-pt-3},
where the~\eqref{eq:Gaussian-init-fixed-pt-2} solution comes with multiplicity $\cC-1$ and the~\eqref{eq:Gaussian-init-fixed-pt-3} has multiplicity $1$. The corresponding effective eigenvectors can of course be solved for in terms of $a_z, b_z, c_z, d_z$ straightforwardly. 
\end{proof}

\begin{proof}[\textbf{\emph{Proof of Corollary~\ref{cor:exact-distribution-at-initialization}: Hessian matrix case}}]
    The effective bulk is the result of Lemma~\ref{lem:bulk-Hessian-0-init}. The results  on the spectral transition and effective outliers are the content of Lemma~\ref{lem:Hessian-effective-outliers-0-init}. 
\end{proof}

\begin{proof}[\textbf{\emph{Proof of Corollary~\ref{cor:Gaussian-init-distribution}}}]
    The effective bulk claim is the result of Lemma~\ref{lem:bulk-Hessian-normal-init} and the effective outliers claim is the content of Lemma~\ref{lem:Hessian-effective-outliers-Gaussian-init}.  
\end{proof}

\subsection{The Gradient matrix at initialization}
We now explain how the above calculations change if we consider the Gradient matrix instead of the Hessian. We will focus on the all-zero $\mathbf{x}\equiv 0$ initialization with equal weight orthonormal means, since our point here is mainly to demonstrate the difference in outlier emergence at initialization between this case and the empirical Hessian. 

\begin{proof}[\textbf{\emph{Proof of Corollary~\ref{cor:exact-distribution-at-initialization}: Gradient matrix case}}]
In the $\mathbf{x}\equiv 0$ case, just as with the Hessian, the quantity $f_j= f_j^G$ is deterministic, but now at a different, \emph{$j$-dependent} value. To be more precise, the Gradient matrix case at $\bfx=0$ is given by 
\begin{align*}
   f_j(\lambda^{-1/2} \vec{g}; \mathbf{G}) =  (\mathbf{1}_{\alpha =\cC(j)} -\varphi_\alpha(0))^2
\end{align*}
which simplifies to $$f_j =\begin{cases} \xi_1:= (\frac{k-1}{k})^2 & \cC(j)=\alpha \\ \xi_0:= \frac{1}{k^2} & \cC(j)\ne \alpha
\end{cases}\,.$$ 
The self-consistent equation for the Stieltjes transform for the bulk~\eqref{e:defmz-with-rho} is 
\begin{align*}
    1+zS(z) &  = \phi \Bigg(p_\alpha \frac{\xi_1 S(z)}{\lambda \phi + \xi_1 S(z)}+ (1-p_\alpha)\frac{\xi_0S(z)}{\lambda \phi + \xi_0 S(z)}\Bigg) \\
    & = \phi S(z)\Bigg(\frac{(\xi_1 p_\alpha + \xi_0 (1-p_\alpha))\lambda \phi + \xi_1 \xi_0S(z)}{(\lambda \phi + \xi_1 S(z))(\lambda \phi + \xi_0 S(z))}\Bigg)\,,
\end{align*}
where $p_\alpha=\sum_{\cC(k)=\alpha} p_k$ is the sum of weights of the mixtures components in $\alpha$.

Now let us consider the corresponding outlier equations. Since $f_j$ is deterministic, using the assumption that  $(p_c)_{c\in [k]} \equiv \frac{1}{k}$, the $k\times k$ matrix $F$~\eqref{eq:F-matrix} reduces to 
\begin{align*}
    F(z) 
    & = \frac{\lambda \phi}{k} \Bigg(\Big(\frac{\xi_1}{\lambda \phi + S(z) \xi_1} + \frac{(k-1)\xi_0 }{\lambda \phi + S(z) \xi_0}\Big)\frac{I_k}\lambda  \\ 
    & \qquad \qquad \quad+ \frac{\xi_1}{\lambda \phi + S(z) \xi_1} \sum_{j: \cC(j) = \alpha} e_j^{\otimes 2}+ \frac{\xi_0}{\lambda \phi + S(z) \xi_0} (I_k -  \sum_{j: \cC(j) = \alpha} e_j^{\otimes 2})\Bigg)\,.
\end{align*}
This is a diagonal matrix taking one value $c_\alpha$ in entries whose class label is $\alpha$, and another value $c_b$ in the non-$\alpha$ entries. Therefore, the solutions to $\det(zI- F)=0$ become $z$ solving the fixed point equations 
\begin{align}
    z  & = \frac{\xi_1}{\lambda \phi + S(z) \xi_1}\Big(\frac{\phi(1+\lambda)}{k}\Big) + \frac{\xi_0}{\lambda \phi + S(z) \xi_0} \Big(\frac{\phi(k-1)}{k}\Big)\,,  \qquad \text{or} \label{eq:G-matrix-outlier-1}\\ 
    z  & = \frac{\xi_1}{\lambda \phi + S(z) \xi_1} \Big(\frac{\phi}{k}\Big) + \frac{\xi_0}{\lambda \phi + S(z) \xi_0}\Big(\frac{\phi(k-1 + \lambda)}{k}\Big)\,, \label{eq:G-matrix-outlier-2}
\end{align}
the former with multiplicity $|\cC^{-1}(\alpha)|$ and the latter with multiplicity $k - |\cC^{-1}(\alpha)|$.

Since $S(z)$ is negative and increasing to the right of the effective bulk's support, each of these have at most one solution in $z$ to the right of the bulk. For $\lambda$ sufficiently large, each has exactly one solution to the right of the bulk, and we can define $\lambda_{c,1},\lambda_{c,2}\ge 0$ as the infimums over $\lambda$ for which the solutions to~\eqref{eq:G-matrix-outlier-1}--\eqref{eq:G-matrix-outlier-2} respectively lie strictly to the right of the effective bulk. 

Finally, pointwise in $z$, for each $\lambda$, because for $k\ge 3$ we have $\xi_1>\xi_0$ and $\frac{\xi}{\lambda \phi + S(z) \xi}$ is increasing in $\xi$, the solution to~\eqref{eq:G-matrix-outlier-1} occurs at larger $z$ than to~\eqref{eq:G-matrix-outlier-2} (when they exist), and $\lambda_{c,1}\le \lambda_{c,2}$. 
\end{proof}

\subsection{Spectral transitions over course of training}\label{sec:logistic-regression-over-training}
We now turn to the spectral transitions (i.e., emergence and splitting of outliers) that arise over the course of training by SGD. The main result here is an understanding of the drifts of the effective bulk and effective outliers over the first $\Omega(n)$ steps of the effective dynamics trajectory. In particular, we will focus on establishing Theorem~\ref{thm:effective-bulk-along-effective-dynamics}. 

In this subsection, we consider $\mathbf{G}_t= \lim_{d\to\infty}\mathbf{G}(\mathbf{x}_{\lfloor tn\rfloor},\bmu)$ where $\mathbf{x}_t$ is evolving according to the SGD~\eqref{eq:SGD-def}. Since the SGD over $\Omega(n)$ time-steps induces an evolution on the summary statistics (the top half of $\mathbf{G}$) as noted in~\eqref{eq:effective-dynamics}, this is well-defined.

\subsubsection*{The bulk does not change early in training} 

\begin{lem}\label{lem:bulk-doesn't-move}
    In the context of Theorem~\ref{thm:effective-bulk-along-effective-dynamics}, for any $\lambda,\phi$ and effective dynamics with $c_\eta\downarrow 0$, for every $z\in \mathbb C \setminus \supp(\nu_{\mathbf{G}_0}^H)$, we have 
    \begin{align*}
        \frac{d}{dt} S_{\nu_{\mathbf{G}_t}^H}(z) \restriction_{t=0} =0\,.
    \end{align*}
\end{lem}

\begin{proof}
As preparation for the calculations that will follow, it helps to compile some initial computations. We again work with the equal weight $p_c \equiv \frac{1}{k}$ case and with the Gram matrix of the means being $I_k$ to simplify computations. 
We use $S'(z)$ to denote the $z$ derivative of $S(z)$. 
Differentiating the bulk equation of~\eqref{e:defmz-with-rho} in $z$, using $\frac{d}{dx}\frac{ax}{c+ax} = \frac{ca}{(c+ax)^2}$, we get 
\begin{align*}
    S(z) + zS'(z) = \frac{\phi}{k} \sum_{j=1}^k p_j\mathbb E \Bigg[ \frac{\lambda \phi S'(z) f_j}{(\lambda\phi + S(z) f_j)^2} \Bigg]\,,
\end{align*}
where for readability we have dropped the $(\lambda^{-1/2} \vec{g},\mathbf{G})$ arguments from $f_j$ though they are implicit. 
Since $S'(z)\ne 0$, we can divide by it. Rearranging, we get the important identity, 
\begin{align}\label{e:important-identity}
    z + \frac{S(z)}{S'(z)} = \phi \mathbb E_{J,\vec{g}}\Bigg[\frac{\lambda \phi f_J}{(\lambda \phi + S(z) f_J)^2}\Bigg]\,,
\end{align}
where $\mathbb E_{J,\vec{g}}$ indicates the expectation is both over $J\sim \text{Unif}([k])$ and over $\vec{g}$. 
Notice that left-hand side here is positive because the right-hand side is ($f_J$ is positive almost surely, for any fixed $\mathbf{G}$).  

Implicitly differentiating the bulk equation~\eqref{e:defmz-with-rho} in time at $t>0$, using $\dot{h}$ to denote the time derivative of a function $h$, we get at each $z$,
\begin{align*}
    z\dot{S}(z)= \phi \mathbb E_{J,\vec{g}}\Bigg[\frac{\lambda \phi (\dot{S}(z) f_J + S(z) \dot f_J)}{(\lambda \phi + S(z) f_J)^2}\Bigg]\,.
\end{align*}
We use here that $\dot f_J$ is well-defined for $t>0$.
Plugging in the identity~\eqref{e:important-identity} and rearranging, we get 
\begin{align}\label{eq:time-derivative-bulk-equation}
   - \dot{S}(z) \frac{S(z)}{S'(z)} = \phi S(z) \mathbb E_{J,\vec{g}}\Bigg[\frac{ \lambda \phi \dot{f}_J}{(\lambda \phi + S(z) f_J)^2}\Bigg]\,.
\end{align}
Since $-S/S'>0$ to the right of the effective bulk, the quantity multiplying $\dot S$ is non-zero for $z$ to the right of the bulk support. In particular, $\dot S$ is well-defined and holomorphic on $\C_+$ for $t>0$, and extends continuously to the compliment of the support of $\nu_\bfG$. 

We now show that this is well-defined as $t\to0$. 
To this end recall the summary statistics evolution from~\cite[Theorem 5.7]{BGJH23} (using that in this subsection we have restricted to the case $\cC = k$ where the map from means to classes is one-to-one) and note that derivative at time of zero of  $\mathbf{G}_{t}$ is given by 
\begin{align}
    \partial_t \langle x^a_t, \mu_b\rangle\vert_{t=0} &  = \begin{cases} \frac{1}{k} (1-\frac{1}{k})=:\xi & a=b\\ -\frac{1}{k^2}= : \xi' & a\ne b
    \end{cases}\,,\label{eq:m-drifts}\\
    \partial_t \langle x_t^{a,\perp} ,x_t^{b,\perp} \rangle\vert_{t=0} & = \begin{cases} \frac{c_\eta p_a}{\lambda}& a=b \\ 0 & \text{else}
    \end{cases}\,, \label{eq:R-perp-drifts}
\end{align}
where we we use $x^{a,\perp}$ to denote the part of $x^a$ orthogonal to $\text{Span}(\mu_1,...,\mu_k)$. To evaluate 
    $\dot f_j$ fix $\bfx_t^1,...,\bfx_t^\cC,\mu_1,...,\mu_k\in \mathbb R^d$ differentiable such that $\partial_t\bfG(\bfx_t)\restriction_{t=0}=\bfG_t\restriction_{t=0}$ for some fixed $d$. If $c_\eta>0$ this can cause non-differentiability of the resulting $\bfx_t$ at $t=0$, but as we may take $c_\eta \to 0$ first, one can take, e.g., $x_t^a = \xi t \mu_a + \sum_{b\ne a} \xi' t \mu_b$. Then, by~\eqref{eq:f_a-kGMM-case}, in law
    \begin{align*}
        f_j(\lambda^{-1/2}\vec{g},\bfG_t) = [\varphi_\alpha (1-\varphi_\alpha) ] (\langle x^1_t,Y\rangle ,...,\langle x_t^\cC,Y\rangle) \qquad \text{for $t\ge 0$}
    \end{align*}
    where $Y \sim \mathcal N(\mu_j, I_d/\lambda)$, by differentiating, we get 
    \begin{align*}
        \dot f_j (\lambda^{-1/2} \vec{g};\mathbf{G}_t) = \langle \nabla \varphi_\alpha , (\langle \dot x_t^1,Y\rangle,...,\langle \dot x_t^\cC,Y\rangle )\rangle (1-2\varphi_\alpha)
    \end{align*}
    which by~\eqref{eq:m-drifts}--\eqref{eq:R-perp-drifts} with $c_\eta =0$, gives 
    \begin{align}\label{eq:partial_t-equation}
        \dot f_j (\lambda^{-1/2} \vec{g};\mathbf{G}_t) & = \langle \nabla \varphi_\alpha , \xi e_j + \xi'\sum_{l\ne j} e_l + \sum_l (\lambda^{-1/2}\sum_{m\in [k]} \xi \delta_{lm} g_{m} + \xi'(1-\delta_{lm})g_m)e_l )\rangle (1-2\varphi_\alpha)\nonumber \\ 
        & = \langle \nabla \varphi_\alpha , \xi e_j + \xi' \sum_{l\ne j} e_l + \sum_l \lambda^{-1/2} (\xi g_l + \sum_{m\ne l} \xi' g_m) e_l\rangle (1-2\varphi_\alpha)\,.
    \end{align}
    Note that at $t=0$, we have that $\varphi_\alpha= 1/k$, and $\nabla \varphi_\alpha= \varphi_\alpha e_\alpha - \varphi_\alpha \sum_{\beta} \varphi_\beta$. Thus taking the limit as first $c_\eta\to0$ then $t\to 0$ of \eqref{eq:time-derivative-bulk-equation}, we see that $\dot S$ is well defined at $t=0$. 
    
    Observe that at $\bfx = 0$,
$f_J = \xi$ where   
$\xi:= \frac{1}{k}(1-\frac{1}{k})$.  This $\dot S(z)=0$ at $t=0$ if we can show that $\mathbb E_{J,\vec{g}}[\dot{f}_J] =0$  there.
    
    Taking the expectation of \eqref{eq:partial_t-equation}, all the terms with $g_m$ in them become zero, leaving 
    \begin{align*}
        \mathbb E[\dot f_j\restriction_{t=0}] = \langle \frac{1}{k} e_\alpha - \frac{1}{k^2} \sum_\beta  e_\beta, \xi e_j + \xi' \sum_{l\ne j}e_l\rangle (1- \frac{2}{k}) = \frac{1}{k}(1-\frac{2}{k}) (\xi \delta_{\alpha j} + \xi'(1-\delta_{\alpha j}))\,,
    \end{align*}
    where we used that $\frac{1}{k} \xi + \frac{k-1}{k} \xi'=0$. 
Taking expected value of $J\sim \text{Unif}([k])$, again using $\frac{1}{k} \xi + \frac{k-1}{k} \xi'=0$, we get zero. Together, these imply $\mathbb E_{J,\vec{g}}[\dot{f}_J]=0$, in turn implying that the right-hand side of~\eqref{eq:time-derivative-bulk-equation} is zero, and therefore, at $t=0$, so is $\dot{S}(z)$ for $z$ to the right of the bulk support.  As $S$ and $\dot S$ is holomorphic on $\C_+$ continuous up to the boundary $\R\setminus \supp(\nu_\bfG)$, we have that for any $z\in\R \setminus \supp\nu_\bfG$, we may analytically extend it to a ball around $z$ by the Schwarz reflection principle. By the uniqueness principle for holomorphic functions, this implies that $\dot S =0$ on said ball, and thus on $\C_+$ as desired. 
\end{proof}

\subsubsection*{Separation of the $\mu_\alpha$ eigenvalue over training}

    We now examine the time derivative of the outlier solutions to show that the outlier corresponding to the mean $\alpha$ being learned grows infinitesimally in time. This implies that if $\lambda,\phi$ are such that there are outliers, this one effective outlier separates itself from the other $k-1$ effective outliers, and if $\lambda,\phi$ are such that there are marginally no outliers, one corresponding to $\mu_a$ will emerge from the bulk over the course of training. 

For this purpose, we let $z_* = z_*(t)$ be an effective outlier at $\mathbf{G} = \mathbf{G}_t$, i.e., a solution to $\det (z_* I - F(z_*))=0$ at $\mathbf{G} = \mathbf{G}_t$. In theory, we would implicitly differentiate $\det(z_* I -F)$ in $t$ to get an expression for $\dot z_*(t)$, but since that effective eigenvalue can have multiplicity $k$ at $t=0$ per Corollary~\ref{cor:exact-distribution-at-initialization} (that will then separate infinitesimally in time), this implicit differentiation is not well defined. Moreover, at $\mathbf{x}\equiv 0$ initialization, the fact that the Gram matrix is zero in its $[\cC,\cC]$ block causes problems for differentiating its square root $\sqrt{\mathbf{\bfG}}_t$. This is because outliers corresponding to orthogonal components of $\mathbf{x}_t$ will spontaneously appear rather than being perturbations around the $\mathbf{x}_0$ solutions. Still, we are interested in the evolution of the effective eigenvalues that correspond to the $k$-means being learned, and which are present even at initialization. This issue is partially resolved by taking the $c_\eta \downarrow 0$ limit of the effective dynamics as assumed in Theorem~\ref{thm:effective-bulk-along-effective-dynamics}. 

\begin{proof}[\textbf{\emph{Proof of Theorem~\ref{thm:effective-bulk-along-effective-dynamics}}}]
    The first bullet point was proved by Lemma~\ref{lem:bulk-doesn't-move}. Recall that under the summary statistic evolution~\eqref{eq:m-drifts}--\eqref{eq:R-perp-drifts} with $c_\eta \downarrow 0$, the arguments in $\dot f_j$ from~\eqref{eq:partial_t-equation} only consist of the Gaussians $\vec{g}_{[k]}$ and therefore the only order $t$ contribution to the set of outlier solutions must be from the $[k,k]$ block. 
    We are focusing on studying the effective outliers at short times perturbatively by Taylor expanding $F= F_{[k,k]}(z; \mathbf{G}_t)$ locally near $t=0$, and argue that we can read off the behavior of the $k$ possible solutions $z_*(t)$ from that. We begin by computing the set of solutions to the linear approximation to $F(z;\mathbf{G}_t)$ as, 
\begin{align}\label{eq:linear-approximation}
    \det(z I_k - (F_{[k,k]}(z;\mathbf{G}_0) + t\partial _t F_{[k,k]}(z;\mathbf{G}_0))=0\,,
\end{align}
where in defining $F$ and $\partial_t F_{[k,k]}$, we are using the consistent choice of $(l_i)_{i\in [k]} =(\mu_i)_{i=1}^k$. At the end of the proof, we will show that the error from higher order terms in the Taylor expansion to the effective outliers are $o(t)$, implying that the derivative at initialization is indeed captured by this linearization.  

By~\eqref{eq:F-matrix}, the 
derivative 
$\partial_t F\restriction_{t=0}$ satisfies 
\begin{align}\label{eq:F-time-derivative-1}
    \partial_t F_{[k,k]} (z, \mathbf{G}_t) = \lambda \phi \mathbb E_{J,\vec{g}} \Bigg[ \frac{\lambda \phi \partial_t f_J}{(\lambda \phi + S(z) f_J)^2} (\frac{\vec{g}_{[k]}}{\sqrt\lambda} + e_J)^{\otimes 2}\Bigg]\,.
\end{align}
Firstly, when we go to evaluate this at $t=0$, notice that $f_j = \frac{1}{k}(1-\frac{1}{k}) = \xi$ for all $j$. Now plugging in the expression from~\eqref{eq:partial_t-equation}, and evaluate at $t=0$, we end up with 
\begin{align*}
    \frac{\lambda \phi(1-2/k)}{(\lambda \phi + S(z) \xi)^2}  \mathbb E_{J,\vec{g}}\Bigg[\langle \nabla \varphi_\alpha , \xi e_J + \xi' \sum_{l\ne J} e_l + \sum_l e_l \lambda^{-1/2} (\xi g_l + \sum_{m\ne l} \xi' g_m) \rangle (\lambda^{-1/2} \vec{g}_{[k]}+ e_J)^{\otimes 2} \Big]\,.
\end{align*}Plugging in for $\nabla \varphi_\alpha =  \frac{1}{k} ( e_\alpha - \sum_{\beta \ne \alpha} e_\beta)$ at $t=0$, and using that $\sum_{\beta} (\xi g_\beta + \xi' \sum_{m\ne \beta} g_m) =0$ because $\xi + (k-1)\xi' =0$, this reduces to 
\begin{align}\label{eq:partial-t-F-equation}
     \frac{\lambda \phi(1-2/k)}{(\lambda \phi + S(z) \xi)^2} \mathbb E_{J,\vec{g}}\Big[ \big(\xi \delta_{\alpha J} + \xi'(1-\delta_{\alpha J}) + \lambda^{-1/2} (\xi g_\alpha + \xi' \sum_{\beta \ne \alpha} g_\beta)\big) \big(\lambda^{-1/2} \vec{g}_{[k]} + e_J)^{\otimes 2}\Big]\,.
\end{align}
Starting with the contribution from $$\mathbb E[ (\xi\delta_{\alpha,J} + \xi'(1-\delta_{\alpha,J}) )(\lambda^{-1/2} \vec{g}_{[k]} +e_J)^{\otimes 2}]\,,$$
     and expanding the $\otimes 2$, the first term gives zero due to the average over $J$ as $\sum_{j} (\xi \delta_{\alpha J} + \xi'(1-\delta_{\alpha,J})) =0$. The cross terms are zero due to the average over $\vec{g}$. The only non-zero term comes from 
    \begin{align}\label{eq:no-g-interaction-term}
        \mathbb E[ (\xi\delta_{\alpha,J} + \xi'(1-\delta_{\alpha,J}) )e_J^{\otimes 2}] = \frac{1}{k} \diag(\xi, \xi',...,\xi')\,,
    \end{align}
    where the $\xi$ is in the $\alpha$-slot. The other term to consider from~\eqref{eq:partial-t-F-equation} is 
    \begin{align*}
        \lambda^{-1/2} \mathbb E_{J,\vec{g}}[ (\xi g_\alpha + \xi' \sum_{\beta \ne \alpha} g_\beta) (\lambda^{-1/2} \vec{g}_{[k]} + e_J)^{\otimes 2}]\,.
    \end{align*}
    Expanding the $\otimes 2$, all terms with an odd number of Gaussians vanish, leaving only 
    \begin{align*}
        \frac{1}{\lambda k} \mathbb E\Big[(\xi g_\alpha + \xi' \sum_{\beta \ne \alpha}g_\beta) (\vec{g}_{[k]} \otimes \mathbf{1}_{[k]}+ \mathbf{1}_{[k]}\otimes \vec{g}_{[k]})\Big] = \frac{1}{\lambda k}\left(\begin{matrix}
           2 \xi  &  \xi + \xi' & \cdots&  \xi+ \xi'   \\ \xi + \xi' & 2\xi' & \cdots & 2\xi' \\ \vdots  & \vdots & \ddots & \vdots\\ 
           \xi + \xi' & 2\xi' & \cdots & 2\xi'
        \end{matrix}\right)\,,
    \end{align*}
    where we are using that without loss of generality $\alpha=1$ and so the first row/column is the $\alpha$'th one. 
    Putting all the above together, we are therefore looking for solutions of 
    \begin{align*}
        \det\big(zI_k - F_{[k,k]}(z;\mathbf{G}_0) - t \psi(z)\mathcal M\big) =0\,,
    \end{align*}
    where 
    \begin{align}\label{eq:psi-def}
        \psi(z) = \frac{\lambda \phi (k-2)}{k^3 (\lambda \phi + S(z) \xi)^2} \,,
    \end{align}
    and 
    \begin{align}
        \mathcal M =  \diag(\xi,\xi',...,\xi') + \frac{1}{\lambda}\left(\begin{matrix}
           2 \xi  &  \xi + \xi' & \cdots&  \xi+ \xi'   \\ \xi + \xi' & 2\xi' & \cdots & 2\xi' \\ \vdots  & \vdots & \ddots & \vdots\\ 
           \xi + \xi' & 2\xi' & \cdots & 2\xi'
        \end{matrix}\right)\,.
    \end{align}
    Recalling $F$ from~\eqref{eq:F-at-0-initialization}, (and multiplying through by $S(z)$ as it is non-zero) we are looking at $z$'s such that the determinant of the following matrix is zero: 
    \begin{align*}
           \big(1 + \frac{\lambda \phi \xi S(z)}{k(\lambda \phi + \xi S(z))} + t\psi(z) S(z)\xi'\big) I_k  + \frac{1}{\lambda}\psi(z)S(z) \left(\begin{matrix}
           2 \xi + \lambda (\xi - \xi')&  \xi + \xi' & \cdots&  \xi+ \xi'   \\ \xi + \xi' & 2\xi' & \cdots & 2\xi' \\ \vdots  & \vdots & \ddots & \vdots\\ 
           \xi + \xi' & 2\xi' & \cdots & 2\xi'
        \end{matrix}\right)\,.
    \end{align*}
    That is, $z$'s such that some eigenvalue of the matrix equals to zero. Call $\mathbf{b}(z)$ the first two terms multiplying the identity here. Explicit linear algebra gives the solutions by   
    \begin{align}\label{eq:the-decreasing-solutions}
        \mathbf{b}(z) + t\psi(z) S(z)\xi' =0 \qquad \text{with multiplicity $k-2$}\,,
    \end{align}
    or otherwise, by solutions to the following two equations with multiplicity one each: 
    \begin{align*}
        \mathbf{b}(z) + t\psi(z) S(z)\Big(\xi'  + \frac{1}{2\lambda} \Big(a + (k-1)c \pm \sqrt{(a- (k-1)c)^2 + 4(k-1)b^2 } \Big)\Big)=0\,,
    \end{align*}
    where we have written  $
        a = 2\xi + \lambda(\xi - \xi') $, $ b = \xi+\xi'$, and $c=2\xi'$. 
    Now observe that $a+(k-1)c = \lambda(\xi - \xi')$, and $a-(k-1)c = 4\xi + \lambda(\xi - \xi')$. Therefore the two solutions (each of multiplicity one) are given by 
    \begin{align*}
        \mathbf{b}(z) + t\psi(z)S(z)\mathfrak{c}_+ =0\,, \qquad \text{and} \qquad \mathbf{b}(z) + t\psi(z)S(z)\mathfrak{c}_- =0\,,
    \end{align*}
    where $$\mathfrak{c}_+ \ge \xi' + \frac{1}{2\lambda}(4\xi + 2\lambda (\xi - \xi')) \ge \xi\,, \qquad \text{and} \qquad \mathfrak c_- \le \xi' <0\,.$$ 
    Notice that $\mathbf{b}(z) =0$ is exactly the initialization outlier equation from~\eqref{eq:zero-init-outlier-equation}, and its solution is the unique solution of an increasing function in $z$ equaling zero. The additional term $t\psi(z)S(z)\zeta$ which arises for $\zeta \in \{\xi',\mathfrak{c}_+,\mathfrak{c}_-\}$ takes the opposite sign of $\zeta$ for $z$ to the right of the support of the initialization effective bulk, because $S(z)$ is negative there, and $\psi(z)$ is positive everywhere. 

    All told, therefore, the solution to $\mathbf{b}(z) + t\psi(z) S(z) \zeta =0$ is to the right of the time zero effective outlier by an order $t$ amount when $\zeta = \mathfrak{c}_+$ (as the zero moves to the right when adding a negative function to an increasing function) and to the left of the time zero effective outlier by an order $t$ amount when $\zeta \in \{\xi', \mathfrak c_-\}$.

    It remains to establish that this implies the effective eigenvectors at time $t$ are within $o(t)$ of the ones solving~\eqref{eq:linear-approximation} to conclude the proof. By differentiating~\eqref{eq:F-time-derivative-1} and using Cauchy--Schwarz and the moment estimates of Lemma~\ref{l:Sd_est}, for every $z$ at least $\epsilon$ away from the effective bulk, the difference between $F(t)$ and the linear approximation is at most $C_\epsilon t^2$.  The following claim translates this to control on the error on the solutions of the effective eigenvalue equations.

\begin{claim}\label{clm:stability-of-z-F-solutions}
    Suppose $\tilde z_{1}(t),...,\tilde z_{k}(t)\in \mathbb R$ are the solutions to~\eqref{eq:linear-approximation} and $z_{*,1}(t),...,z_{*,k}(t)$ are  solutions of $det(zI_k - F(t))$ for $F(t)$ that is entrywise at most $o(t)$ away from $F_{[k,k]}(\mathbf{G}_0) + t\partial_t F_{[k,k]}(\mathbf{G}_0)$. Then each solution $z_{*,i}$ is within $o(t)$ of a solution $\tilde z_i$. 
\end{claim}

\begin{proof}
    By analyticity of $\det(z I  - F)$ in $z$ and the entries of $F$, if $F'$ is within $O(\epsilon)$ of $F$ entry-wise, then for each $z$, we have that the roots of the polynomial $\det(z I - F')$ are within a ball of radius $O(\epsilon)$ (in $\mathbb C$) of those of $\det(z I - F)$. 
\end{proof}

The proof is concluded from the observation that at $\lambda>\lambda_c^H$, at $t=0$, the effective eigenvalues are strictly bounded away from the effective bulk per Lemma~\ref{cor:exact-distribution-at-initialization}. Then, the solutions $z_{*,1}(t),...,z_{*,k}(t)$ not only are within $o(t)$ in $\mathbb C$ of the solutions to~\eqref{eq:linear-approximation}, they are also real because they are eigenvalues of a real symmetric matrix. When we take the derivative, the $o(t)$ error vanishes, leaving the same sign on the derivative as we obtained for solutions to~\eqref{eq:linear-approximation}. 
\end{proof}

\subsection{Existence of outliers throughout summary statistic space}
The last concrete consequence to establish is Corollary~\ref{cor:BBP-at-all-points}, that at \emph{every} point $\mathbf{x}$ of fixed summary statistic values $\mathbf{G}$, at $\lambda$ sufficiently large, there exists at least one outlier eigenvalue. Again, for concreteness, we take the means $\mu_1,...,\mu_k$ to be orthonormal. 

Our reasoning goes by perturbing the effective bulk and effective outliers about their $\lambda = \infty$ solutions. Along the typical SGD trajectory after a burn-in period, this type of result was established in~\cite{BGJH23} for $\phi$ large; the main distinctions here are that the following lemma applies at \emph{every} point in parameter space, and holds for any sample complexity $\phi>0$ (only $\lambda$ needs to be large).  

\begin{proof}[\textbf{\emph{Proof of Corollary~\ref{cor:BBP-at-all-points}}}]
    Let us consider the $\lambda = \infty$ limit of the effective bulk, which we can probe by examining the effective bulk equation~\eqref{e:defmz-with-rho} as $\lambda \to\infty$. Recall from Proposition~\ref{p:free-mult-conv-wishart} that the solution to~\eqref{e:defmz-with-rho} is also the limiting bulk distribution for matrices of the form $\frac{1}{n}W D W^\top$ where $W$ is formed out of independent Gaussians of variance $I_d/\lambda$, and $D$ is as before. Then, we can bound the spectral norm of $\frac{1}{n} W D W^\top$ by $\frac{1}{n}\|WW^\top \| \|D\|$. Bounding $\|D\|$ by the maximum of the support of $D_{\ell\ell}$ (which is $1$ in the logistic regression case), and bounding $\frac{1}{n}\|WW^\top\| = O(1/\lambda)$ as it is a Wishart matrix with  explicit Marchenko--Pastur limit, we see that the spectral norm of the effective bulk is eventually almost surely at most $C/\lambda$ for some $C$ that only depends on $\phi$.

We next show that the effective outliers stabilize as $\lambda\to\infty$, to values of $z$ that are strictly positive (uniformly in $\lambda$). Let us expand the matrix $F$~\eqref{eq:F-matrix} in the regime of $\lambda$ large. By the asymptotics above on $f_j^H$, expanding the $\otimes 2$ in $F$, we find that
\begin{align*}
	F =   \lambda \phi \sum_{j} p_j \Bigg(& \mathbb E\Bigg[ \frac{f_j^\infty}{\lambda \phi + S_\lambda(z) f_j^\infty} \Big( 1+ O(\lambda^{-1/2} \|\vec{g}\|)\Big)\Bigg]  e_j^{\otimes 2} \\
    & +  \mathbb E\Bigg[\frac{f_j}{\lambda \phi + S_\lambda(z) f_j} \Big( \lambda^{-1/2} \text{Sym}(\vec{g} \otimes \sqrt{\mathbf G}_{[q]j}) + \lambda^{-1} \vec{g}^{\otimes 2}\Big)\Bigg]\Bigg)\,,
\end{align*}
where the subscript $\lambda$ is to clarify the $\lambda$ dependence.
Using that $f_j^\infty$ is deterministic for the first term, and Cauchy--Schwarz and~\eqref{l:Sd_est} for the latter two terms, for $z\notin \supp(\nu_{\mathbf{G}}^H)$, we get       
\begin{align*}
	F = \lambda \phi \sum_j p_j \frac{f_j^\infty}{\lambda \phi + S_\lambda(z) f_j^\infty} \big(e_j^{\otimes 2} + O(\lambda^{-1/2})\big) = \sum_j p_j f_j^\infty \Big( 1- \frac{S_\lambda(z)f_j^\infty}{\lambda \phi} + \cdots \Big) \big(e_j^{\otimes 2} + O(\lambda^{-1/2})\big) \,.
\end{align*} 
Therefore, $k$ of the solutions to $\det(z I_q - F) = 0$ are close (as $\lambda \to\infty$) to the solutions to 
\begin{align*}
	\det (z I_q -  \sum_j p_j f_j^\infty e_j^{\otimes 2})=0
\end{align*}
which are clearly given by $(p_j f_j^\infty)_{j=1}^{k}$. Now recall again that for any $\mathbf{G}$, $f_j^\infty$ is strictly positive (independent of $\lambda$) for every $j\in [k]$. Since the right-edge of $\supp(\nu_{\mathbf{G}}^{H})$ goes to $0$ as $\lambda \to \infty$, these $k$ solutions will all exist for $\lambda$ sufficiently large. Since all we used about $f_j^H$ was boundedness of $f_j^\infty$ away from zero for $\lambda$ large, the same holds true for the Gradient matrix.
\end{proof}

\section{Other examples for which spectral characterization applies}\label{sec:other-examples}
In this section, we demonstrate how the Hessian and Gradient matrices for other well-studied problems described in Section~\ref{subsec:further-applications} fit into the framework of Theorem~\ref{thm:main-A-D-A^T}--\ref{thm:main-A-D-A^T-outliers}.

\subsection{First layer of $k$-GMM classification with two-layer network}

In this subsection, we verify that if we perform the $k$-GMM classification task with a multi-layer network (as would be required if the classes are not linearly separable, for instance as exemplified by the well-known example of XOR data points) of fixed intermediate layer width, then the first layer Hessian and Gradient matrix fall into our framework and Theorems~\ref{thm:main-A-D-A^T}--\ref{thm:main-A-D-A^T-limit} give their effective bulk distribution and outlier locations.  

\begin{proof}[\textbf{\emph{Proof of Corollary~\ref{cor:multi-layer-GMM}}}]
We begin with writing the derivatives of the loss function~\eqref{eq:loss-function-2-layer} to demonstrate that the diagonal matrix in the Hessian and Gradient matrices indeed only depends on $Y_\ell$ through its inner products with $\mathbf{W}^\alpha$ and the means $\mu_1,...,\mu_k$.  

Differentiating the loss function from~\eqref{eq:loss-function-2-layer} 
 in $W_i^\alpha \in \mathbb R^d$, gives the gradient in the first layer weights, 
\begin{align*}
    \nabla_{W_i^\alpha} L =  v^\alpha_i  g'(W^\alpha_i \cdot Y)  \Big(- y_\alpha  + \frac{ e^{v^\alpha \cdot g(\mathbf{W}^\alpha Y)}}{\sum_{c\in \cC}e^{v^c\cdot g(\mathbf{W}^c Y)}}\Big)  Y\,.
\end{align*}
Letting 
\begin{align}\label{eq:yhat}
\widehat y_\alpha = \frac{ e^{v^\alpha \cdot g(\mathbf{W}^\alpha Y)}}{\sum_{c\in \cC}e^{v^c\cdot g(\mathbf{W}^c Y)}}
\end{align}
this can be compactly expressed as 
\begin{align*}
     \nabla_{W_i^\alpha} L = - v^\alpha_i  g'(W^\alpha_i \cdot Y) (y_\alpha - \widehat y_\alpha )  Y\,,
\end{align*}
Differentiating again, we get the first-layer Hessian 
\begin{align*}
    \nabla_{W_i^\alpha  W_j^\beta}^2 L =\Big( - \delta_{ij} \delta_{\alpha \beta} v^\alpha_i  g''(W^\alpha_i \cdot Y)(y_\alpha -\widehat y_\alpha )    & + \delta_{ij} \delta_{\alpha\beta} (v^\alpha_i  g'(W^\alpha_i \cdot Y))^2  \widehat y_\alpha \\
    & \qquad \qquad - v^\alpha_i g'(W_i^\alpha \cdot Y) v^\beta_j g'(W_j^\beta \cdot Y)\widehat y_\alpha \widehat y_\beta\Big)Y^{\otimes 2}\,,
\end{align*}
and tensoring the gradient with itself, the first-layer Gradient matrix 
\begin{align*}
    \nabla_{W_i^\alpha} L \otimes \nabla_{W_j^\beta} L= v_i^\alpha v_j^\beta g'(W_i^\alpha \cdot Y)g'(W_j^\beta \cdot Y) (y_\alpha - \widehat y_\alpha)(y_\beta - \widehat y_\beta)Y^{\otimes 2}\,. 
\end{align*}
In particular, the on-diagonal blocks (which are the ones of most interest) look like 
\begin{align}
 \nabla_{W_i^\alpha} L^{\otimes 2} & = \Big( (v_{i}^\alpha g'(W_i^\alpha \cdot  Y) (y_\alpha -\widehat y_\alpha)\Big)^2 Y^{\otimes 2} \label{eq:multilayer-GMM-Gradient-matrix} \\
    \nabla_{W_i^\alpha}^2 L &  =\Big( - v^\alpha_i  g''(W^\alpha_i \cdot Y)(y_\alpha -\widehat y_\alpha )  +  (v^\alpha_i  g'(W^\alpha_i \cdot Y))^2  \widehat y_\alpha(1-\widehat y_\alpha)\Big)Y^{\otimes 2} \label{eq:multilayer-GMM-Hessian-matrix} \,.
\end{align}

Since the coefficients of $Y^{\otimes 2}$ only depend on inner products of $Y\sim \mathcal P_Y$ with the vectors of $\mathbf{W} = (W_i^c)_{c\in [\cC],i\in [K]}$, it is evident that this can be written as a function of $K\cC + k$ independent Gaussians, together with the $K\cC$ quantities $(v_i^\alpha)$ and the inner products of $(W_i^c)_{i,c}$ with each other and with the mean vectors $\mu_1,...,\mu_k$. In particular, for any fixed second-layer weights $(v_i^\alpha)_{i\in [K]}$, the quantity multiplying $Y^{\otimes 2}$ depends on $((\mathbf{W}_{i}^c)_{i\in[K],c\in [\cC]},(\mu_j)_{j\in [k]})$ through their summary statistic matrix $\mathbf{G}$. This yields that Assumption~\ref{assumption:meas} holds. 

In order to justify applicability of Theorems~\ref{thm:main-A-D-A^T}--\ref{thm:main-A-D-A^T-outliers}, it remains to verify Assumption~\ref{assumption:D}. Since $y_\alpha,\widehat y_\alpha \in [0,1]$, as long as $g'$ and $g''$ are bounded functions, for fixed values of $v$ we get that the coefficients of $Y^{\otimes 2}$ are uniformly compactly supported on $\mathbb R$. Next consider the probability that the coefficient of $Y^{\otimes 2}$ in~\eqref{eq:multilayer-GMM-Gradient-matrix}--\eqref{eq:multilayer-GMM-Hessian-matrix} are less than $\epsilon$ in absolute value. 

For~\eqref{eq:multilayer-GMM-Gradient-matrix}, for $v_i^\alpha \ne 0$, this is bounded by the probability that $|y_\alpha - \widehat y_\alpha |\le \epsilon^{1/4}$ plus the probability that $|g'(W_i^\alpha Y)|\le \epsilon^{1/4}$. The former probability reduces to that $\widehat y_\alpha$ is within $\epsilon^{1/4}$ to $0$ or $1$, which  is at most $O(\epsilon)$ so long as $g$ has polynomial growth at $\pm \infty$ by the argument in the proof of Theorem~\ref{mainthm:logistic-regression}, around~\eqref{eq:softmax-prob-of-near-zero}. The latter event has probability at most $O(\epsilon^{1/4})$ because of Assumption~\ref{assumption:Gaussian-non-degeneracy} on $g'$, and the assumption that $W_i^\alpha \ne 0$. 

The argument for the Hessian case,~\eqref{eq:multilayer-GMM-Hessian-matrix} is similar, with the differences being that the probability of $|\widehat y_\alpha(1 - \widehat y_\alpha)|\le \epsilon^{1/4}$ also per~\eqref{eq:softmax-prob-of-near-zero}, and then asking not only that $g''$ and $g'$ do not put too much mass near zero, but that they do not put too much mass near any single number (as that could lead to a cancellation between the two terms which shifts the mass near zero). This is again provided by Assumption~\ref{assumption:Gaussian-non-degeneracy} we make on $g', g''$ because if they have density and $W_i^\alpha\ne 0$, the probability that each falls within $\epsilon$ of a point $x \in \mathbb R$ is at most $O(\epsilon)$.  

Finally, if $\|\bfG^{(d)} -\bfG\|= o(1/\log d)$, then by Lemma~\ref{lem:bounded-differentiable-implies-strong-convergence}, one gets Assumption~\ref{assumption:strong-convergence} if the coefficients in~\eqref{eq:multilayer-GMM-Hessian-matrix}--\eqref{eq:multilayer-GMM-Gradient-matrix} are uniformly continuous as functions of the inner products $(\mathbf{W}_i^c Y)_{i,c}$; this evidently holds if the coefficients have bounded derivatives, which follows if $g$ is thrice bounded differentiable (since the softmax functions are clearly infinitely many times bounded differentiable). 
\end{proof}

\subsection{Second layer of $k$-GMM classification with two-layer network}
In this subsection, we verify Lemma~\ref{lem:second-layer-only-summary-statistic}, that even though the second layer Hessian/Gradient matrix in the fixed-width regime we consider is a $K\times K$ matrix in the $n,d\to\infty$ limit, this random matrix's law only depends on the parameters of the network through their summary statistic values. 

\begin{proof}[\textbf{\emph{Proof of Lemma~\ref{lem:second-layer-only-summary-statistic}}}]
   Differentiating the loss function from~\eqref{eq:loss-function-2-layer} in $v= (v^\alpha)_{\alpha \in [K]}$, gives the gradient in the second layer weights, 
   \begin{align*}
       \nabla_{v^\alpha} L = g(\mathbf{W}^\alpha Y) \Big( - y_\alpha + \frac{e^{ v^\alpha \cdot g(\mathbf{W}^\alpha Y)}}{\sum_{c\in [\cC] }e^{ v^c \cdot g(\mathbf{W}^c\cdot Y)}}\Big)\,.
   \end{align*}
   Let $\widehat{y}_{\alpha}$ be as in~\eqref{eq:yhat} to express this as $\nabla_{v^\alpha} L = - g(\mathbf{W}^\alpha Y)(y_\alpha - \widehat{y}_\alpha)$. Differentiating in $v^\alpha$ a second time (focusing only on the diagonal second layer blocks, though the same would hold for the second layer off-diagonal blocks), 
   \begin{align*}
       \nabla_{v^\alpha v^\alpha} L = \hat y_\alpha (1-\hat y_\alpha) g(\mathbf{W}^\alpha Y)^{\otimes 2}  \qquad \text{and} \qquad \nabla_{v^\alpha} L^{\otimes 2} = (y_\alpha - \widehat y_\alpha)^2 g(\mathbf{W}^\alpha Y)^{\otimes 2}\,.
   \end{align*}
   Evidently, this only depends on $\mathbf{W} = (\mathbf{W}^{\alpha}_{i})_{\alpha\in [\cC],i\in [K]}$ through inner products with $Y$ and therefore its law only depends on $\mathbf{W}$ through its Gram matrix, and its inner products with the hidden class means. 
   
   To push this to the $n,d\to\infty$ limit and be able to replace the dependency on $\mathbf{G}^{(d)}$ with the limiting summary statistic matrix $\mathbf{G}$, it suffices to observe continuity of the quantities $\nabla_{v^\alpha v^\alpha}$ and $\nabla_{v^\alpha} L^{\otimes 2}$ as functions of the summary statistics $\mathbf{G}$. This follows from our assumption that $g\in C_b^2$ and the continuity of the softmax function appearing in $y$ and $\widehat y$. 
\end{proof}

\subsection{Parametric regression for multi-index model}
We next move to the proof of Corollary~\ref{cor:multi-index-regression}, showing the applicability of Theorems~\ref{thm:main-A-D-A^T}--\ref{thm:main-A-D-A^T-limit} to the loss function of~\eqref{eq:multi-index-regression-loss}.  

\begin{proof}[\textbf{\emph{Proof of Corollary~\ref{cor:multi-index-regression}}}]
Let us first make some gradient and Hessian calculations. We are interested in the matrices' blocks in one of the $k$ chunks of $\mathbf{X} = (\mathbf X^1,...,\mathbf{X}^k)$, where each $\mathbf{X}^c\in \mathbb R^d$.   
\begin{align*}
    (\nabla_{c} L)^{\otimes 2}(\mathbf{X}) & = \Big(2(g(\mathbf{X} Y) - y) \partial_c g(\mathbf{X}Y) \Big)^2 Y^{\otimes 2}\,,\\
    \nabla^2_{cc} L (\mathbf{X})& =  \Big( 2\partial_c g(\mathbf{X} Y )^2 + 2(g(\mathbf{X} Y)- y) \partial_{cc}g(\mathbf{X} Y)\Big) Y^{\otimes 2}\,.
    \end{align*}
Taking the empirical Hessian or empirical Gradient matrix, this is of the form of~\eqref{eq:A-D-A^t} where $Y \sim \mathcal N(0, I_d)$ (which is a trivial case of the GMM distribution where the mean is zero and the signal-to-noise is $1$). Notice that $y$ only depends on $Y$ through its inner products with $\Theta_*$, and the other quantities appearing in the multiplier of $Y^{\otimes 2}$ in $(\nabla_{c} L)^{\otimes 2}$ and $\nabla_{cc}^2 L$ only depend on $Y$ through its inner products with $\mathbf{X}$. Indeed, if we form the Gram matrix of summary statistics out of $\mathbf{G}= (\mathbf{X},\Theta)^{\otimes 2}$ then in each of the Gradient matrix and Hessian cases above, the coefficient of $Y^{\otimes 2}$ is equal in law to a function $f(\vec{g},\mathbf{G})$ for a $q=2k$-dimensional i.i.d.\ Gaussian vector $\vec{g}$. Thus, Assumption~\ref{assumption:meas} holds. 

Turning to Assumption~\ref{assumption:D}, we first reason that under our assumptions on $g$, the support of the law of $f(\vec{g},\mathbf{G})$ is compact in $\mathbb R$. This is immediate if $g, \nabla g, \nabla^2 g$ are all bounded on $\mathbb R^k$. The other thing to check is it's behavior near zero, and in particular to show that the probability of being less than $\epsilon$ is $o(1)$ in $\epsilon$. It is sufficient to establish that the coefficient (which is a continuous function of $2k$ i.i.d.\ Gaussian random variables) has density at zero. Since $g$ satisfies the non-degeneracy assumption of Assumption~\ref{assumption:Gaussian-non-degeneracy}, and $\mathbf{G}$ is full rank, $g(\mathbf{X}Y) -y$ is non-zero almost surely. 
In turn, it is sufficient for each of $g$ and $\partial_c g$  to have density everywhere on $\mathbb R$ when applied to either of $\mathbf{X}$ or ${\Theta}_*$. This holds given Assumption~\ref{assumption:Gaussian-non-degeneracy} holds for each of $h\in \{g, \partial_c g, \partial_{cc}g\}$ so long as the laws of $\mathbf{X} Y, \Theta Y$ have full support, which requires $\mathbf{G}$ to be full rank. 

Last, if $\|\bfG^{(d)} -\bfG\|= o(1/\log d)$, then by Lemma~\ref{lem:bounded-differentiable-implies-strong-convergence}, one gets Assumption~\ref{assumption:strong-convergence} if the coefficients of $Y^{\otimes 2}$ above are uniformly continuous as functions of the inner products $\mathbf{X}Y, \Theta_* Y$; this evidently holds if the coefficients have bounded derivatives, which follows if $g$ is thrice bounded differentiable. 
\end{proof}

\subsection{First layer of two-layer learning for multi-index model}

Our last example to which Theorems~\ref{thm:main-A-D-A^T}--\ref{thm:main-A-D-A^T-outliers} apply is the first layer Hessian and Gradient matrix of learning multi-index models with two layer networks.   

\begin{proof}[\textbf{\emph{Proof of Corollary~\ref{cor:multi-layer-multi-index}}}]
Differentiating the loss in the $W_c$ direction for $c\in \{1,...,K\}$, we get 
\begin{align*}
    (\nabla_{W_c} L)^{\otimes 2} & = 4v_c^2 \sigma'(W_c Y)^2(h_{\NN}(Y) - h_*(Y))^2 \, Y^{\otimes 2}\,, \\ 
    \nabla_{W_cW_c}L  & = 2v_c \Big( \sigma''(W_c Y) (h_{\NN} (Y) - h_*(Y)) + v_c \sigma'(W_c Y)^2\Big) \, Y^{\otimes 2}\,.
\end{align*}
We claim that it fits into the paper's general framework with $q = K + k$ as the way the coefficients of $Y^{\otimes 2}$   depend on the data $Y$ is only through its projections on the $q$ vectors, $(W_c)_{c=1}^K$ and $\Theta_* = (\Theta_*^1,...,\Theta_*^k)$. This implies that Assumption~\ref{assumption:meas} holds. 

Turning to Assumption~\ref{assumption:D}, similar to the proofs in the preceding examples, for boundedness of the support of the coefficient of $Y^{\otimes 2}$, it is sufficient for $g$ to be bounded and $\sigma$ to be in $C^2_b(\mathbb R)$. For the second part of the assumption, the $(h_{\NN}(Y) - h_*(Y))$ part is at least $\epsilon$ away from zero with probability $1-o_\epsilon(1)$. It is sufficient to show that the laws of the coefficients to have bounded density everywhere on $\mathbb R$, which in this case would follow (presuming that $v_c \ne 0$) if they satisfy the non-degeneracy condition of Assumption~\ref{assumption:Gaussian-non-degeneracy}, and $\mathbf{W}$'s Gram matrix is of full rank (as that guarantees equivalence of absolute continuity when applied to $Z$ and when applied to $\mathbf{W} Y$. 

Finally, if $\|\bfG^{(d)} -\bfG\|= o(1/\log d)$, then by Lemma~\ref{lem:bounded-differentiable-implies-strong-convergence}, one gets Assumption~\ref{assumption:strong-convergence} if the coefficients of $Y^{\otimes 2}$ above are uniformly continuous as functions of the inner products $\mathbf{W}Y, \Theta_* Y$; this evidently holds if the coefficients have bounded derivatives, which follows if the link function $g$ is bounded differentiable, so that $h_{\NN}$ is, and if the activation $\sigma$ is thrice bounded differentiable. 
\end{proof}

\section{Supplementary figures}\label{sec:supplementary-figures}
We collect some additional figures further demonstrating the theorems in cases where test or train data is used to generate the empirical matrices. 
\vspace{-.1cm}
% \begin{figure}[h]\label{fig:Hessian-spectrum-evolution-train}
% \includegraphics[width = .32\textwidth]{Figures/train-hessian0phi=4snr=6D=20000.png}
% \includegraphics[width = .32\textwidth]{Figures/train-hessian16000phi=4snr=6D=20000.png}
% \includegraphics[width = .32\textwidth]{Figures/train-hessian48000phi=4snr=6D=20000.png}
%     \caption{Analogue of Figure~\ref{fig:Hessian-spectrum-evolution-test} with train data rather than test data.}
% \end{figure}

\vspace{-.1cm}
\begin{figure}[h]
\includegraphics[width = .32\textwidth]{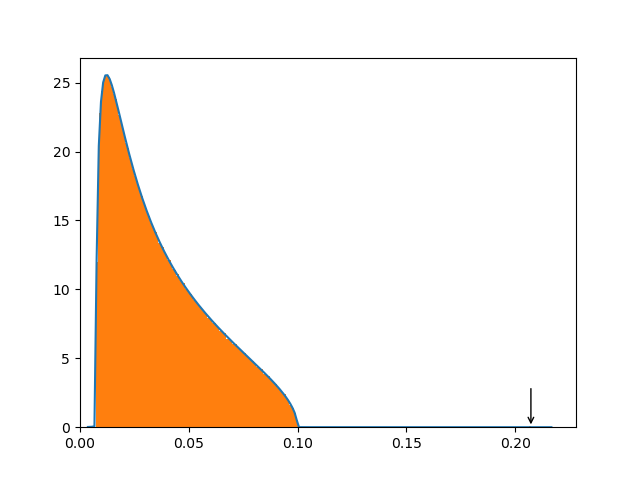}
\includegraphics[width = .32\textwidth]{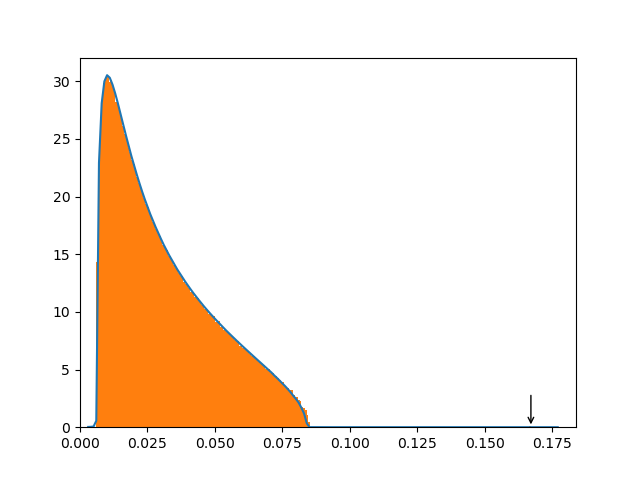}
\includegraphics[width = .32\textwidth]{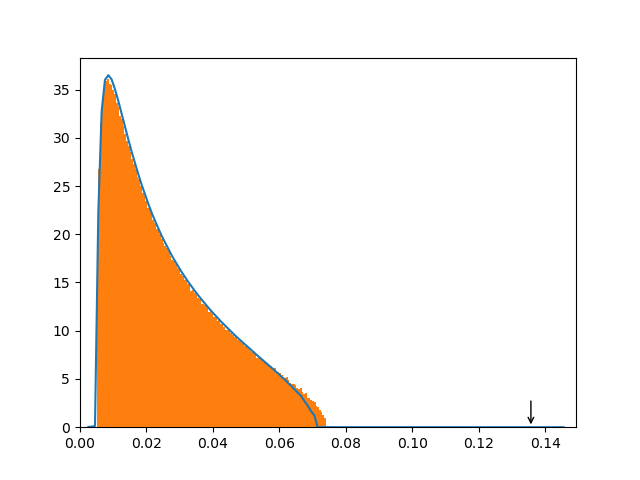}
    \caption{Analogue of Figure~\ref{fig:Hessian-spectrum-evolution-test}  with the empirical Gradient matrix rather than empirical Hessian.}\label{fig:G-spectrum-evolution-test}
\end{figure}

\vspace{-.2cm}
\begin{figure}[h]
\includegraphics[width = .32\textwidth]{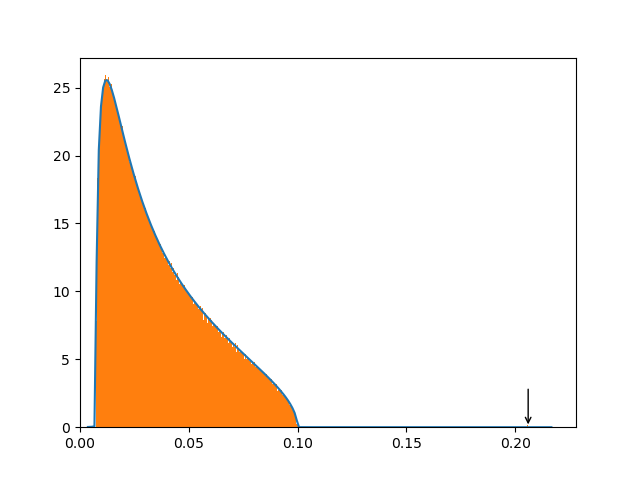}
\includegraphics[width = .32\textwidth]{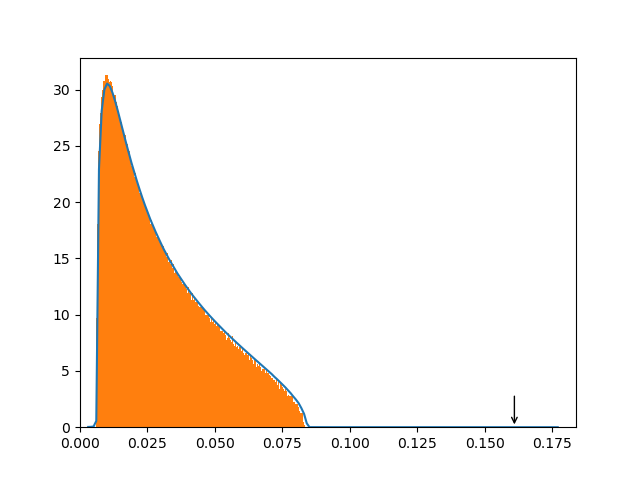}
\includegraphics[width = .32\textwidth]{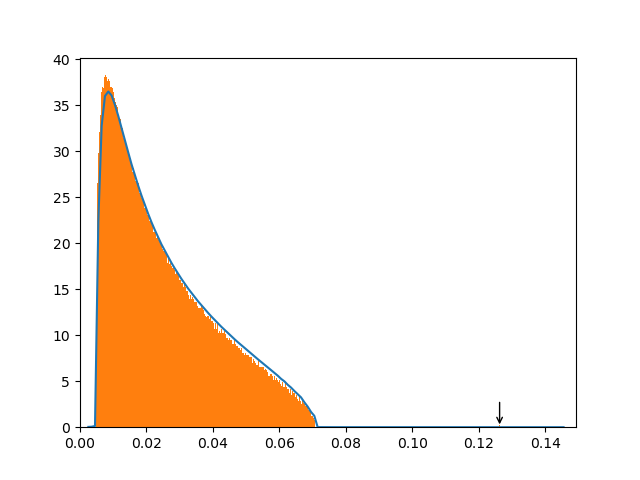}
    \caption{Analogue of Figure~\ref{fig:G-spectrum-evolution-test} but with  train data rather than test data.}\label{fig:G-spectrum-evolution-train}
\end{figure}

\bibliographystyle{plain}
\bibliography{references}
\end{document}